\documentclass[11pt]{amsart}
\usepackage[margin=2.5cm]{geometry}
\usepackage{tikz-cd}
\usepackage{graphicx}
\usepackage{mathdots}		
\usepackage{amssymb}
\usepackage{amsmath}
\usepackage{relsize}
\usepackage{subfig, graphicx}
\usepackage{mathrsfs}
\usepackage{amssymb, amsthm, indentfirst}
\usepackage{relsize}
\usepackage{hyperref}
\usepackage{stmaryrd}
\usepackage{lineno}
\usepackage{mathabx}

\title[Betti--Whittaker periods under duality]{Betti--Whittaker periods under duality: \\ variations and applications}
\author[Castellano, \ Chen, \ Darshan, \ Raghuram]{Giancarlo Castellano, \ \ Shih-Yu Chen, \ \ Nasit Darshan, \ \ A.\,Raghuram}

\address{Institut f\"ur Informatik, University of Applied Sciences Wiener Neustadt, Wiener Neustadt, Austria.} 

\email{giancarlo.castellano@fhwn.ac.at}

\address{Department of Mathematics, National Tsing Hua University, 101, Section 2, Kuang-Fu Road, Hsinchu, Taiwan, R.O.C.} 

\email{sychen.math@gmail.com}

\address{Department of Mathematics, National Tsing Hua University, 101, Section 2, Kuang-Fu Road, Hsinchu, Taiwan, R.O.C.} 

\email{darshan.nasit1111@gmail.com}

\address{Dept.\,of Mathematics, Fordham University at Lincoln Center, New York, NY 10023, USA.} 

\email{araghuram@fordham.edu}

\usepackage{mathtools} 							
\usepackage{amssymb}						    
\usepackage{amsmath}
\usepackage{mathdots}
\usepackage{mathrsfs}							
\usepackage{stmaryrd} 							

\newtheorem{theorem}{Theorem}[section]
\newtheorem{proposition}[theorem]{Proposition}
\newtheorem{lemma}[theorem]{Lemma}
\newtheorem{corollary}[theorem]{Corollary}
\theoremstyle{definition}
\newtheorem{remark}[theorem]{Remark}

\let\oleft\left
\let\oright\right
\renewcommand{\left}{\mathopen{}\mathclose\bgroup\oleft}
\renewcommand{\right}{\aftergroup\egroup\oright}

\makeatletter
\newcommand{\extp}{\@ifnextchar^\@extp{\@extp^{\,}}}
\def\@extp^#1{\mathop{\bigwedge\nolimits^{\!#1}}}
\makeatother

\usepackage{tikz}
\usetikzlibrary{cd}

\newcommand{\Z}{\mathbb{Z}}
\newcommand{\Q}{\mathbb{Q}}
\newcommand{\R}{\mathbb{R}}
\newcommand{\C}{\mathbb{C}}

\newcommand{\eps}{\varepsilon}

\newcommand{\defeq}{:=}

\newcommand{\dual}{^{\vee}}

\newcommand{\lquot}{\backslash} 				

\DeclareMathOperator{\Hom}{Hom}

\DeclareMathOperator{\Aut}{Aut}

\newcommand{\longemb}{\lhook\joinrel\longrightarrow}

\DeclareMathOperator{\GL}{GL}
\DeclareMathOperator{\Orth}{O}
\DeclareMathOperator{\U}{U}
\DeclareMathOperator{\SO}{SO}
\DeclareMathOperator{\Res}{Res}

\DeclareMathOperator{\Lie}{Lie}
\newcommand{\frg}{\mathfrak{g}}
\newcommand{\frk}{\mathfrak{k}}

\newcommand{\A}{\mathbb{A}} 					
\newcommand{\Gsum}{\mathcal{G}} 				

\newcommand{\cusp}{\mathrm{cusp}}
\newcommand{\Cusp}{\mathcal{A}_\cusp}
\newcommand{\W}{\mathcal{W}} 					
\newcommand{\cS}{\mathcal{S}}                    

\newcommand{\bp}{\begin{pmatrix}}  				
\newcommand{\ep}{\end{pmatrix}} 				
\DeclareMathOperator{\diag}{diag} 				
\newcommand{\Sm}{C^\infty} 						
\newcommand{\pr}{\mathrm{pr}} 					

\newcommand{\M}{\mathcal{M}} 					
\newcommand{\Coh}{\mathrm{Coh}}

\newcommand{\twistW}[1][\sigma]{#1_\W}

\newcommand{\invol}{\theta}

\newcommand{\involW}{\invol_{\W}}

\begin{document}

\begin{abstract}
One of the authors (Chen) had previously proved a result on the behavior of Betti--Whittaker periods under duality for 
cohomological cuspidal automorphic representations of ${\rm GL}_n/\Q$ under some regularity assumptions while using their relation to $L$-values as an 
anchor in his proof. In this article we prove a generalization of this result to ${\rm GL}_n$ over any number 
field $F$ without any regularity assumptions and without recourse to $L$-values, while using the outer-automorphism of ${\rm GL}_n$ as the main tool. 
Then, 
using results of Harder and one of the other authors (Raghuram), we give applications to new rationality results for the ratios of special 
values of general triple product $L$-functions and for general twisted Asai $L$-functions. We also give a new proof of a previous 
result of Bhagwat and Raghuram on the special values of $L$-functions for orthogonal groups. 
We present variations on period relations for the Betti--Shalika periods under duality, and the behavior of Betti--Whittaker periods 
under Galois automorphisms of $F$. 
\end{abstract}

\maketitle

\section*{Introduction}

Attached to a cohomological cuspidal automorphic representation $\Pi$ of ${\rm GL}_n$ over a number field $F$ and a permissible signature $\epsilon$ 
are the Betti--Whittaker periods $p^\epsilon(\Pi)$ that arise from a 
comparison of two different rational structures on different models for $\Pi$. These periods play a fundamental role in understanding 
the special values of various $L$-functions attached to $\Pi.$ One of the authors, Chen, had previously proved a result (see \cite{Chen2023}) on the 
behavior of such periods under duality when the base field $F$ is $\Q$, and under certain regularity assumptions on $\Pi$. 
The main theorem of {\it loc.\,cit.}\ proves that the ratio 
$$
\frac{p^\eps(\Pi\dual)}{\Gsum(\omega_{\Pi})^{1-n}\cdot p^\eps(\Pi)}
$$
is algebraic and Galois equivariant, where $\Pi^\vee$ is the contragredient of $\Pi$ and $\Gsum(\omega_\Pi)$ is the Gauss sum of the central character 
$\omega_\Pi$ of $\Pi$. The proof of this theorem was based on induction on $n$ and while using known results on the 
special values of $L$-functions for $\GL_n \times \GL_{n-1}$.

\smallskip

In this article we prove that such a period relation holds for a cohomological cuspidal automorphic representation $\Pi$ of ${\rm GL}_n$ over any number field $F$, 
without any regularity assumptions, and without recourse to $L$-values, while using the non-trivial outer-automorphism of ${\rm GL}_n$ as the main tool. 
The principal theorem of this article is stated and proved as Theorem~\ref{thm:period-relation}. 
It is a fundamental fact that the outer-automorphism of the algebraic group ${\rm GL}_n/F$ given by `transpose-inverse' 
realizes the contragredient representation; this not only works for algebraic finite-dimensional representations where it is an exercise with highest weights, but 
also works for the local components of global automorphic representations---a celebrated result of Gelfand and Kazhdan \cite{GK1971}---and so also works globally by 
multiplicity one for cuspidal representations for ${\rm GL}_n/F.$ 
The point of view taken in our proof is that, after all, the definition of the periods does not involve any $L$-function, 
and so one should expect that the proof of such a period relation should also not involve $L$-functions. Another guiding principle 
from the motivic point of view, afforded by Deligne's conjecture \cite{Deligne1979}, is that such a period relation should arise due to some 
intrinsic algebraic reason, and the period relation should then explain relations between $L$-values.

\smallskip

After reviewing some preliminaries (Section \ref{sec:prelims}) and basic cohomological ingredients (Section \ref{sec:coh-loc-sym-spaces}), 
we review the definition of the Betti--Whittaker periods in Section~\ref{sec:aut-C-BW-periods}. In the definition of these periods, one compares 
a rational structure on the Whittaker model $\W(\Pi_f)$ of the finite-part $\Pi_f$ of $\Pi$ with a cohomological model for $\Pi_f$, but for the latter there is a 
choice of cohomology degree. For applications to special values of Rankin--Selberg $L$-functions for $\GL_n \times \GL_{n-1}/F$ 
one often considers the bottom-most degree 
$b_n^F$ in which we see cuspidal cohomology for $\GL_n/F$; however, we may and do also consider the top-most degree $t_n^F$. 
The top-degree periods, for example, played an important role in the special values of the Adjoint $L$-function; see \cite{BR2017}. 
In either of these extreme degrees one has the same period relation as above. In Section~\ref{sec:dual-rep}, we introduce 
the precise form of the outer-automorphism $g \mapsto \theta(g):= \omega_n\cdot {}^tg^{-1} \cdot \omega_n^{-1},$ 
where $\omega_n$ denotes the signed anti-diagonal representative of the long Weyl element; this is better suited when we consider Whittaker models. 
In Section~\ref{sec:period-rel} we state and prove the main result on period relations. 
The proof  follows from considering the variation of rational structures across a diagram of maps \eqref{eq:cube-diagram} induced by the
Betti--Whittaker comparison, Galois action, and the automorphism $\theta$. 

\smallskip

There is an archimedean sub-problem in the proof of Theorem~\ref{thm:period-relation} that is deferred to Section~\ref{sec:arch}. 
The comparison map between the Whittaker model 
$\W(\Pi_f)$ and a cohomological model, in either of the two extreme degrees $b_n^F$ or $t_n^F$, 
depends on the choice of a generator $[\Pi_\infty]^\eps$ 
of a one-dimensional relative Lie algebra cohomology group in that degree attached to the archimedean component $\Pi_\infty$ of $\Pi.$ 
One needs to know that such generators are well-behaved with respect to $\theta.$ 
This archimedean period relation is stated and proved in Theorem~\ref{thm:archi-period-relation}. The proof, via K\"unneth theorem, 
boils down to proving it separately for representations of $\GL_n(\R)$ (see Theorem~\ref{thm:archi-period-relation-1}) and for representations of 
$\GL_n(\C)$ (see Theorem~\ref{thm:archi-period-relation-3}). 
 
\smallskip 

Once the main period relation (Theorem~\ref{thm:period-relation}) is in place, in Section~\ref{sec:critical-l-values} we consider its consequences for the 
special values of various automorphic $L$-functions. 
Recall that an essentially self-dual cuspidal representation $\Pi$ of 
$\GL_n/F$ is either $\chi$-symplectic or $\chi$-orthogonal for some Hecke character $\chi$. 
For the applications one needs to generalize an arithmeticity result of 
Gan and Raghuram [GR13, Theorem 5.3] for $\chi$-symplectic or $\chi$-orthogonal representations, which is done in 
Corollary~\ref{coro:arithmeticity} when $F$ admits a real place and $n$ is even (see also Remark \ref{rem:Clozel-Kret-Taibi}). Under such conditions, for $\chi$-orthogonal representations, in 
Theorem~\ref{thm:opposite-period-relation}, we prove that the ratio $p^\varepsilon(\Pi)/p^{-\varepsilon}(\Pi)$ is essentially the Gauss sum
$\Gsum(\chi^{\frac{n}{2}}\omega_\Pi^{-1})$. This gives a new proof of the main theorem of Bhagwat and Raghuram \cite{BR2020} on the 
ratios of special values for $L$-functions of split orthogonal group ${\rm O}(n/2, n/2)$; see Remark~\ref{rem:bhagwat-raghuram-1}. 
Now suppose $\Pi_i$ is a cuspidal automorphic representation of $\GL_{n_i}(\A_F),$ for $i=1,2,3$. The triple product 
$L$-function $L(s,\Pi_1\times\Pi_2\times\Pi_3)$ conjecturally admits all the usual properties for automorphic $L$-functions. 
For $F$ totally real, $n_1n_2$ even,  
$\Pi_1$ and $\Pi_2$ essentially self-dual, and assuming the existence of Langlands transfer $\Pi_1 \boxtimes \Pi_2$ on $\GL_{n_1n_2}$, 
using \cite{HR2020}, we prove in Theorem~\ref{thm:ratio-triple} a rationality result for ratios of critical values for triple product $L$-function. 
Next, suppose $E/F$ is an extension of totally real fields of even degree, $\Sigma$ a cuspidal representation of $\GL_n/E$ and 
$\Pi'$ a cuspidal representation of $\GL_{n'}/F$, the twisted Asai $L$-function $L(s,{\rm As}(\Sigma) \times \Pi')$ 
conjecturally admits all the usual properties for automorphic $L$-functions. When $\Sigma$ is essentially self-dual, 
we prove in Theorem~\ref{thm:ratio-Asai}, under some additional hypotheses, a rationality result for ratios of critical values for the
twisted Asai $L$-function. 
It is worth pointing out that the results in Theorem~\ref{thm:ratio-triple} and Theorem~\ref{thm:ratio-Asai} are outside the purview of 
what one hopes to prove via techniques of Eisenstein cohomology as in general they are not amongst the Langlands--Shahidi $L$-functions. 

\smallskip

We also present variations on period relations. Attached to a cohomological cuspidal automorphic representation of $\GL_n/F$,  
when $F$ has a real place and $n$ is even and $\Pi$ being $\chi$-symplectic, 
one has the Betti--Shalika periods $p^\epsilon_\cS(\Pi)$ generalizing the scope of the periods defined 
in Grobner--Raghuram \cite{GR2014a}.
In Theorem~\ref{thm:period-relation-2} we prove an analogue of Theorem~\ref{thm:period-relation} for these Betti--Shalika periods. Finally, in 
Theorem~\ref{thm:period-relation-3} we study the behavior of Betti--Whittaker periods 
under a Galois automorphism of $F$ construed as an algebraic automorphism of the algebraic group ${\rm Res}_{F/\Q}(\GL(n))$. 

\smallskip 

As this paper was nearing completion, we became aware that Yubo Jin, Dongwen Liu, and Binyong Sun have recently and independently proved related period relations for Betti–Whittaker periods under duality for a certain class of isobaric automorphic representations; see \cite{JLS2026}.

\section{Preliminaries}
\label{sec:prelims}

\noindent
Throughout, $F$ will denote a fixed algebraic number field.

\subsection{Subgroups of \texorpdfstring{$\GL_n$}{GLn}}%
\label{sec:GLn-and-subgroups}%

Let $n$ be a fixed positive integer.
Throughout, we shall denote
\begin{align*}
    G_0 := \GL_n/F, \quad
    G   := \Res_{F/\Q} G_0.
\end{align*}
More generally, the subscript ``\textsubscript{0}'' will be used for algebraic groups over $F$, 
and the same, unadorned letter will stand for the corresponding algebraic group over $\Q$, 
obtained by restriction of scalars.
We write $B_0$ for the standard Borel subgroup of $G_0$ consisting of upper-triangular matrices; $N_0$ for the unipotent radical of $B_0$; $T_0$ for the subgroup of diagonal matrices; 
and $Z_0$ for the subgroup of scalar matrices.
By restricting scalars, we obtain $\Q$-groups $B$, $N$, $T$, and $Z$, respectively.
The maximal $\Q$-split torus in $Z$ will be denoted by $S$.
Note that
\[
    G \supset B = TN \supset T \supset Z \supset S.
\]

\subsection{The real points of \texorpdfstring{$G$}{\textit{G}}}

Fix an $\R$-algebra isomorphism
\[
F_\infty:=F \otimes_\Q \R \cong \prod_{v \in S_r} \R \times \prod_{v \in S_c} \C,
\]
where $S_r$ and $S_c$ denote the sets of real and complex places of $F$, respectively.
For each archimedean place $v$, this identification depends on a choice of embedding of $F$ which we denote by $\tau_v$.
With respect to this identification we have
\begin{align}\label{eq:real-identification}
G(\R) = G_0(F \otimes_\Q \R)
   \cong
   \prod_{v \in S_r} \GL_n(\R)
   \times
   \prod_{v \in S_c} \GL_n(\C).
\end{align}
Let $\pi_0(G(\R))$ denote the component group of the Lie group $G(\R)$.
Then
\[
\widehat{\pi_0(G(\R))}
   \cong
   \prod_{v \in S_r} \widehat{\Z/2\Z},
\]
and any character $\varepsilon \in \widehat{\pi_0(G(\R))}$ decomposes as $\varepsilon = (\varepsilon_v)_{v \in S_r}$.
We write 
\[
\pm 1 \in \prod_{v \in S_r}\widehat{\Z/2\Z}
\] 
for the constant sign character $(\pm 1)_{v\in S_r}$.

\subsection{Maximal compact subgroup at infinity}%
\label{sec:max-cp-infty}%

For each archimedean place $v$, let $C_v$ be the maximal compact subgroup of $\GL_n(F_v)$ defined by
\[
C_v:= \begin{cases}
{\rm O}(n) & \mbox{ if $v \in S_r$},\\
{\rm U}(n) & \mbox{ if $v \in S_c$}.
\end{cases}
\]
Via the isomorphism (\ref{eq:real-identification}), we have the standard maximal compact subgroup of $G(\R)$:
\[
C_\infty \cong
\prod_{v \in S_r} \Orth(n)
\times
\prod_{v \in S_c} \U(n).
\]

Let $K_\infty := C_\infty S(\R)$. 
Note that
\[
K_\infty^\circ \cong
\left(
\prod_{v \in S_r} \SO(n)
\times
\prod_{v \in S_c} \U(n)
\right) \R_+^\times
\]
under the isomorphism (\ref{eq:real-identification}),
where, as usual, the superscript ``${}^\circ$'' denotes the connected component of the identity of the topological group in question.
Moreover, by the inclusion \(K_\infty \hookrightarrow G(\R)\) we identify $\pi_0(G(\R)) = K_\infty / K_\infty^\circ$.
Let $K_v:= C_v\R_+^\times$.

\subsection{Adèles and additive characters}\label{sec:add-char-AF/F}

Let $\A = \A_\Q$ denote the ring of adèles of $\Q$, and $\A_f$ its subring of finite adèles.
Since $G = \Res_{F/\Q}\GL_n$, we have natural identifications
$G(\A) = \GL_n(\A_F),$ and $G(\A_f) = \GL_n(\A_{F,f}),$
where $\A_F$ denotes the adèle ring of $F$.

Let $\psi_F : \A_F/F \to \C^\times$ be the standard additive character of $\A_F$ defined by $\psi_F := \psi_\Q \circ \mathrm{tr}_{F/\Q}$,
where $\psi_\Q = \prod_v \psi_{\Q_v}$ is the additive character of $\A$ with
\[
\psi_{\Q_p}(x) := e^{-2\pi i x} \quad (x \in \Z[p^{-1}]), 
\quad 
\psi_{\R}(x) := e^{2\pi i x} \quad (x \in \R).
\]

\subsection{Gauss sums and algebraic Hecke characters}%
\label{sec:Hecke-char-Gauss-sums}%

Let $\chi : \A_{F,f}^\times\rightarrow \C^\times$ be a continuous character. 
For each finite place $v$ of $F$, we denote by $\chi_v$ the component of $\chi$ at $v$.
The \emph{Gauss sum} of $\chi$ is defined by
\[
\mathcal{G}(\chi):= |\delta_{F/\Q}|^{-1/2}\prod_{v\nmid \infty}\varepsilon(0,\chi_v,\psi_{F_v}),
\]
where $\delta_{F/\Q}$ is the discriminant of $F/\Q$ and $\varepsilon(s,\chi_v,\psi_{F_v})$ denotes the local $\varepsilon$-factor of $\chi_v$ with respect to $\psi_{F_v}$ (cf.\,\cite{Tate1979}). 
For $\sigma \in \mathrm{Aut}(\C)$, define character ${}^\sigma\chi$ of $\A_{F,f}^\times$ by ${}^\sigma\chi(x) := \sigma(\chi(x))$ for all $x \in \A_{F,f}^\times$.
Then we have
\begin{align}\label{eq:Gauss-sum}
\sigma(\mathcal{G}(\chi)) = {}^\sigma\chi(u_\sigma)\,\mathcal{G}({}^\sigma\chi),
\end{align}
where $u_\sigma \in \prod_p \Z_p^\times \subset \A_{F,f}^\times$ is the unique element such that
$\sigma(\psi_\Q(x)) = \psi_\Q(u_\sigma x)$ for all $x \in \A_f$.
In particular, for continuous complex characters $\chi_1,\chi_2$ of $\A_{F,f}^\times$, we have
\begin{align}\label{eq:Gauss-sum-2}
\sigma \left(\frac{\mathcal{G}(\chi_1\chi_2)}{\mathcal{G}(\chi_1)\mathcal{G}(\chi_2)}\right) 
= \frac{\mathcal{G}({}^\sigma\chi_1 {}^\sigma\chi_2)}{\mathcal{G}({}^\sigma\chi_1)\,\mathcal{G}({}^\sigma\chi_2)}.
\end{align}

Let $\omega: \A_F^\times \to \C^\times$ be a \emph{Hecke character} of $F$, that is, a continuous complex character of $F^\times \backslash \A_F^\times$.
Let $\omega_\infty$ and $\omega_f$ denote the archimedean and finite parts of $\omega$, respectively.
The Gauss sum of $\omega$ is defined to be the Gauss sum of $\omega_f$, that is, $\mathcal{G}(\omega):=\mathcal{G}(\omega_f)$.
We say that $\omega$ is \emph{algebraic} if there exists
\[
a(\omega)=(a(\omega)_\tau)_{\tau:F\rightarrow\C}\in \prod_{\tau:F\rightarrow\C}\Z,
\]
called the \emph{infinity type} of $\omega$, such that
\[
\omega_\infty(x_\infty)
=
\prod_{v\in S_r} x_v^{a(\omega)_{\tau_v}}
\prod_{v\in S_c} x_v^{a(\omega)_{\tau_v}}\overline{x}_v^{a(\omega)_{\overline{\tau}_v}}
\]
for all totally positive $x_\infty\in F_\infty^\times$.
Suppose that $\omega$ is algebraic. Then there exists an integer ${\sf w}(\omega)$,
called the \emph{pure weight} of $\omega$, such that
\begin{itemize}
\item if $S_r\neq\varnothing$, then $a(\omega)_\tau={\sf w}(\omega)$ for all $\tau$;
\item if $S_r=\varnothing$, then $a(\omega)_{\sigma\circ\tau}+a(\omega)_{\sigma\circ\overline{\tau}}
= {\sf w}(\omega)$
for all $\tau$ and all $\sigma\in{\rm Aut}(\C)$.
\end{itemize}
The \emph{signature} $\varepsilon(\omega)\in \prod_{v\in S_r}\widehat{\Z/2\Z}$ of $\omega$ is defined by
\[
\varepsilon(\omega)_v
:=
(-1)^{{\sf w}(\omega)}\,\omega_v(-1),
\quad v\in S_r .
\]
For $\sigma \in {\rm Aut}(\C)$, let ${}^\sigma \omega$ be the algebraic Hecke character of $F$ such that ${}^\sigma\omega_f = \sigma(\omega_f)$. Note that ${}^\sigma\omega_\infty$ is determined by
\[
a({}^\sigma\omega) = (a(\omega)_{\sigma^{-1}\circ \tau})_{\tau:F\rightarrow \C},\quad \varepsilon({}^\sigma\omega) = \varepsilon(\omega).
\]

\subsection{Cuspidal automorphic representations}%
\label{sec:cusp-aut-rep}%

We denote by $\mathcal{A}(G(\Q)\backslash G(\A))$ the space of automorphic forms on $G(\A)$ (cf.\,\cite{BJ1979} or \cite{Bump1998}%
\footnote{%
	As in \cite{Bump1998}, we take our cusp forms to be $C_\infty$-finite, where $C_\infty$ is the standard maximal compact subgroup of $G(\R)$; see \S\,\ref{sec:max-cp-infty}.%
}%
).
This space is equipped with a natural $(\frak{g}_\infty,C_\infty)\times G(\A_f)$-module structure, where $\frg_\infty$ is the complexified Lie algebra of $G(\R)$.
Let $\Cusp(G(\Q)\backslash G(\A))$ be the subspace of cusp forms.
It is semisimple and an irreducible constituent $\Pi$ is called a \emph{cuspidal automorphic representation} of $G(\A)$. 
It occurs with multiplicity one, and we denote by $V_\Pi$ its underlying representation space. 
We also denote by $\omega_\Pi$ the central character of $\Pi$.

\subsection{Whittaker models}%
\label{sec:Wh-mod}%

Let $U$ be a maximal unipotent subgroup of $G$, and 
\[
\xi:U(\Q)\backslash U(\A) \rightarrow \C^\times
\] 
be a non-degenerate additive character of $U(\A)$. 
For an automorphic form $\phi$ on $G(\A)$, the Whittaker function of $\phi$ with respect to $\xi$ is defined by
\begin{equation*}
	W_{\xi}(g;\phi) = \int_{U(\Q)\lquot U(\A)} \phi(u g) \overline{\xi(u)} \,du^{\rm Tam}, \quad g \in G(\A).
\end{equation*}
Here $du^{\rm Tam}$ is the Tamagawa measure on $U(\A)$.
Let $\Pi$ be a cuspidal automorphic representation of $G(\A)$.
The space 
\[
\W(\Pi, \xi) \defeq \{ W_{\xi}(\phi) \,\vert\, \phi \in V_\Pi\}
\] 
is called the \emph{Whittaker model of $\Pi$ with respect to $\xi$}; it is a model of $\Pi$ in the sense that we have an isomorphism of $(\frak{g}_\infty,C_\infty)\times G(\A_f)$-modules:
\begin{align}\label{eq:Whittaker-realization}
V_\Pi \longrightarrow \W(\Pi, \xi), \quad \phi \longmapsto W_{\xi}(\phi).
\end{align}
For each place $v$ of $\Q$, let $\mathcal{W}(\Pi_v,\xi_{v})$ be the Whittaker model of $\Pi_v$ with respect to $\xi_{v} := \xi\vert_{U(\Q_v)}$. 
It is realized in the space of smooth functions $W:G(\Q_v) \rightarrow \C$ such that
\[
W(ug) = \xi_{v}(u)W(g),\quad u \in U(\Q_v),\ g \in G(\Q_v)
\]
and $W$ is of moderate growth when $v=\infty$. 
Then we have an isomorphism
\[
{\bigotimes_v}' \mathcal{W}(\Pi_v,\xi_{v}) \longrightarrow \W(\Pi, \xi),\quad \bigotimes_v W_v \longmapsto \prod_v W_v.
\]
A similar factorization holds for the Whittaker model $\mathcal{W}(\Pi_f,\xi_f)$ of $\Pi_f = \otimes_{p}'\Pi_p$ with respect to $\xi_f:= \xi\vert_{U(\A_f)}$. 

Recall the unipotent radical $N$ of the standard Borel subgroup of $G$ in \S\,\ref{sec:GLn-and-subgroups}.
Let $\psi_{N} : N(\Q)\backslash N(\A) \rightarrow \C^\times$ be the standard non-degenerate additive character defined by
\begin{align}\label{eq:additive}
\psi_{N}(u) := \psi_F(u_{12}+u_{23}+\cdots+u_{n-1,n}),\quad u=(u_{ij}) \in N(\A) = N_0(\A_F),
\end{align}
where $\psi_F$ is as in \S\,\ref{sec:add-char-AF/F}. In this case, we write 
\[
\mathcal{W}(\Pi_v) = \mathcal{W}(\Pi_v,\psi_{N,v}),\quad
\mathcal{W}(\Pi_f) = \mathcal{W}(\Pi_f,\psi_{N,f}),\quad\mathcal{W}(\Pi) = \mathcal{W}(\Pi,\psi_N).
\]

\section{Cohomology of locally symmetric spaces}
\label{sec:coh-loc-sym-spaces}

\subsection{Adèlic locally symmetric spaces}%
\label{sec:loc-sym-space}%


Consider the symmetric space
\[
X_\infty \coloneqq G(\R)/K_\infty^\circ .
\]
It carries a left action of $G(\Q) \subset G(\R)$ and a right action of $K_\infty/K_\infty^\circ = \pi_0(G(\R))$.
The product space
\[
X \coloneqq X_\infty \times G(\A_f)
\]
then admits commuting left and right actions of $G(\Q)$ and
$G(\A_f)\times\pi_0(G(\R))$, respectively. The action of $G(\Q)$ is
continuous, free, and properly discontinuous, so the quotient space
\[
S^G \coloneqq G(\Q)\backslash X
\]
is Hausdorff.
For an open compact subgroup $K_f\subset G(\A_f)$ we define the
\emph{adèlic locally symmetric space with level structure $K_f$} by
\[
S^G_{K_f}
\coloneqq S^G/K_f
= G(\Q)\backslash G(\A)/K_\infty^\circ K_f .
\]
In general $S^G_{K_f}$ is an orbifold. If $K_f$ is \emph{neat}
(cf.\,\cite[\S\,15.2]{GH2024}), then $S^G_{K_f}$ is a smooth
manifold. Every $K_f$ contains a neat subgroup $K_f'$ of finite index,
which may moreover be chosen normal in $K_f$.
An inclusion $K_f' \subset K_f$ induces a natural map
$
S^G_{K_f} \longrightarrow S^G_{K_f'}.$ 
The spaces $S^G_{K_f}$ therefore form an inverse system whose inverse
limit is naturally isomorphic to $S^G$
(cf.\ \cite[Proposition~1.9]{Rohlfs1996}).

\subsection{Weights of the algebraic torus}

Let $X^*(T\times\C)$ be the group of the algebraic characters (integral weights) of the torus $T\times\C$. We have
\begin{align*}
X^*(T\times\C) = \prod_{\tau:F \rightarrow \C}X^*(T_0\times_{F,\tau}\C) &\cong \prod_{\tau:F \rightarrow \C} \Z^n,\\
\lambda = (\lambda^\tau)_{\tau:F \rightarrow \C} &\longmapsto (\lambda_1^\tau,...,\lambda_n^\tau)_{\tau:F \rightarrow \C}.
\end{align*}
The ${\rm Aut}(\C)$-action on $X^*(T\times\C)$ is given by
\[
{}^\sigma\lambda = (\lambda^{\sigma^{-1}\circ\tau})_{\tau:F \rightarrow \C},\quad \sigma \in {\rm Aut}(\C).
\]
The \emph{rationality field} $\Q(\lambda)$ of $\lambda$ is the fixed field in $\C$ of the stabilizer of $\lambda$ in ${\rm Aut}(\C)$.
Define subgroups 
\[
X^*(T\times\C)\supset X^+(T\times \C) \supset X_{\rm alg}^+(T\times \C) \supset X_0^+(T\times \C) \supset X_{00}^+(T\times\C)
\]
as follows.
\begin{itemize}
\item $X^+(T\times\C)$ (dominant integral weights): an integral weight $\lambda$ is \emph{dominant} if 
\[
\lambda_1^\tau \geq \cdots \geq \lambda_n^\tau,\quad \tau \in {\rm Hom}(F,\C).
\]
\item $X_{\rm alg}^+(T\times\C)$ (algebraic dominant integral weights): a dominant integral weight $\lambda$ is \emph{algebraic} if 
\[
\sum_{i = 1}^n \lambda_i^\tau = \sum_{i = 1}^n \lambda_i^{\tau'}, \quad \tau, \tau' \in \Hom(F, \C).
\]
\item $X_0^+(T\times\C)$ (pure weights): a dominant integral weight $\lambda$ is \emph{pure} if it is algebraic and there exists an integer ${\sf w}$, called the \emph{purity weight} of $\lambda$, such that
\[
\lambda_i^\tau + \lambda_{n+1-i}^{\bar\tau} = {\sf w},\quad 1 \leq i \leq n,\ \tau \in \Hom(F,\C).
\]
\item $X_{00}^+(T\times\C)$ (strongly-pure weights): a pure weight $\lambda$ with purity weight ${\sf w}$ is \emph{strongly-pure} if 
\[
\lambda_i^{\sigma^{-1}\circ\tau} + \lambda_{n+1-i}^{\sigma^{-1}\circ\bar{\tau}} = {\sf w},\quad 1 \leq i \leq n,\ \tau \in \Hom(F,\C),\ \sigma \in {\rm Aut}(\C).
\]
In other words, $\lambda$ is strongly-pure if both $\lambda$ and ${}^\sigma\lambda$ are pure with the same purity weight for every $\sigma\in\Aut(\C)$.
\end{itemize}

\subsection{Finite-dimensional representations}%
\label{sec:irr-reps}%

The irreducible algebraic representations of $G\times\C$ are parametrized by the dominant integral weights of $T\times\C$.
More precisely, for $\lambda = (\lambda^\tau)_{\tau:F \rightarrow \C} \in X^+(T\times\C)$, via the isomorphism 
\[
G\times\C = \prod_{\tau:F\rightarrow \C}G_0\times_{F,\tau}\C \cong \prod_{\tau:F\rightarrow \C} \GL_n/\C,
\]
let $(\rho_\lambda,\mathcal{M}_\lambda)$ be the irreducible algebraic representation of $G\times\C$ defined by
\[
(\rho_\lambda,\mathcal{M}_\lambda)
:= \bigotimes_{\tau:F\rightarrow \C}(\rho_{\lambda^\tau},\mathcal{M}_{\lambda^\tau}),
\]
where $(\rho_{\lambda^\tau},\mathcal{M}_{\lambda^\tau})$ denotes the irreducible algebraic representation of $\GL_n(\C)$ with highest weight $\lambda^\tau$ with respect to the upper triangular Borel subgroup.
Then $(\rho_\lambda,\mathcal{M}_\lambda)$ has highest weight $\lambda$ with respect to $B \times \C$.
The group $G(\Q)=\GL_n(F)$ acts diagonally on $\mathcal{M}_\lambda$, that is,
\[
\rho_\lambda(\gamma)\cdot\bigotimes_\tau m_\tau
=
\bigotimes_\tau \rho_{\lambda^\tau}(\tau(\gamma))\,m_\tau,
\quad \gamma \in G(\Q).
\]
Note that the resulting representation of $G(\Q)$ is defined over $\Q(\lambda)$ (cf.\,\cite[Lemma 7.1]{GR2014b}).

Let $(\rho,\mathcal{M})$ be an algebraic representation of $G\times\C$. 
For $\sigma \in {\rm Aut}(\C)$, recall the \emph{$\sigma$-conjugate representation} $({}^\sigma\rho,{}^\sigma\mathcal{M})$ of $(\rho,\mathcal{M})$ 
is the algebraic representation of $G\times\C$ defined by
\[
{}^\sigma\mathcal{M}:=\mathcal{M} \otimes_{\C,\sigma}\C
\]
and
\[
{}^\sigma\rho(g)(m\otimes z):= \rho(\sigma^{-1}g)m \otimes z,\quad g \in G(\C)
\]
where
\[
\sigma^{-1}(g_\tau)_{\tau:F\rightarrow \C} = (\sigma^{-1}(g_{\sigma\circ\tau}))_{\tau:F\rightarrow \C}.
\]
Note that the natural map
\begin{align}\label{eq:Galois-identity}
t_\sigma : \mathcal{M} \longrightarrow {}^\sigma\mathcal{M},\quad m \longmapsto m\otimes 1
\end{align}
is $\sigma$-linear and $G(\Q)$-equivariant.
The \emph{dual representation} of $(\rho,\M)$ is denoted by $(\rho^\vee,\M^\vee)$.

For $\lambda \in X^+(T\times\C)$, the representations $({}^\sigma\rho_\lambda,{}^\sigma\mathcal{M}_\lambda)$ 
and $(\rho_{{}^\sigma\lambda},\mathcal{M}_{{}^\sigma\lambda})$ are isomorphic, since they both have highest weight ${}^\sigma\lambda$.
More explicitly, for each $\tau \in \Hom(F,\C)$ we fix a $\Q$-rational structure $\mathcal{M}_{\lambda^\tau,\Q}$ of $\mathcal{M}_{\lambda^\tau}$, 
and let $t_{\tau, \sigma} : \mathcal{M}_{\lambda^\tau} \rightarrow \mathcal{M}_{\lambda^\tau}$ be the $\sigma$-linear 
isomorphism acting as the identity on $\mathcal{M}_{\lambda^\tau,\Q}$. Then we have an isomorphism of algebraic representations
\begin{align}\label{eq:finite-rep}
({}^\sigma\rho_\lambda,{}^\sigma\mathcal{M}_\lambda) \longrightarrow (\rho_{{}^\sigma\lambda},\mathcal{M}_{{}^\sigma\lambda}),\quad 
\left(\bigotimes_\tau m_\tau\right)\otimes 1\longmapsto \bigotimes_\tau t_{\sigma^{-1}\circ\tau,\sigma}\bigl(m_{\sigma^{-1}\circ\tau}\bigr).
\end{align}
Also note that $(\rho_\lambda^\vee,\M_\lambda^\vee)$ and $(\rho_{\lambda^\vee},\M_{\lambda^\vee})$ are isomorphic, 
where $\lambda^\vee \in X^+(T\times\C)$ is the dual of $\lambda$ defined by
\[
(\lambda^\vee)^\tau := (-\lambda_n^\tau,...,-\lambda_1^\tau),\quad \tau \in \Hom(F,\C).
\]

\subsection{Sheaf cohomology and Hecke action}%
\label{sec:sheaf-locsym}%

Let $(\rho,\mathcal{M})$ be an algebraic representation of $G\times{E}$ for some field extension $E/\Q$. 
For an open compact subgroup $K_f \subset G(\A_f)$, we denote by $\widetilde{\mathcal{M}}$ the sheaf of $E$-vector spaces on $S_{K_f}^G$ defined as follows. 
For each open subset $U \subset S_{K_f}^G$, we set
\[
\widetilde{\mathcal{M}}(U)
:=
\left\{
s : \pr^{-1}(U) \to \mathcal{M}
\ \middle|\
\begin{aligned}
&s \text{ is locally constant, and} \\
&s(\gamma x)=\rho(\gamma)s(x)
\text{ for all } \gamma\in G(\Q),\, x\in \pr^{-1}(U)
\end{aligned}
\right\}.
\]
where $\pr$ denotes the canonical projection $\pr : X/K_f \longrightarrow S_{K_f}^G$.
If $K_f$ is neat, then $\widetilde{\mathcal{M}}$ is a locally constant sheaf.
We consider the sheaf cohomology groups
\[
H^\bullet(S_{K_f}^G,\widetilde{\mathcal{M}}).
\]
An inclusion $K_f' \subset K_f$ of open compact subgroups of $G(\A_f)$ induces a canonical map
\[
H^\bullet(S_{K_f}^G,\widetilde{\mathcal{M}})
\longrightarrow
H^\bullet(S_{K_f'}^G,\widetilde{\mathcal{M}}).
\]
Passing to the direct limit over all open compact subgroups, we define
\[
H^\bullet(S^G,\widetilde{\mathcal{M}})
:=
\varinjlim_{K_f} H^\bullet(S_{K_f}^G,\widetilde{\mathcal{M}}).
\]
For $x = x_\infty \times x_f \in \pi_0(G(\R)) \times G(\A_f)$, right multiplication by $x$ defines a $G(\Q)$-equivariant diffeomorphism
\begin{align}\label{eq:m_x}
m_x : X/K_f \longrightarrow X/(x_f^{-1}K_f x_f).
\end{align}
This induces a canonical map on cohomology
\[
m_x^\bullet :
H^\bullet(S_{x_f^{-1}K_f x_f}^G,\widetilde{\mathcal{M}})
\longrightarrow
H^\bullet(S_{K_f}^G,\widetilde{\mathcal{M}}).
\]
Passing to the direct limit, we obtain a map
\[
m_x^\bullet :
H^\bullet(S^G,\widetilde{\mathcal{M}})
\longrightarrow
H^\bullet(S^G,\widetilde{\mathcal{M}}),
\]
which defines an action of $\pi_0(G(\R)) \times G(\A_f)$ on $H^\bullet(S^G,\widetilde{\mathcal{M}})$, called the \emph{Hecke action}.
Moreover, we have a natural identification
\[
H^\bullet(S_{K_f}^G,\widetilde{\mathcal{M}})\cong H^\bullet(S^G,\widetilde{\mathcal{M}})^{K_f},
\]
where the superscript ``${}^{K_f}$'' on the right-hand side is the usual notation for the subspace of $K_f$-invariant elements.

\subsection{Betti and de Rham cohomology}%
\label{sec:Hecke-action}%

We keep the notation of \S\,\ref{sec:sheaf-locsym} and
assume that $E=\C$. 
We recall two standard complexes computing the sheaf cohomology at the transcendental level: the smooth singular complex and the de Rham complex. 
We also describe explicitly the Hecke action in each realization.

Let $\left(C_\bullet^\infty(X/K_f),\partial\right)$ denote the complex of smooth singular chains. 
Recall that the boundary operator $\partial$ sends a smooth singular $q$-simplex $\alpha : \Delta^q \to X/K_f$ to a $(q-1)$-simplex
\[
\partial \alpha
:=
\sum_{i=0}^q (-1)^i \alpha \circ \iota_i ,
\]
where $\iota_i : \Delta^{q-1} \hookrightarrow \Delta^q$ is the inclusion of the $i$-th face of the standard simplex.
The group $G(\Q)$ acts on smooth simplices by
\[
(\gamma\cdot\alpha)(t) := \gamma\cdot\alpha(t),
\quad 
\gamma \in G(\Q),\ \alpha \in C_\bullet^\infty(X/K_f).
\]
It is clear that $\gamma\circ\partial = \partial \circ \gamma$, that is, $G(\Q)$ acts on the chain complex.
Taking the $\C$-linear dual, we obtain the smooth singular cochain complex
\[
\left(\Hom_\Z(C_\bullet^\infty(X/K_f), \mathcal{M}), \partial^\ast\right).
\]
This complex carries a natural action of $G(\Q)$ defined by
\[
(\gamma\cdot \ell)(\alpha)
:=
\rho(\gamma)\,\ell(\gamma^{-1}\cdot \alpha),
\]
for $\gamma \in G(\Q)$ and $\ell \in \Hom_\C(C_\bullet^\infty(X/K_f), \mathcal{M})$.
The \emph{Betti/singular cohomology} of $\widetilde{\mathcal M}$ on $S_{K_f}^G$ is the cohomology of the subcomplex of $G(\Q)$-invariants:
\[
H_{\rm B}^\bullet(S_{K_f}^G,\mathcal M)
:=
H^\bullet\left(
\Hom_\Z(C_\bullet^\infty(X/K_f), \mathcal M)^{G(\Q)}
\right).
\]
An inclusion $K_f' \subset K_f$ induces a canonical map
\[
C_\bullet^\infty(X/K_f) \longrightarrow C_\bullet^\infty(X/K_f'),
\]
which commutes with both the boundary operator $\partial$ and the $G(\Q)$-action. 
Taking duals, we obtain a cochain map, and hence an induced map on cohomology
\[
H_{\rm B}^\bullet(S_{K_f}^G,\widetilde{\mathcal{M}})
\longrightarrow
H_{\rm B}^\bullet(S_{K_f'}^G,\widetilde{\mathcal{M}}).
\]
Taking the direct limit over all open compact subgroups, we define
\[
H_{\rm B}^\bullet(S^G,\widetilde{\mathcal{M}})
:=
\varinjlim_{K_f} H_{\rm B}^\bullet(S_{K_f}^G,\widetilde{\mathcal{M}}).
\]
For $x = x_\infty \times x_f \in \pi_0(G(\R)) \times G(\A_f)$, the diffeomorphism $m_x$ in \eqref{eq:m_x} induces a pullback on smooth singular chains by precomposition
\[
(m_x)_\sharp : C_\bullet^\infty(X/K_f)
\longrightarrow
C_\bullet^\infty(X/x_f^{-1}K_f x_f),
\quad
\alpha \longmapsto m_x \circ \alpha.
\]
Taking the dual, we obtain a cochain map
\begin{align*}
m_{x,{\rm B}} :
\left(\Hom_\Z(C_\bullet^\infty(X/x_f^{-1}K_f x_f),\mathcal{M}),\partial^\ast\right)
&\longrightarrow
\left(\Hom_\Z(C_\bullet^\infty(X/K_f),\mathcal{M}),\partial^\ast\right), \\
\ell &\longmapsto \ell \circ (m_x)_\sharp.
\end{align*}
It is straightforward to verify that $m_{x,{\rm B}}$ commutes with the $G(\Q)$-action.
Therefore it induces a canonical map on cohomology
\[
m_{x,{\rm B}}^\bullet :
H_{\rm B}^\bullet(S_{x_f^{-1}K_f x_f}^G,\widetilde{\mathcal{M}})
\longrightarrow
H_{\rm B}^\bullet(S_{K_f}^G,\widetilde{\mathcal{M}}).
\]
Passing to the direct limit, we obtain a map
\[
m_{x,{\rm B}}^\bullet :
H_{\rm B}^\bullet(S^G,\widetilde{\mathcal{M}})
\longrightarrow
H_{\rm B}^\bullet(S^G,\widetilde{\mathcal{M}}),
\]
which defines the Hecke action of $\pi_0(G(\R)) \times G(\A_f)$ on
$H_{\rm B}^\bullet(S^G,\widetilde{\mathcal{M}})$.

Let $\left(\Omega^\bullet(X/K_f;\M),d\right)$ denote the de Rham cochain complex of
$\M$-valued smooth differential forms on $X/K_f$.
Recall that the differential $d$ sends a $\M$-valued smooth $q$-form $\omega$
on $X/K_f$ to a $(q+1)$-form given by
\[
\begin{aligned}
(d\omega)(X_0,\dots,X_q)
&:=
\sum_{i=0}^q (-1)^i
X_i\!\left(\omega(X_0,\dots,\widehat{X_i},\dots,X_q)\right) \\
&\quad +
\sum_{0 \le i < j \le q}
(-1)^{i+j}\,
\omega\!\left([X_i,X_j],\,X_0,\dots,\widehat{X_i},\dots,\widehat{X_j},\dots,X_q\right),
\end{aligned}
\]
for smooth vector fields $X_0,\dots,X_q$ on $X/K_f$.
The group $G(\Q)$ acts on $\M$-valued smooth differential forms by
\[
\gamma\cdot\omega := \rho(\gamma)\bigl((\gamma^{-1})^*\omega\bigr),
\quad
\gamma\in G(\Q),\ \omega\in\Omega^\bullet(X/K_f;\M).
\]
It is straightforward to verify that this action commutes with the differential. Hence $G(\Q)$ acts on the de Rham complex.
The \emph{de Rham cohomology} of $\widetilde{\mathcal M}$ on $S_{K_f}^G$
is defined as the cohomology of the subcomplex of $G(\Q)$-invariants:
\[
H_{\rm dR}^\bullet(S_{K_f}^G,\widetilde{\mathcal M})
:=
H^\bullet\!\left(\Omega^\bullet(X/K_f;\M)^{G(\Q)}\right).
\]
If $K_f' \subset K_f$ is an inclusion of neat open compact subgroups, the natural covering map induces a pullback
\[
\Omega^\bullet(X/K_f;\M)
\longrightarrow
\Omega^\bullet(X/K_f';\M),
\]
which commutes with both the differential $d$ and the $G(\Q)$-action.
Consequently, we obtain a map on cohomology
\[
H_{\rm dR}^\bullet(S_{K_f}^G,\widetilde{\mathcal M})
\longrightarrow
H_{\rm dR}^\bullet(S_{K_f'}^G,\widetilde{\mathcal M}).
\]
Taking the direct limit over all open compact subgroups, we define
\[
H_{\rm dR}^\bullet(S^G,\widetilde{\mathcal M})
:=
\varinjlim_{K_f}
H_{\rm dR}^\bullet(S_{K_f}^G,\widetilde{\mathcal M}).
\]
For $x=x_\infty\times x_f \in \pi_0(G(\R))\times G(\A_f)$,
the diffeomorphism $m_x$ in \eqref{eq:m_x} induces a cochain map
\begin{align*}
m_{x,{\rm dR}} :
\bigl(\Omega^\bullet(X/x_f^{-1}K_f x_f;\mathcal M),d\bigr)
&\longrightarrow
\bigl(\Omega^\bullet(X/K_f;\mathcal M),d\bigr), \\
\omega &\longmapsto m_x^*\omega .
\end{align*}
One checks easily that $m_{x,{\rm dR}}$ commutes with the $G(\Q)$-action.
Hence it induces a canonical map on cohomology
\[
m_{x,{\rm dR}}^\bullet :
H_{\rm dR}^\bullet(S_{x_f^{-1}K_f x_f}^G,\widetilde{\mathcal M})
\longrightarrow
H_{\rm dR}^\bullet(S_{K_f}^G,\widetilde{\mathcal M}).
\]
Passing to the direct limit, we obtain a map
\[
m_{x,{\rm dR}}^\bullet :
H_{\rm dR}^\bullet(S^G,\widetilde{\mathcal M})
\longrightarrow
H_{\rm dR}^\bullet(S^G,\widetilde{\mathcal M}),
\]
which defines the Hecke action of $\pi_0(G(\R)) \times G(\A_f)$ on
$H_{\rm dR}^\bullet(S^G,\widetilde{\mathcal M})$.

The Betti and de Rham complexes introduced above are related by the classical de Rham theorem. This yields a canonical \emph{comparison isomorphism}
\[
I^\bullet :
H_{\rm dR}^\bullet(S^G,\widetilde{\mathcal M})
\longrightarrow
H_{\rm B}^\bullet(S^G,\widetilde{\mathcal M}).
\]
More precisely, for each level $K_f$ there is a cochain map
\begin{align*} 
I : \left(\Omega^\bullet(X/K_f;\M),d\right) &\longrightarrow \left(\Hom_\Z(C_\bullet^\infty(X/K_f), \mathcal{M}), \partial^\ast\right),\\ \omega &\longmapsto \left(\alpha \mapsto \int_{\Delta^\bullet}\alpha^*\omega\right). 
\end{align*}
One checks that $I$ commutes with the action of $G(\Q)$. Passing to $G(\Q)$-invariants and taking the direct limit over $K_f$, we obtain the comparison isomorphism $I^\bullet$ above.
The definitions of the Hecke actions imply that
\[
I\circ m_{x,{\rm dR}}
=
m_{x,{\rm B}}\circ I,\quad x\in\pi_0(G(\R))\times G(\A_f).
\]
Consequently, the comparison isomorphism $I^\bullet$ is equivariant with respect to the Hecke action.

\subsection{Automorphic and cuspidal cohomology}%
\label{sec:coh-groups-transc-level}%

We keep the notation of \S\,\ref{sec:sheaf-locsym} and assume that $E=\C$. 
The \emph{automorphic cohomology} of $G$ with respect to $(\rho,\M)$ is the relative $(\frak{g}_\infty,K_\infty^\circ)$-cohomology group
\[
H^\bullet(\frak{g}_\infty,K_\infty^\circ;\mathcal{A}(G(\Q)\backslash G(\A))\otimes\M).
\]
Franke \cite{Franke1998} proved Borel's conjecture that the natural inclusion
\[
\mathcal{A}(G(\Q)\backslash G(\A)) \longemb C^\infty(G(\Q)\backslash G(\A))
\]
induces an isomorphism
\[
H^\bullet(\frak{g}_\infty,K_\infty^\circ;\mathcal{A}(G(\Q)\backslash G(\A))\otimes\M)
\longrightarrow
H^\bullet(\frak{g}_\infty,K_\infty^\circ;C^\infty(G(\Q)\backslash G(\A))\otimes\M).
\]
On the other hand, for each level $K_f$ there is an isomorphism between the relative Lie algebra complex and the de Rham complex (cf.\,\cite[VII, Corollary 2.7]{BW2000})
\[
\left(
\Hom_{K_\infty^\circ}\left(\extp^\bullet \frg_\infty/\frk_\infty,
\Sm(G(\Q)\lquot G(\A)/K_f)\otimes\M\right),d
\right)
\longrightarrow
\left(\Omega^\bullet(X/K_f;\M)^{G(\Q)},d\right).
\]
Passing to the direct limit over $K_f$ yields a $\pi_0(G(\R))\times G(\A_f)$-equivariant isomorphism
\[
H^\bullet(\frak{g}_\infty,K_\infty^\circ;C^\infty(G(\Q)\backslash G(\A))\otimes\M)
\longrightarrow
H_{\rm dR}^\bullet(S^G,\widetilde\M).
\]
In the sequel we identify automorphic cohomology with $H_{\rm dR}^\bullet(S^G,\widetilde\M)$ via these isomorphisms.

The \emph{cuspidal cohomology} of $G$ with respect to $(\rho,\M)$ is
\[
H_{\rm cusp}^\bullet(S^G,\widetilde\M)
:=
H^\bullet(\frak{g}_\infty,K_\infty^\circ;
\mathcal{A}_{\rm cusp}(G(\Q)\backslash G(\A))\otimes\M).
\]
By the result of Borel \cite[Corollary 5.5]{Borel1981}, the inclusion
\[
\mathcal{A}_{\rm cusp}(G(\Q)\backslash G(\A))
\longemb
C^\infty(G(\Q)\backslash G(\A))
\]
induces a $\pi_0(G(\R))\times G(\A_f)$-equivariant injective homomorphism
\[
H_{\rm cusp}^\bullet(S^G,\widetilde\M)
\longrightarrow
H_{\rm dR}^\bullet(S^G,\widetilde\M),
\]
whose image lies in the interior cohomology. Since
$\mathcal{A}_{\rm cusp}(G(\Q)\backslash G(\A))$ is semisimple as a $G(\A_f)$-module, the same holds for the cuspidal cohomology.
Let $\Coh_\cusp(G,\M)$ denote the set of isomorphism classes of irreducible admissible representations of $G(\A_f)$ occurring in $H_{\rm cusp}^\bullet(S^G,\widetilde\M)$. 
For $\varepsilon \in \widehat{\pi_0(G(\R))}$ and $\Pi_f\in\Coh_\cusp(G,\M)$, write
\[
H_{\rm cusp}^\bullet(S^G,\widetilde\M)(\varepsilon\times\Pi_f)
\]
for the corresponding isotypic component under the action of $\pi_0(G(\R)) \times G(\A_f)$.
When $(\rho,\M)=(\rho_\lambda,\M_\lambda)$ with $\lambda\in X^+(T\times\C)$, we write
\[
\Coh_\cusp(G,\lambda):=\Coh_\cusp(G,\M_\lambda).
\]
In this case $\Coh_\cusp(G,\lambda)$ is nonempty only if
$\lambda\in X_{00}^+(T\times\C)$ (cf.\,\S\,\ref{sec:Galois-action-coho-mod}).

\subsection{Multiplicity one in bottom and top degrees}

Let $\lambda \in X_0^+(T \times \C)$. 
We denote by $\Omega(\lambda)$ the set of isomorphism classes of irreducible, admissible, generic $(\frak{g}_\infty, C_\infty)$-modules $\pi$ such that
\[
H^\bullet(\frak{g}_\infty, K_\infty^\circ; \pi \otimes \M_\lambda) \neq 0.
\]
We refer to \cite[\S\,2.4]{Raghuram2016} for an explicit description of $\Omega(\lambda)$ in terms of parabolically induced representations.
In particular, one has $\Omega(\lambda)^\vee = \Omega(\lambda^\vee)$, and 
\[
|\Omega(\lambda)| = \begin{cases}
1 & \mbox{ if $n$ is even or $S_r$ is empty},\\
2^{|S_r|} & \mbox{ otherwise}.
\end{cases}
\]
Also, the relative Lie algebra cohomology groups are nonvanishing if and only if
\[
b_n^F \leq \bullet \leq t_n^F,
\]
where $b_n^F$ and $t_n^F$ denote the bottom and top degrees, respectively, defined by
\begin{align*}
b_n^F &:= 
|S_r|\cdot\lfloor \tfrac{n^2}{4} \rfloor + |S_c|\cdot\tfrac{n(n-1)}{2},\\
t_n^F &:= |S_r|\cdot\lfloor \tfrac{n-1}{2} \rfloor + |S_c|\cdot(n-1) + b_n^F + [F:\Q]-1.
\end{align*}
We say that $\varepsilon \in \widehat{\pi_0(G(\R))}$ is \emph{permissible} for $\pi \in \Omega(\lambda)$ if the $\varepsilon$-isotypic component of the $(\frak{g}_\infty,K_\infty^\circ)$-cohomology of $\pi \otimes \M_\lambda$ is nonzero.
More precisely, every $\varepsilon$ is permissible if either $n$ is even or $S_r$ is empty. If $n$ is odd and $S_r$ is nonempty, then the only permissible character for $\pi$ is given by (cf.\,\S\,\ref{sec:Hecke-char-Gauss-sums})
\[
\varepsilon(\pi) := \varepsilon(\omega_{\pi}).
\]
Furthermore, for any permissible $\varepsilon$, we have
\[
\dim_\C H^\bullet(\frak{g}_\infty,K_\infty^\circ;\pi \otimes \M_\lambda)(\varepsilon) = 1,
\quad \text{for } \bullet = b_n^F,\, t_n^F.
\]

\section{\texorpdfstring{$\Aut(\C)$}{Aut(\textbf{C})}-actions and Betti–Whittaker periods}
\label{sec:aut-C-BW-periods}

\subsection{\texorpdfstring{$\Aut(\C)$}{Aut(\textbf{C})}-action on the cohomological models}
\label{sec:Galois-action-coho-mod}

Let $(\rho,\mathcal{M})$ be an algebraic representation of $G \times \C$. 
The action of ${\rm Aut}(\C)$ on the sheaf cohomology of $\widetilde{\mathcal{M}}$ is defined via its Betti realization.
Let $\sigma \in {\rm Aut}(\C)$. 
Recall the $\sigma$-conjugate representation $({}^\sigma\rho,{}^\sigma\mathcal{M})$ introduced in \S\,\ref{sec:irr-reps}, 
together with the $\sigma$-linear $G(\Q)$-equivariant isomorphism
$t_\sigma : \mathcal{M} \longrightarrow {}^\sigma\mathcal{M}$ 
defined in \eqref{eq:Galois-identity}.
For each open compact subgroup $K_f \subset G(\A_f)$, define a $\sigma$-linear isomorphism
\[
\sigma_{\rm B} : \Hom_\Z (C_\bullet^\infty(X/K_f),\mathcal{M})
\longrightarrow 
\Hom_\Z (C_\bullet^\infty(X/K_f),{}^\sigma\mathcal{M}),
\quad 
\ell \longmapsto t_\sigma \circ \ell .
\]
It is immediate that $\sigma_{\rm B}$ commutes with the boundary operator, the $G(\Q)$-action, and the Hecke action. More precisely,
\[
\sigma_{\rm B}\circ \partial^* = \partial^*\circ\sigma_{\rm B}, 
\quad 
\sigma_{\rm B}\circ \gamma = \gamma \circ \sigma_{\rm B}, 
\quad 
\sigma_{\rm B} \circ m_{x,{\rm B}} = m_{x,{\rm B}}\circ \sigma_{\rm B}
\]
for all $\gamma \in G(\Q)$ and $x \in \pi_0(G(\R))\times G(\A_f)$.
Passing to $G(\Q)$-invariants and taking the direct limit over $K_f$, we obtain a $\sigma$-linear and $\pi_0(G(\R))\times G(\A_f)$-equivariant isomorphism
\[
\sigma_{\rm B}^\bullet :
H_{\rm B}^\bullet(S^G,\widetilde{\mathcal{M}})
\longrightarrow
H_{\rm B}^\bullet(S^G,{}^\sigma\!\widetilde{\mathcal{M}}).
\]
Using the comparison isomorphism, we define the induced $\sigma$-linear and $\pi_0(G(\R))\times G(\A_f)$-equivariant isomorphism on the de Rham realization
\[
\sigma_{\rm dR}^\bullet :
H_{\rm dR}^\bullet(S^G,\widetilde{\mathcal{M}})
\longrightarrow
H_{\rm dR}^\bullet(S^G,{}^\sigma\!\widetilde{\mathcal{M}}), \ \ {\rm by} \ \ 
\sigma_{\rm dR}^\bullet
:=
(I^\bullet)^{-1}
\circ
\sigma_{\rm B}^\bullet
\circ
I^\bullet .
\]
In \cite[Théorème 3.19]{Clozel1990}, Clozel proved the ${\rm Aut}(\C)$-equivariance of the cuspidal cohomology. More precisely, we have
\[
\sigma_{\rm dR}^\bullet H_{\rm cusp}^\bullet(S^G,\widetilde{\mathcal{M}})
=
H_{\rm cusp}^\bullet(S^G,{}^\sigma\!\widetilde{\mathcal{M}}).
\]

Let $\Pi$ be a cuspidal automorphic representation of $G(\A)$. 
Following Clozel in \emph{loc.\ cit.}, we say that $\Pi$ is \emph{regular algebraic} if the infinitesimal character of $\Pi_\infty$ is regular and lies in
\[
\prod_{\tau : F \rightarrow \C} (\Z+\tfrac{n-1}{2})^n.
\]
By \cite[Théorème 3.13 and Lemme 4.9]{Clozel1990}, $\Pi$ is regular algebraic if and only if $\Pi_f \in \Coh_{\rm cusp}(G,\lambda)$ 
for some $\lambda \in X_{00}^+(T\times\C)$.
Assume we are in this case.
For every $\sigma \in {\rm Aut}(\C)$, the irreducible admissible $(\frak{g}_\infty,C_\infty)\times G(\A_f)$-module
${}^\sigma\Pi := {}^\sigma\Pi_\infty \otimes {}^\sigma\Pi_f$
is again cuspidal automorphic and regular algebraic, where 
\[
{}^\sigma\Pi_f := \Pi_f \otimes_{\C,\sigma} \C
\] 
and ${}^\sigma\Pi_\infty$ is uniquely determined by 
\[
{}^\sigma\Pi_\infty \in \Omega({{}^\sigma\lambda}),
\]
and $\varepsilon({}^\sigma\Pi_\infty) = \varepsilon(\Pi_\infty)$ if $n$ is odd and $S_r$ is nonempty.
Moreover, we have ${}^\sigma\Pi_f \in \Coh_{\rm cusp}(G,{}^\sigma\lambda)$. 
The \emph{rationality field} $\Q(\Pi)$ of $\Pi$ is the fixed field in $\C$ of the stabilizer of $\Pi$ in ${\rm Aut}(\C)$; it is a number field.
Let $\bullet \in \{b_n^F,t_n^F\}$ be an extreme degree, and $\varepsilon \in \widehat{\pi_0(G(\R))}$ be permissible for $\Pi_\infty$. By taking the Galois invariants over $\Q(\Pi)$, we obtain a $\Q(\Pi)$-rational structure on the $\varepsilon\times\Pi_f$-isotypic component of the cuspidal cohomology (cf.\,\cite[Lemme 3.2.1]{Clozel1990}):
\begin{align}\label{eq:Coh-rational-structure}
\begin{split}
&H_{\rm cusp}^\bullet({S}^G,\widetilde{\M}_\lambda)(\varepsilon\times\Pi_f)^{{\rm Aut}(\C/\Q(\Pi))}\\
& := \left.\left\{\omega\in H_{\rm cusp}^\bullet(\mathcal{S}_n,\widetilde{\M}_\lambda)(\varepsilon\times\Pi_f)\,\right\vert\,\sigma_{\rm dR}^\bullet (\omega)=\omega\mbox{ for $\sigma \in {\rm Aut}(\C/\Q(\Pi))$}\right\}.
\end{split}
\end{align}
Here we have identified $({}^\sigma\rho_\lambda,{}^\sigma\M_\lambda)$ with $(\rho_{{}^\sigma\lambda},\mathcal{M}_{{}^\sigma\lambda})$ for $\sigma \in {\rm Aut}(\C)$ via the isomorphism (\ref{eq:finite-rep}).
In particular, ${}^\sigma\M_\lambda = \M_\lambda$ if $\sigma \in {\rm Aut}(\C/\Q(\Pi))$, since $\Q(\lambda) \subset \Q(\Pi)$.

\subsection{${\rm Aut}(\C)$-action on the Whittaker models}%
\label{sec:Galois-action-Wh-mod}%

Let $\Pi$ be a regular algebraic cuspidal automorphic representation of $G(\A)$.
For $\sigma \in {\rm Aut}(\C)$, as in Harder \cite{harder}, let 
\[
\underline{u}_\sigma:=\diag(u_\sigma^{-(n-1)}, u_\sigma^{-(n-2)}, \dotsc, 1) \in T(\A_f),
\]
where $u_\sigma \in \prod_p \Z_p^\times \subset \A_{F,f}^\times$ is the unique element such that
$\sigma(\psi_\Q(x)) = \psi_\Q(u_\sigma x)$ for all $x \in \A_f$.
We define a $\sigma$-linear $G(\A_f)$-equivariant isomorphism 
\[
\sigma_\W : \W(\Pi_f)\longrightarrow \W({}^\sigma\Pi_f) 
\]
by
\[
(\sigma_\W W)(g):= \sigma(W(\underline{u}_\sigma g)),\quad g \in G(\A_f).
\]
By taking the Galois invariants over $\Q(\Pi)$, we obtain a subspace of $\W(\Pi_f)$ over $\Q(\Pi)$:
\begin{align}\label{eq:Wh-mod-rational-structure}
\mathcal{W}(\Pi_f)^{{\rm Aut}(\C/\Q(\Pi))} := \left.\left\{W\in \mathcal{W}(\Pi_f)\,\right\vert\,\sigma_\W W=W\mbox{ for $\sigma \in {\rm Aut}(\C/\Q(\Pi))$}\right\}.
\end{align}
Note that $\mathcal{W}(\Pi_f)^{{\rm Aut}(\C/\Q(\Pi))}$ is nonzero by the newform theory for $G(\A_f)$ \cite{JPSS1981}, the existence of the Kirillov model for $\Pi_f$ \cite[Theorem 5.20]{BZ1976}, and \cite[Lemme I.1]{Wald1985B}.
In particular, this defines a $\Q(\Pi)$-rational structure of $\W(\Pi_f)$.

\subsection{Betti–Whittaker periods}%
\label{sec:comparison-isomorphism}%
These periods were originally defined by Harder \cite{harder} and Hida \cite{hida-duke} independently. 
The definition given below is as in Raghuram--Shahidi \cite{RS2008}. 
Let $\Pi$ be a regular algebraic cuspidal automorphic representation of $G(\A)$ such that $\Pi_f \in \Coh_{\rm cusp}(G,\lambda)$ 
for some $\lambda \in X_{00}^+(T\times\C)$.
Let 
\[
\Upsilon_\Pi : \mathcal{W}(\Pi) \longrightarrow V_\Pi
\]
be the inverse of the isomorphism in (\ref{eq:Whittaker-realization}) with $\xi = \psi_N$. It induces a $\pi_0(G(\R))\times G(\A_f)$-equivariant isomorphism
\[
\Upsilon_\Pi^\bullet : H^\bullet(\frak{g}_\infty,K_\infty^\circ;\W(\Pi)\otimes\M_\lambda) \longrightarrow H_{\rm cusp}^\bullet(S^G,\widetilde\M_\lambda)(\Pi_f).
\]
Let $\bullet \in \{b_n^F,t_n^F\}$ be an extreme degree, and $\varepsilon \in \widehat{\pi_0(G(\R))}$ be permissible for $\Pi_\infty$.
Fix a generator
\[
[\Pi_\infty]^\varepsilon \in H^\bullet(\frak{g}_\infty,K_\infty^\circ;\W(\Pi_\infty)\otimes\M_\lambda)(\varepsilon).
\]
With respect to this choice, we define a $G(\A_f)$-equivariant isomorphism 
\[
\mathcal F_{\Pi_f,[\Pi_\infty]^\varepsilon} : \W(\Pi_f) \longrightarrow H_{\rm cusp}^\bullet(S^G,\widetilde\M_\lambda)(\varepsilon\times\Pi_f)
\]
by
\[
\mathcal{F}_{\Pi_f,[\Pi_\infty]^\varepsilon}(W):= \Upsilon_\Pi^\bullet([\Pi_\infty]^\varepsilon\otimes W).
\]
By comparing the $\Q(\Pi)$-rational structures of $\Pi_f$ in (\ref{eq:Coh-rational-structure}) and (\ref{eq:Wh-mod-rational-structure}), there exists a unique element
\[
p^\varepsilon(\Pi) \in \C^\times /\Q(\Pi)^\times
\]
such that 
\[
\mathcal{F}_{\Pi_f,[\Pi_\infty]^\varepsilon}\left(\frac{\mathcal{W}(\Pi_f)^{{\rm Aut}(\C/\Q(\Pi))}}{p^\varepsilon(\Pi)}\right) = H_{\rm cusp}^\bullet({S}^G,\widetilde{\M}_\lambda)(\varepsilon\times\Pi_f)^{{\rm Aut}(\C/\Q(\Pi))}.
\]
We call $p^\varepsilon(\Pi)$ the \emph{bottom-degree (resp.\,top-degree) Betti--Whittaker period} of $\Pi$ with respect to 
$[\Pi_\infty]^\varepsilon$ if $\bullet = b_n^F$ (resp.\,$\bullet = t_n^F$).

For each $\sigma \in {{\rm Aut}(\C)}$, we fix a generator $[{}^\sigma\Pi_\infty]^\varepsilon$ and obtain the associated Betti--Whittaker period $p^\varepsilon({}^\sigma\Pi)$ of ${}^\sigma\Pi$.
{For compatibility, we assume $[{}^\sigma\Pi_\infty]^\varepsilon = [\Pi_\infty]^\varepsilon$ if $\sigma \in {\rm Aut}(\C/\Q(\Pi))$.}
We normalize the periods $(p^\varepsilon({{}^\sigma}\Pi))_{\sigma \in {{\rm Aut}(\C)}}$ such that 
\begin{align}\label{eq:W-H-periods}
\sigma_{\rm dR}^\bullet\circ \frac{\mathcal{F}_{\Pi_f,[\Pi_\infty]^\varepsilon}}{p^\varepsilon(\Pi)} = \frac{\mathcal{F}_{{}^\sigma\Pi_f,[{}^\sigma\Pi_\infty]^\varepsilon}}{p^\varepsilon({}^\sigma\Pi)} \circ \sigma_\W,\quad \sigma \in {\rm Aut}(\C).
\end{align}
In other words, the diagram
\[
	\begin{tikzcd}[column sep=8em]
		\W(\Pi_f) \arrow[r,"\frac{\mathcal{F}_{\Pi_f,[\Pi_\infty]^\varepsilon}}{p^\varepsilon(\Pi)}"] \arrow[d,"\twistW"] & H_{\rm cusp}^{\bullet}(S^G, \widetilde \M_{\lambda})(\eps\times\Pi_f) \arrow[d,"\sigma_{\rm dR}^\bullet"] \\ \W({}^\sigma\Pi_f) \arrow[r,"\frac{\mathcal{F}_{{}^\sigma\Pi_f,[{}^\sigma\Pi_\infty]^\varepsilon}}{p^\varepsilon({}^\sigma\Pi)}"] & H_{\rm cusp}^{\bullet}(S^G, {}^\sigma\!\widetilde{\mathcal{M}}_\lambda)(\eps\times{}^\sigma\Pi_f)
	\end{tikzcd}
\]
is commutative for all $\sigma \in {\rm Aut}(\C)$.

\section{The dual representations}
\label{sec:dual-rep}

\subsection{An algebraic automorphism of $G$}\label{sec:alg-auto-1}

Let $\theta : G \rightarrow G$ be an algebraic automorphism defined by
\[
\theta(g):= \omega_n\cdot {}^tg^{-1} \cdot \omega_n^{-1},
\]
where $\omega_n$ denotes the signed anti-diagonal representative of the long Weyl element:
\begin{align}\label{eq:long-Weyl}
\omega_n:=
\begin{pmatrix}
0 & 0 & \cdots & (-1)^n\\
\vdots & \vdots & \iddots & \vdots\\
0 & 1 & \cdots & 0\\
-1 & 0 & \cdots & 0
\end{pmatrix}.
\end{align}
We also write $\theta(g) = {}^\theta g$. Note that $\theta$ is an involution, ${}^\theta C_\infty = C_\infty$, ${}^\theta N = N$, and 
\begin{align}\label{eq:theta-additive}
\psi_N({}^\theta u) = \psi_N(u),\quad u \in N(\A).
\end{align}

\subsection{The dual of an algebraic representation}%
\label{sec:dual-alg-rep}%

Let $(\rho,\M)$ be an algebraic representation of $G\times\C$. Define its \emph{$\theta$-twist representation} $({}^\theta\rho,{}^\theta\M)$ by 
\[
{}^\theta\M:=\M,\quad {}^\theta\rho(g):=\rho({}^\theta g).
\]

\begin{lemma}
Let $\lambda \in X^+(T\times\C)$. We have
\[
({}^\theta\rho_\lambda,{}^\theta\M_\lambda) \cong (\rho_\lambda^\vee,\M_\lambda^\vee).
\]
\end{lemma}

\begin{proof}
The proof---an easy exercise involving highest weights---is left to the reader. 
\end{proof}

\subsection{The dual of a local representation}\label{sec:local-dual}

Let $v$ be a place of $\Q$. Let $(\pi,V)$ be a representation of $G(\Q_p)$ if $v=p$ is a rational prime, and a $(\mathfrak g_\infty,C_\infty)$-module if $v=\infty$. 
We define the \emph{$\theta$-twist representation} $({}^\theta\pi,{}^\theta V)$ by ${}^\theta V := V$ and
\[
\begin{cases}
{}^\theta\pi(g) := \pi({}^\theta g), \quad g \in G(\Q_p) &\mbox{if $v=p$}, \\[6pt]
{}^\theta\pi(Y) := \pi(\theta^* Y), \quad {}^\theta\pi(k) := \pi({}^\theta k), \quad 
 Y \in \mathfrak g_\infty,\; k \in C_\infty&\mbox{if $v=\infty$}.
\end{cases}
\]
It is clear that $(\pi,V)$ is irreducible admissible if and only if $({}^\theta\pi,{}^\theta V)$ is. Moreover, in this case, we have
\begin{align}\label{eq:GK}
({}^\theta\pi,{}^\theta V) \cong (\pi^\vee,V^\vee),
\end{align}
by the result of Gelfand--Kazhdan \cite[Theorem 7.3]{BZ1976} when $v=p$, and can be directly verified by using parabolic induction when $v = \infty$.

\begin{lemma}\label{lem:GK}
Assume that $(\pi,V)$ is irreducible, admissible, and generic. Then we have
\[
\W(\pi^\vee) = \left\{{}^\theta W \,\vert\, W \in \W(\pi)\right\}.
\]
Here ${}^\theta W(g):=W({}^\theta g)$ for $g \in G(\Q_v)$.
\end{lemma}

\begin{proof}
Let $V=\W(\pi)$ be the Whittaker model of $\pi$ with respect to $\psi_{N,v}$.
Then $\pi$ acts on $V$ by right translation (and by the
corresponding differential action when $v=\infty$).
Let $\W_{\psi_{N,v}}$ denote the space of Whittaker functions on $G(\Q_v)$
with respect to $\psi_{N,v}$.
Consider the $\C$-linear map
$
\W(\pi)\longrightarrow \W_{\psi_{N,v}},$ given by $W\longmapsto {}^\theta W.$ 
It is straightforward to verify that this map induces an injective equivariant homomorphism
$({}^\theta\pi,{}^\theta V)\longemb \W_{\psi_{N,v}}.$ 
The assertion then follows from \eqref{eq:GK} and the multiplicity one for the Whittaker models of $\pi^\vee$.
\end{proof}

\subsubsection{${\rm Aut}(\C)$-action under duality}\label{sec:aut-local}

Assume that $v = p$ and that $(\pi,V)$ is irreducible, admissible, and generic. For $\sigma \in \mathrm{Aut}(\C)$, 
we define, as in \S\ref{sec:Galois-action-Wh-mod}, a $\sigma$-linear $G(\Q_p)$-equivariant isomorphism
$\sigma_{\W} : \W(\pi) \longrightarrow \W({}^\sigma\pi)$ by
$(\sigma_{\W} W)(g) := \sigma\bigl(W(\underline{u}_{\sigma,p} g)\bigr), \ g \in G(\Q_p),$ 
where $\underline{u}_{\sigma,p} \in T(\Q_p)$ denotes the $p$-component of $\underline{u}_{\sigma}$.
By Lemma \ref{lem:GK}, we have a $\C$-linear isomorphism
\[
\theta_{\W} : \W(\pi) \longrightarrow \W(\pi^\vee), \quad W \longmapsto {}^\theta W.
\]
The following lemma concerns the $\mathrm{Aut}(\C)$-equivariance of $\theta_{\W}$.

\begin{lemma}\label{lem:Wh-rational}
Assume that $v = p$ and that $(\pi,V)$ is irreducible, admissible, and generic.
We have 
\[
\sigma_\W\circ \theta_\W = {}^\sigma\omega_\pi(u_{\sigma,p})^{n-1} \cdot \theta_\W\circ \sigma_\W,\quad \sigma \in {\rm Aut}(\C).
\]
\end{lemma}

\begin{proof}
This is clear, since
$
{}^\theta \underline{u}_{\sigma,p} = u_{\sigma,p}^{n-1}\cdot \underline{u}_{\sigma,p}.
$
\end{proof}

\subsubsection{The dual of a local generator}\label{sec:archi-local-dual}

Assume that $v = \infty$ and that $\pi \in\Omega(\lambda)$ for some $\lambda \in X_0^+(T\times\C)$. Let $\theta_\W$ be the $\C$-linear 
isomorphism from $\W(\pi)$ to $\W(\pi^\vee)$ defined similarly as in \S\,\ref{sec:aut-local}.
Let $\bullet \in \{b_n^F,t_n^F\}$ be an extreme degree, and $\varepsilon \in \widehat{\pi_0(G(\R))}$ be permissible for $\pi$.
Let
\[
[\pi]^\varepsilon \in H^\bullet(\frak{g}_\infty,K_\infty^\circ;\W(\pi)\otimes\M_\lambda)(\varepsilon)
\]
be a generator.
Since $\pi$ is essentially unitary, we have (cf.\,\cite[II, Proposition 3.1]{BW2000})
\begin{align*}
H^\bullet(\frak{g}_\infty,K_\infty^\circ;\W(\pi)\otimes\M_\lambda) &= {\rm Hom}_{K_\infty^\circ}\left(\extp^\bullet \frg_\infty/\frk_\infty,
\W(\pi)\otimes\M_\lambda\right),\\
H^\bullet(\frak{g}_\infty,K_\infty^\circ;\W(\pi^\vee)\otimes{}^\theta\M_\lambda) &= {\rm Hom}_{K_\infty^\circ}\left(\extp^\bullet \frg_\infty/\frk_\infty,
\W(\pi^\vee)\otimes{}^\theta\M_\lambda\right).
\end{align*}
We define the dual generator 
\[
{}^\theta[\pi]^\varepsilon \in H^\bullet(\frak{g}_\infty,K_\infty^\circ;\W(\pi^\vee)\otimes{}^\theta\M_\lambda)(\varepsilon)
\]
by
\[
{}^\theta[\pi]^\varepsilon:=(\theta_\W \otimes {\rm id})\circ [\pi]^\varepsilon \circ \extp^\bullet d\theta.
\]
In other words, we have the following commutative diagram
\begin{equation}\label{eq:theta-twist}
\begin{tikzcd}[row sep=normal, column sep=normal]
\extp^{\bullet}\frak{g}_\infty/\frak{k}_\infty \arrow[r, "{[\pi]}^\varepsilon"]  & \mathcal{W}(\pi)\otimes \mathcal{M}_{\lambda}\arrow[d, "\theta_\W \otimes {\rm id}"]\\
\extp^{\bullet}\frak{g}_\infty/\frak{k}_\infty \arrow[r, "{}^\theta{[\pi]}^\varepsilon"]\arrow[u, "\extp^\bullet d\theta"] & \mathcal{W}(\pi^\vee)\otimes \mathcal{M}_{\lambda-{\sf w}}.
\end{tikzcd}
\end{equation}

\subsection{The dual of a cuspidal automorphic representation}

The algebraic automorphism $\theta$ induces a $\C$-linear isomorphism
\begin{align*}
C^\infty(G(\Q)\backslash G(\A))
\longrightarrow
C^\infty(G(\Q)\backslash G(\A)),\quad
\phi
\longmapsto
{}^\theta\phi
\end{align*}
where ${}^\theta\phi(g):=\phi({}^\theta g)$.
When restricted to the space of automorphic forms
$\mathcal{A}(G(\Q)\backslash G(\A))$, this map defines a
$(\frak g_\infty,C_\infty)\times G(\A_f)$-equivariant
isomorphism of the space with itself.
Indeed, a function $\phi$ is $C_\infty$-finite if and only if
${}^\theta\phi$ is $C_\infty$-finite, since
${}^\theta C_\infty=C_\infty$.
The same argument applies to the other conditions required
for $\phi$ to be an automorphic form on $G(\A)$.

The following lemma is a well-known consequence of (\ref{eq:GK}) and Lemma \ref{lem:GK}, we include a proof for the convenience of the reader.
\begin{lemma}
Let $\Pi$ be a cuspidal automorphic representation of $G(\A)$.
Then we have
\[
V_{\Pi^\vee} = \left\{{}^\theta \phi \,\vert\, \phi \in V_{\Pi}\right\},\quad \W(\Pi^\vee) = \left\{{}^\theta W \,\vert\, W \in \W(\Pi)\right\}.
\]
\end{lemma}

\begin{proof}
Note that the subspace ${}^\theta V_\Pi \subset \mathcal{A}_{\rm cusp}(G(\Q)\backslash G(\A))$ is an irreducible admissible $(\frak g_\infty,C_\infty)\times G(\A_f)$-submodule.
Indeed, a subspace $W \subset V_\Pi$ is invariant if and only if ${}^\theta W \subset {}^\theta V_\Pi$ is invariant.
Thus ${}^\theta V_\Pi = V_{\Pi'}$ for some cuspidal automorphic representation $\Pi'$ of $G(\A)$.
By (\ref{eq:theta-additive}), we have
\[
{}^\theta(W_{\psi_N}(\phi)) = W_{\psi_N}({}^\theta\phi),\quad \phi \in \mathcal{A}(G(\Q)\backslash G(\A)).
\]
In particular, 
\[
{}^\theta \W(\Pi) = \W(\Pi').
\] 
On the other hand, by Lemma \ref{lem:GK} we have
\[
{}^\theta \W(\Pi) = \W(\Pi^\vee).
\]
It then follows from the uniqueness of the Whittaker model that $\Pi' = \Pi^\vee$. This completes the proof.
\end{proof}

\begin{remark}
In fact, by the strong multiplicity one theorem, it suffices to have (\ref{eq:GK}) at unramified places. 
This follows from 
\[
{}^\theta {\rm Ind}_{B(\Q_p)}^{G(\Q_p)}(\chi_1 \boxtimes\cdots\boxtimes\chi_n) = {\rm Ind}_{B(\Q_p)}^{G(\Q_p)}(\chi_n^{-1} \boxtimes\cdots\boxtimes\chi_1^{-1})
\]
for characters $\chi_1,...,\chi_n$ of $F_p^\times = (F\otimes_\Q\Q_p)^\times$.
\end{remark}

\subsection{The dual of the cohomological models}%
\label{sec:dual-arithm}%

Let $(\rho,\M)$ be an algebraic representation of $G\times\C$. The algebraic automorphism $\theta$ induces natural cochain maps for the Betti and 
de~Rham complexes. For an open compact subgroup $K_f$ of $G(\A_f)$, let
\[
\theta: X/{}^\theta K_f \longrightarrow X/K_f,\quad x\longmapsto {}^\theta x.
\]

For each open compact subgroup $K_f$ of $G(\A_f)$, define a cochain map
\begin{align*}
\theta_{\rm B} :
\left(\Hom_\Z(C_\bullet^\infty(X/K_f),\mathcal{M}),\partial^\ast\right)
&\longrightarrow
\left(\Hom_\Z(C_\bullet^\infty(X/{}^\theta K_f),{}^\theta\mathcal{M}),\partial^\ast\right),
\quad \ell \longmapsto \ell \circ \theta_\sharp\\
\theta_{\rm dR} :
\left(\Omega^\bullet(X/K_f;\M),d\right)
&\longrightarrow
\left(\Omega^\bullet(X/{}^\theta K_f;{}^\theta\M),d\right),
\quad \omega \longmapsto \theta^*\omega .
\end{align*}
Here $\theta_\sharp(\alpha):=\theta\circ\alpha$.
Let ${\star} = {\rm B}$ or ${\star} = {\rm dR}$.
One can easily verify: 
\begin{align*}
\theta_{\star}\circ \gamma &= {}^\theta\gamma \circ \theta_{\star}, \quad
\theta_{\star} \circ m_{x,{\star}} = m_{{}^\theta x,{\star}}\circ \theta_{\star}
\end{align*}
for all $\gamma \in G(\Q)$ and $x \in \pi_0(G(\R))\times G(\A_f)$.
Passing to $G(\Q)$-invariants and taking the direct limit over $K_f$, we obtain a $\C$-linear isomorphism
\begin{align*}
\theta_{\star}^\bullet :
H_{\star}^\bullet(S^G,\widetilde{\mathcal{M}})
&\longrightarrow
H_{\star}^\bullet(S^G,{}^\theta\widetilde{\mathcal{M}})
\end{align*}
such that
\begin{align}\label{eq:theta-coh-rational}
\theta_{\star}^\bullet \circ m_{x,{\star}}^\bullet = m_{{}^\theta x,{\star}}^\bullet\circ \theta_{\star}^\bullet, \quad x \in \pi_0(G(\R))\times G(\A_f).
\end{align}
In particular, $\theta_{\star}^\bullet$ is $\pi_0(G(\R))$-equivariant.
Moreover, we have the following result on the compatibility of the comparison isomorphism with $\theta_{\star}$.

\begin{proposition}\label{prop:theta-dR-B}
We have
\begin{align*}
I^\bullet\circ \theta_{\rm dR}^\bullet = \theta_{\rm B}^\bullet\circ I^\bullet.
\end{align*}
\end{proposition}

\begin{proof}
Indeed, 
\begin{align*}
I(\theta^*\omega)(\alpha) = \int_{\Delta^\bullet}\alpha^*(\theta^*\omega) = \int_{\Delta^\bullet}(\theta\circ\alpha)^*\omega = I(\omega)(\theta_\sharp(\alpha)).
\end{align*}
Therefore, 
$I\circ \theta_{\rm dR} = \theta_{\rm B}\circ I.$
The assertion then follows immediately.
\end{proof}

In the following proposition, we prove that $\theta_{\star}^\bullet$ commutes with the ${\rm Aut}(\C)$-action.

\begin{proposition}\label{prop:theta-deRham-rational}
Let ${\star} = {\rm B}$ or ${\star} = {\rm dR}$.
We have
\[
\sigma_{\star}^\bullet \circ \theta_{\star}^\bullet =  \theta_{\star}^\bullet \circ \sigma_{\star}^\bullet,\quad \sigma \in {\rm Aut}(\C).
\]
In other words, the diagram
\begin{equation*}
		\begin{tikzcd}
			H^\bullet_{\star}(S^G, \widetilde \M) \arrow[r,"\theta_{\star}^\bullet"] \arrow[d,"\sigma_{\star}^\bullet"] & 
			H_{\star}^\bullet(S^G, {}^\theta\widetilde \M) \arrow[d,"\sigma_{\star}^\bullet"] \\
			H^\bullet_{\star}(S^G, {}^\sigma\!\widetilde \M) \arrow[r,"\theta_{\star}^\bullet"] & 
			H^\bullet_{\star}(S^G, {}^{\sigma\circ\theta}\!\widetilde \M) 
		\end{tikzcd}
		\end{equation*}
commutes for any $\sigma \in {\rm Aut}(\C)$.
\end{proposition}

\begin{proof}
Implicitly in the statement, we have
\[
({}^{\sigma}({}^\theta\rho),{}^{\sigma}({}^\theta\M)) = ({}^{\theta}({}^\sigma\rho),{}^{\theta}({}^\sigma\M)),\quad \sigma \in {\rm Aut}(\C).
\]
Indeed, ${}^{\sigma}({}^\theta\M) = {}^{\theta}({}^\sigma\M) = \M \otimes_{\C,\sigma}\C$ by definition.
For $g \in G(\C)$ and $m\otimes z \in \M \otimes_{\C,\sigma}\C$, we have
\begin{align*}
{}^{\sigma}({}^\theta\rho)(g) (m\otimes z) 
& = {}^\theta\rho(\sigma^{-1}g)m\otimes z = \rho(\theta(\sigma^{-1}g))m \otimes z \\
& = \rho(\sigma^{-1}\theta(g))m\otimes z = {}^\sigma\rho(\theta(g))(m\otimes z) = {}^{\theta}({}^\sigma\rho)(m\otimes z).
\end{align*}


By the definition of $\sigma^\bullet_{\mathrm{dR}}$ and Proposition \ref{prop:theta-dR-B}, it suffices to prove the statement in the case 
$\star = \mathrm{B}$. Since both $\sigma_{\mathrm{B}}$ and $\theta_{\mathrm{B}}$ are cochain maps, it is enough to verify that
\[
\sigma_{\mathrm{B}} \circ \theta_{\mathrm{B}} = \theta_{\mathrm{B}} \circ \sigma_{\mathrm{B}}.
\]
Consider the commutative diagram
\[
\begin{tikzcd}
	\mathcal{M} \arrow[r,"{\mathrm{id}^1}"] \arrow[d,"{t_\sigma^1}"] 
	& {}^\theta\mathcal{M} \arrow[d,"{t_\sigma^2}"] \\
	{}^\sigma\mathcal{M} \arrow[r,"{\mathrm{id}^2}"] 
	& {}^{\sigma}({}^\theta\mathcal{M}) = {}^{\theta}({}^\sigma\mathcal{M})
\end{tikzcd}
\]
where the vertical maps are the $\sigma$-linear isomorphisms defined in \eqref{eq:Galois-identity}, and the horizontal maps are identity maps. 
The commutativity of this diagram follows immediately from the definition.
For $\ell \in \operatorname{Hom}_{\mathbb{Z}}(C_\bullet^\infty(X/K_f), \mathcal{M})$, we compute
\[
(\sigma_{\mathrm{B}} \circ \theta_{\mathrm{B}})(\ell)
= t_\sigma^2 \circ \mathrm{id}^1 \circ \ell \circ \theta_\sharp
= \mathrm{id}^2 \circ t_\sigma^1 \circ \ell \circ \theta_\sharp
= (\theta_{\mathrm{B}} \circ \sigma_{\mathrm{B}})(\ell).
\]
This completes the proof.
\end{proof}

Recall that in \S\ref{sec:coh-groups-transc-level} we identified
\(H_{\mathrm{dR}}^\bullet(S^G,\mathcal{M})\) with the automorphic cohomology of \(G\)
with respect to the coefficient system \((\rho,\mathcal{M})\).
On the level of cochain \((\mathfrak{g}_\infty,K_\infty^\circ)\)-complexes,
the map 
\begin{align*}
\theta_{\mathrm{dR}} \ : \ &\left(
\operatorname{Hom}_{K_\infty^\circ}\left(
\extp^\bullet\mathfrak{g}_\infty/\mathfrak{k}_\infty,
C^\infty\bigl(G(\mathbb{Q})\backslash G(\mathbb{A})/K_f\bigr)\otimes \mathcal{M}
\right), d
\right) \\
&\longrightarrow
\left(
\operatorname{Hom}_{K_\infty^\circ}\left(
\extp^\bullet\mathfrak{g}_\infty/\mathfrak{k}_\infty,
C^\infty\bigl(G(\mathbb{Q})\backslash G(\mathbb{A})/{}^\theta\!K_f\bigr)\otimes {}^\theta\!\mathcal{M}
\right), d
\right),
\end{align*}
is given by
\begin{align}\label{eq:Lie-deRham}
f \longmapsto (\theta \otimes \mathrm{id}) \circ f \circ \extp^\bullet d\theta.
\end{align}

\section{The period relation}%
\label{sec:period-rel}%

\subsection{Statement and proof of the period relation}

Let $\Pi$ be a regular algebraic cuspidal automorphic representation of $G(\A)$ such that $\Pi_f \in \Coh_{\rm cusp}(G,\lambda)$ for 
some $\lambda \in X_{00}^+(T\times\C)$.
Let $\bullet \in \{b_n^F,t_n^F\}$ be an extreme degree, and $\varepsilon \in \widehat{\pi_0(G(\R))}$ be permissible for $\Pi_\infty$.
Fix a generator
\[
[\Pi_\infty]^\varepsilon \in H^\bullet(\frak{g}_\infty,K_\infty^\circ;\W(\Pi_\infty)\otimes\M_\lambda)(\varepsilon).
\]
As in \S\,\ref{sec:archi-local-dual}, its dual generator is denoted by
\[
{}^\theta[\Pi_\infty]^\varepsilon \in H^\bullet(\frak{g}_\infty,K_\infty^\circ;\W(\Pi_\infty^\vee)\otimes{}^\theta\M_\lambda)(\varepsilon).
\]
We then obtain the associated Betti--Whittaker periods $p^\varepsilon(\Pi)$ and $p^\varepsilon(\Pi^\vee)$ of $\Pi$ and $\Pi^\vee$ 
respectively as in \S\,\ref{sec:comparison-isomorphism}.
Similarly, for each $\sigma \in {{\rm Aut}(\C)}$ we fix a generator $[{}^\sigma\Pi_\infty]^\varepsilon$ and obtain the associated Betti--Whittaker periods $p^\varepsilon({}^\sigma\Pi)$ and $p^\varepsilon({}^\sigma\Pi^\vee)$.
For compatibility, we assume $[{}^\sigma\Pi_\infty]^\varepsilon = [\Pi_\infty]^\varepsilon$ if $\sigma \in {\rm Aut}(\C/\Q(\Pi))$.
The periods $(p^\varepsilon({}^\sigma\Pi))_{\sigma \in {{\rm Aut}(\C)}}$ and $(p^\varepsilon({}^\sigma\Pi^\vee))_{\sigma \in {{\rm Aut}(\C)}}$ are normalized by (\ref{eq:W-H-periods}).

Following is our main result on the period relation under duality.

\begin{theorem}\label{thm:period-relation}
We have
	\begin{equation*}
		\sigma\left( \frac{p^\eps(\Pi\dual)}{\Gsum(\omega_{\Pi})^{1-n}\cdot p^\eps(\Pi)} \right) = 
		\frac{p^\eps({}^\sigma\Pi\dual)}{\Gsum({}^\sigma\omega_{\Pi})^{1-n}\cdot p^\eps({}^\sigma\Pi)},\quad \sigma \in {\rm Aut}(\C).
	\end{equation*}
	In particular,
	\begin{equation*}
		p^\eps(\Pi\dual) \sim_{\Q(\Pi)} \Gsum(\omega_{\Pi})^{1-n}\cdot p^\eps(\Pi),
	\end{equation*}
	where ``$\sim_{\Q(\Pi)}$'' denotes equality up to a nonzero element in $\Q(\Pi)$.
\end{theorem}

\begin{proof}
Consider the diagram
\begingroup
\begin{equation}\label{eq:cube-diagram}
	\begin{tikzpicture}[baseline= (a).base]
		\node[scale=0.813] (a) at (0,0){
		\begin{tikzcd}[column sep={4.8cm,between origins}]
		& [0.5cm] \W(\Pi_f^\vee) \arrow[rr,"\mathcal F_{\Pi_f^\vee,{}^\theta[\Pi_\infty]^\varepsilon}"] \arrow[d,"\twistW"] & & 
		[0.5cm] H_{\rm cusp}^\bullet(S^G, {}^\theta\widetilde \M_{\lambda})(\eps\times\Pi_f^\vee) \arrow[d,"\sigma_{\rm dR}^\bullet"] \\[20pt]
		& [0.5cm] \W({}^\sigma\Pi_f^\vee) \arrow[ddl,swap,"\involW" {xshift=-5pt,yshift=-5pt},<-] & & 
		[0.5cm] H_{\rm cusp}^\bullet(S^G, {}^{\sigma\circ\theta}\widetilde\M_{\lambda})(\eps\times{}^\sigma\Pi_f^\vee) 
		\arrow[ll,<-,swap,"\mathcal F_{{}^\sigma\Pi_f^\vee,{}^\theta[{}^\sigma\Pi_\infty]^\varepsilon}"] \\
		\W(\Pi_f)
  \arrow[rr,crossing over,swap,
    "{\mathcal F_{\Pi_f,[\Pi_\infty]^\varepsilon}}" {yshift=-0.5ex}] \arrow[d,"\twistW"] \arrow[uur,"\involW" {xshift=-5pt,yshift=-5pt}] & & 
    H_{\rm cusp}^\bullet(S^G, \widetilde \M_{\lambda})(\eps\times \Pi_f) \arrow[d,"\sigma_{\rm dR}^\bullet"] 
    \arrow[uur,crossing over,"\theta_{\rm dR}^\bullet" {xshift=15pt,yshift=10pt}] & 
    \\[20pt] \W({}^\sigma\Pi_f) \arrow[rr,"\mathcal F_{{}^\sigma \Pi_f,[{}^\sigma\Pi_\infty]^\varepsilon}"]  & & 
    H_{\rm cusp}^\bullet(S^G, {}^\sigma\widetilde \M_{\lambda})(\eps\times{}^\sigma \Pi_f) \arrow[uur,"\theta_{\rm dR}^\bullet" {xshift=15pt,yshift=10pt}] &
		\end{tikzcd}
	};
	\end{tikzpicture}
\end{equation}
\endgroup
We emphasize that this diagram is \emph{not} commutative.
While half of the faces are commutative, the remaining ones become commutative after appropriate normalization of the maps by scaling.
By Lemma \ref{lem:Wh-rational} together with \eqref{eq:Gauss-sum}, we have
\[
\sigma_\W \circ \frac{\theta_\W}{\mathcal{G}(\omega_\Pi)^{n-1}}
=
\frac{\theta_\W}{\mathcal{G}({}^\sigma\omega_\Pi)^{n-1}} \circ \sigma_\W.
\]
Combining this with \eqref{eq:W-H-periods}, we obtain
\begin{multline*}
\sigma_{\rm dR}^\bullet
\circ
\frac{\mathcal{F}_{\Pi_f^\vee,{}^\theta[\Pi_\infty]^\varepsilon}}{p^\varepsilon(\Pi^\vee)}
\circ
\frac{\theta_\W}{\mathcal{G}(\omega_\Pi)^{n-1}}
\circ
\left(\frac{\mathcal{F}_{\Pi_f,[\Pi_\infty]^\varepsilon}}{p^\varepsilon(\Pi)}\right)^{-1}
\\
\ = \ 
\frac{\mathcal{F}_{{}^\sigma\Pi_f^\vee,{}^\theta[{}^\sigma\Pi_\infty]^\varepsilon}}{p^\varepsilon({}^\sigma\Pi^\vee)}
\circ
\frac{\theta_\W}{\mathcal{G}({}^\sigma\omega_\Pi)^{n-1}}
\circ
\left(\frac{\mathcal{F}_{{}^\sigma\Pi_f,[{}^\sigma\Pi_\infty]^\varepsilon}}{p^\varepsilon({}^\sigma\Pi)}\right)^{-1}
\circ
\sigma_{\rm dR}^\bullet.
\end{multline*}
Therefore, the desired period relation is equivalent to
\[
\sigma_{\rm dR}^\bullet
\circ
\mathcal{F}_{\Pi_f^\vee,{}^\theta[\Pi_\infty]^\varepsilon}
\circ
\theta_\W
\circ
\mathcal{F}_{\Pi_f,[\Pi_\infty]^\varepsilon}^{-1}
=
\mathcal{F}_{{}^\sigma\Pi_f^\vee,{}^\theta[{}^\sigma\Pi_\infty]^\varepsilon}
\circ
\theta_\W
\circ
\mathcal{F}_{{}^\sigma\Pi_f,[{}^\sigma\Pi_\infty]^\varepsilon}^{-1}
\circ
\sigma_{\rm dR}^\bullet.
\]
On the other hand, we claim that
\[
\mathcal{F}_{\Pi_f^\vee,{}^\theta[\Pi_\infty]^\varepsilon}
\circ
\theta_\W
\circ
\mathcal{F}_{\Pi_f,[\Pi_\infty]^\varepsilon}^{-1}
=
\theta_{\rm dR}^\bullet,
\quad
\mathcal{F}_{{}^\sigma\Pi_f^\vee,{}^\theta[{}^\sigma\Pi_\infty]^\varepsilon}
\circ
\theta_\W
\circ
\mathcal{F}_{{}^\sigma\Pi_f,[{}^\sigma\Pi_\infty]^\varepsilon}^{-1}
=
\theta_{\rm dR}^\bullet.
\]
In other words, the top and bottom faces of \eqref{eq:cube-diagram} are commutative.
To verify this, write
\[
[\Pi_\infty]^\varepsilon = \sum_i X_i^* \otimes W_i \otimes m_i
\]
for some \(X_i \in \extp^\bullet(\frak{g}_\infty/\frak{k}_\infty)^*\), 
\(W_i \in \W(\Pi_\infty)\), and \(m_i \in \M_\lambda\).
By the definition of \({}^\theta[\Pi_\infty]^\varepsilon\) and \eqref{eq:Lie-deRham}, for any \(W \in \W(\Pi_f)\), we have
\begin{align*}
(\mathcal{F}_{\Pi_f^\vee,{}^\theta[\Pi_\infty]^\varepsilon} \circ \theta_\W)(W)
&=
\sum_i \extp^\bullet(d\theta)^* X_i^*
\otimes
\Upsilon_{\Pi^\vee}({}^\theta W_i \otimes {}^\theta W)
\otimes m_i, \\
(\theta_{\rm dR}^\bullet \circ \mathcal{F}_{\Pi_f,[\Pi_\infty]^\varepsilon})(W)
&=
\sum_i \extp^\bullet(d\theta)^* X_i^*
\otimes
{}^\theta \Upsilon_{\Pi}(W_i \otimes W)
\otimes m_i.
\end{align*}
By \eqref{eq:theta-additive}, we have
\[
\Upsilon_{\Pi^\vee} \circ \theta = \theta \circ \Upsilon_\Pi,
\]
which proves the claim.
Consequently, the period relation holds if and only if the right-hand face of \eqref{eq:cube-diagram} is commutative. 
This follows from Proposition \ref{prop:theta-deRham-rational}, since 
\(\sigma_{\rm dR}^\bullet\) is \(\pi_0(G(\R)) \times G(\A_f)\)-equivariant, 
and \(\theta_{\rm dR}^\bullet\) is \(\pi_0(G(\R))\)-equivariant and $\theta$-twist equivariant under $G(\A_f)$ by (\ref{eq:theta-coh-rational}).
This completes the proof.
\end{proof}

\section{Archimedean period relations}
\label{sec:arch}

\subsection{Period relation under $\det$-twist and $\theta$-twist}

Let $\pi \in \Omega(\lambda)$ be an irreducible, admissible, generic $(\frak{g}_\infty,C_\infty)$-module for some pure weight $\lambda \in X_0^+(T\times\C)$. We denote by ${\sf w}$ the purity weight of $\lambda$. 
Let $\bullet \in \{b_n^F,t_n^F\}$ be an extreme degree, and $\varepsilon \in \widehat{\pi_0(G(\R))}$ be permissible for $\pi$.
Fix a generator
\[
[\pi]^\varepsilon \in H^\bullet(\frak{g}_\infty,K_\infty^\circ;\W(\pi)\otimes\M_\lambda)(\varepsilon).
\]
Assume 
\[
\lambda^{\overline{\tau}} = \lambda^{\tau}
\]
for any embedding $\tau :F\rightarrow \C$, that is,
\[
\lambda^\vee = \lambda-{\sf w}.
\] 
Then
\[
\pi^\vee \cong \pi \otimes \nu_{\sf w}\circ\det,
\]
where $\nu_{\sf w}$ is a character of $F_\infty^\times$ defined by
\begin{align}\label{eq:det-character}
\nu_{{\sf w},v}(a) := \begin{cases}
 a^{\sf w} & \mbox{ if $v \in S_r$},\\
 |a|_\C^{\sf w} & \mbox{ if $v \in S_c$}.
 \end{cases}
\end{align}
In this case, define a dual generator 
\begin{align}\label{eq:archi-local-dual-2}
\widetilde{[\pi]}^\varepsilon \in H^\bullet(\frak{g}_\infty,K_\infty^\circ;\W(\pi^\vee)\otimes\M_{\lambda-{\sf w}})(\varepsilon)
\end{align}
such that the diagram
\[
\begin{tikzcd}[row sep=normal, column sep=normal]
\extp^{\bullet}\frak{g}_\infty/\frak{k}_\infty \arrow[r, "{[\pi]}^\varepsilon"]  & \mathcal{W}(\pi)\otimes \mathcal{M}_{\lambda}\arrow[d, "(\,\cdot\,\otimes \nu_{\sf w}\circ\det) \otimes {\rm id}"]\\
\extp^{\bullet}\frak{g}_\infty/\frak{k}_\infty \arrow[r, "\widetilde{[\pi]}^\varepsilon"]\arrow[u, "{\rm id}"] & \mathcal{W}(\pi^\vee)\otimes \mathcal{M}_{\lambda-{\sf w}}
\end{tikzcd}
\]
is commutative. Here $(\rho_{\lambda-{\sf w}}, \mathcal{M}_{\lambda-{\sf w}})$ is the irreducible algebraic representation of $G\times\C$ of highest weight $\lambda-{\sf w}$ realized by 
\[
\mathcal{M}_{\lambda-{\sf w}}:= \mathcal{M}_\lambda,\quad \rho_{\lambda-{\sf w}}:= \rho_\lambda \otimes {\det}^{-{\sf w}}.
\]
On the other hand, as in \S\,\ref{sec:archi-local-dual}, we have another dual generator
\[
{}^\theta[\pi]^\varepsilon \in H^\bullet(\frak{g}_\infty,K_\infty^\circ;\W(\pi^\vee)\otimes{}^\theta\M_\lambda)(\varepsilon).
\]
Let 
\[
\Phi_\lambda : (\rho_{\lambda-{\sf w}}, \mathcal{M}_{\lambda-{\sf w}}) \longrightarrow ({}^\theta\rho_{\lambda}, {}^\theta\mathcal{M}_{\lambda})
\]
be the unique $G(\C)$-equivariant isomorphism normalized so that 
\begin{align}\label{eq:canonical-normalization}
\Phi_\lambda(m_\lambda) = m_\lambda
\end{align}
for any highest weight vector $m_\lambda \in \M_\lambda$.
Then there exists a nonzero constant $c^\varepsilon(\pi)$ such that
\begin{align}\label{eq:local-period}
({\rm id}\otimes \Phi_\lambda)\circ \widetilde{[\pi]}^\varepsilon = c^\varepsilon(\pi)\cdot {}^\theta[\pi]^\varepsilon.
\end{align}
It is clear that the constant is independent of the choice of $[\pi]^\varepsilon$.
Let $v$ be an archimedean place of $F$. Similarly we can define a constant $c^{\varepsilon_v}(\pi_v)$ for the $v$-component of $\pi$ with respect to the isomorphism in (\ref{eq:real-identification}).
More precisely, let $\frak{g}_v$ be the complexified Lie algebra of $\GL_n(F_v)$, 
\[
\mathcal{M}_{\lambda_v}:=\begin{cases}
\mathcal{M}_{\lambda^{\tau_v}} & \mbox{ if $v \in S_r$}\\
\mathcal{M}_{\lambda^{\tau_v}} \otimes \mathcal{M}_{\lambda^{{\tau}_v}} & \mbox{ if $v \in S_c$}
\end{cases},
\]
and
\[
b_n:= \begin{cases}
\lfloor \tfrac{n^2}{4} \rfloor & \mbox{ if $v \in S_r$}\\
\tfrac{n(n-1)}{2} & \mbox{ if $v \in S_c$}
\end{cases}, \quad t_n:= b_n + \begin{cases}
\lfloor \tfrac{n-1}{2} \rfloor & \mbox{ if $v \in S_r$}\\
n-1 & \mbox{ if $v \in S_c$}
\end{cases}.
\]
Then $c^{\varepsilon_v}(\pi_v)$ is the ratio between two generators in the $\varepsilon_v$-isotypic component of the $(\frak{g}_v,K_v^\circ)$-cohomology of $\mathcal{W}(\pi_v^\vee) \otimes \mathcal{M}_{\lambda_v-{\sf w}}$ and $\mathcal{W}(\pi_v^\vee) \otimes {}^\theta\mathcal{M}_{\lambda_v}$ in an extreme degree $\bullet \in \{b_n,t_n\}$, with respect to the $\GL_n(\C)$-equivariant isomorphism $\Phi_{\lambda_v}$ between $\mathcal{M}_{\lambda_v-{\sf w}}$ and ${}^\theta\mathcal{M}_{\lambda_v}$ normalized as above.

The following period relation is the main result of this section.

\begin{theorem}\label{thm:archi-period-relation}
Let $\pi \in \Omega(\lambda)$ for some $\lambda \in X_{0}^+(T\times\C)$. Assume 
\[
\lambda^{\sigma\circ\tau} = \lambda^{\sigma\circ\overline{\tau}},\quad \tau \in \Hom(F,\C),\ \sigma \in {\rm Aut}(\C).
\]
Then we have
\[
c^\varepsilon(\pi) = c^\varepsilon({}^\sigma \pi) \in \{\pm1\},\quad \sigma \in {\rm Aut}(\C).
\]
\end{theorem}

\begin{proof}
Note that $\Phi_\lambda = \otimes_{v \mid \infty}\Phi_{\lambda_v}$ by definition.
Therefore, by the K\"unneth formula, we have
\[
c^\varepsilon(\pi) = \prod_{v \mid \infty}c^{\varepsilon_v}(\pi_v).
\]
In Theorems \ref{thm:archi-period-relation-1} and \ref{thm:archi-period-relation-3} below, we show that for each archimedean place $v$, there exists $\eta_{v,n} \in \{\pm1\}$ depending only on $n$ and $F_v$ such that
\[
c^{\varepsilon_v}(\pi_v) = \eta_{v,n}.
\]
We write $\eta_{v,n} = \eta_{\R,n}$ (resp.\, $\eta_{v,n} = \eta_{\C,n}$) if $v \in S_r$ (resp.\,$v \in S_c$).
Then we have
\[
c^\varepsilon(\pi) = \eta_{\R,n}^{|S_r|}\cdot  \eta_{\C,n}^{|S_c|}.
\]
The assertion then follows immediately.
\end{proof}

\subsection{Real cases}
\label{sec:glnr}

In this and the following subsections, we drop the subscript $v$ for brevity and write
\[
\frak{g} = \frak{g}_v,\quad C=C_v,\quad K = K_v,\quad \mathcal{M}_\lambda = \mathcal{M}_{\lambda_v},\quad \Phi_\lambda = \Phi_{\lambda_v},\quad \pi = \pi_v,\quad \varepsilon = \varepsilon_v.
\]

In this subsection, let $v \in S_r$ be a real place.
For any given generator
\[
[\pi]^\varepsilon \in H^\bullet(\frak{g},K^\circ; \mathcal{W}(\pi^\vee)\otimes \mathcal{M}_\lambda)(\varepsilon),
\]
we can associate generators $\widetilde{[\pi]}^\varepsilon$ and ${}^\theta[\pi]^\varepsilon$ as above. 
Then $c^\varepsilon(\pi)$ is the scalar such that 
\[
({\rm id}\otimes \Phi_\lambda)\circ \widetilde{[\pi]}^\varepsilon = c^\varepsilon(\pi)\cdot {}^\theta[\pi]^\varepsilon.
\]

The main result of this section is to prove the following period relation.

\begin{theorem}\label{thm:archi-period-relation-1}
The scalar $c^\varepsilon(\pi)$ belongs to $\{\pm 1\}$ and depends only on $n$.
\end{theorem}

Recall the signed anti-diagonal representative $\omega_n$ of the longest Weyl element, as defined in \eqref{eq:long-Weyl}.
Let
\[
\vartheta:\GL_n\longrightarrow \GL_n,\quad
\vartheta(g)={}^tg^{-1},
\]
be the algebraic automorphism given by the inverse transpose. Thus
\[
\theta={}^{\omega_n}\vartheta.
\]
With respect to this decomposition, we decompose the diagram \eqref{eq:theta-twist}
into the following two subdiagrams:
\[
\begin{tikzcd}[row sep=normal, column sep=normal]
\extp^{\bullet}\frak{g}/\frak{k}
\arrow[r, "F_1"]
&
\mathcal{W}(\pi,\psi_{N_0})\otimes \mathcal{M}_{\lambda}
\arrow[d, "\vartheta_\W \otimes {\rm id}"]
\\
\extp^{\bullet}\frak{g}/\frak{k}
\arrow[r, "{}^\vartheta\!F_1"]
\arrow[u, "\extp^\bullet d\vartheta"]
&
\mathcal{W}(\pi^\vee,{}^\vartheta\psi_{N_0})\otimes {}^\vartheta\mathcal{M}_{\lambda}
\end{tikzcd}
\]
and
\[
\begin{tikzcd}[row sep=normal, column sep=normal]
\extp^{\bullet}\frak{g}/\frak{k}
\arrow[r, "F_2"]
&
\mathcal{W}(\pi^\vee,{}^\vartheta\psi_{N_0})\otimes {}^\vartheta\mathcal{M}_{\lambda}
\arrow[d, "{\rm Int}(\omega_n)^* \otimes {\rm id}"]
\\
\extp^{\bullet}\frak{g}/\frak{k}
\arrow[r, "{}^{\omega_n}F_2"]
\arrow[u, "\extp^\bullet {\rm Ad}(\omega_n)"]
&
\mathcal{W}(\pi^\vee,\psi_{N_0})\otimes {}^\theta\mathcal{M}_{\lambda}
\end{tikzcd}
\]
for any $\C$-linear homomorphisms $F_1$ and $F_2$.
It is easy to verify that, for all $k \in K$, we have
\[
k\cdot {}^\vartheta F_1 = {}^\vartheta(\vartheta(k)\cdot F_1),\quad k\cdot {}^{\omega_n}F_2 = {}^{\omega_n}(\omega_nk\omega_n^{-1}\cdot F_2).
\]
Hence in particular
\[
{}^\theta[\pi]^\varepsilon = {}^{\omega_n}({}^\vartheta[\pi]^\varepsilon).
\]

\begin{lemma}\label{lem:theta-to-vartheta}
We have
\[
(L_{\omega_n^{-1}} \otimes {}^\theta\rho_\lambda(\omega_n)^{-1} ) \circ {}^\theta[\pi]^\varepsilon = \varepsilon(-1)^n\cdot {}^\vartheta[\pi]^\varepsilon.
\]
Here \(L_g\) denotes the operator induced by left multiplication by \(g\)
\end{lemma}

\begin{proof}
Let $F \in {\rm Hom}_\C(\extp^{\bullet}\frak{g}/\frak{k}, \,\mathcal{W}(\pi^\vee,{}^\vartheta\psi_{N_0})\otimes {}^\vartheta\mathcal{M}_{\lambda})$. 
Note that
\[
L_{g^{-1}} \circ {\rm Int}(g)^* = R_{g^{-1}}, \quad g\in \GL_n(\R).
\]
Therefore, for $X \in \extp^{\bullet}\frak{g}/\frak{k}$, we have
\begin{align*}
& \left((L_{\omega_n^{-1}} \otimes {}^\theta\rho_\lambda(\omega_n)^{-1} ) \circ {}^{\omega_n}F \right) (X)\\
& = \left(\left((L_{\omega_n^{-1}}\circ {\rm Int}(\omega_n)^*) \otimes {}^\theta\rho_\lambda(\omega_n)^{-1} \right) \circ F \right) ({\rm Ad}(\omega_n)X)\\
& = \left( (R_{\omega_n^{-1}} \otimes {}^\theta\rho_\lambda(\omega_n)^{-1}) \circ F \right)({\rm Ad}(\omega_n)X)\\
& = (\omega_n^{-1}\cdot F)(X).
\end{align*}
Since $\det \omega_n = (-1)^n$, applying this identity to \(F={}^\vartheta[\pi]^\varepsilon\) gives the assertion.
\end{proof}

Define a $\GL_n(\C)$-equivariant isomorphism
\[
\varphi_\lambda:={}^\theta\rho_\lambda(\omega_n)^{-1} \circ \Phi_\lambda : (\rho_{\lambda-{\sf w}}, \mathcal{M}_{\lambda-{\sf w}}) \longrightarrow ({}^\vartheta\rho_{\lambda}, {}^\vartheta\mathcal{M}_{\lambda}).
\]
By our normalization of $\Phi_\lambda$, we have
\begin{align}\label{eq:normalization}
\varphi_\lambda(m_\lambda) = {}^\theta\rho_\lambda(\omega_n)^{-1}m_\lambda
\end{align}
for any highest weight vector $m_\lambda \in \mathcal{M}_\lambda$.
By Lemma \ref{lem:theta-to-vartheta}, Theorem \ref{thm:archi-period-relation-1} is equivalent to the following period relation.

\begin{theorem}\label{thm:archi-period-relation-2}
There exists $\eta \in \{\pm1\}$ depending only on $n$ such that
\[
\widetilde{[\pi]}^\varepsilon = \eta\varepsilon(-1)^n\cdot (L_{\omega_n}\otimes \varphi_\lambda^{-1})\circ {}^\vartheta[\pi]^\varepsilon.
\]
\end{theorem}

We explain the strategy of the proof, postponing the technical ingredients to the following subsections.

Let $\mathbf{c}^{\varepsilon}(\pi)$ be the ratio between the two generators.
The goal is to show that $\mathbf{c}^{\varepsilon}(\pi)\varepsilon(-1)^n$ is a sign depending only on $n$.
The proof is by induction on $n$. We realize $\pi$ as a normalized induced representation from a maximal Levi subgroup
\[
M_P\cong \GL_{n_1}\times \GL_{n_2},
\quad n_1+n_2=n,
\]
with $n_1n_2$ even, in such a way that the inducing data
$\pi_i\otimes\delta_{P,i}$ are again cohomological. After choosing a nonzero $M_P(\C)$-equivariant map
\[
\imath:
\mathcal M_{\lambda^{(1)}}\otimes \mathcal M_{\lambda^{(2)}}
\longrightarrow
\extp^{\frac{n_1n_2}{2}}\frak u_P^*\otimes \mathcal M_\lambda,
\]
Delorme's lemma, recalled in Lemma \ref{lem:Delorme}, identifies the cohomology of $\pi$ with the tensor product of the cohomologies of the two factors, with the contribution of a balanced Kostant representative recalled in \S\,\ref{sec:Kostant}.

There are then two ways to compare the generator attached to $\pi$ with the corresponding generator for the opposite parabolic. On the analytic side, Shahidi's formula for local coefficients recalled in \eqref{eq:Shahidi}, together with Weselmann's computation of the cohomological intertwining operator recalled in Proposition \ref{prop:Weselmann}, gives a comparison involving
\[
(\sqrt{-1})^{\frac{n_1n_2}{2}}
\varepsilon(0,\pi_1\times\pi_2^\vee,\psi_\R).
\]
On the algebraic side, Proposition \ref{prop:Delorme-theta-det} checks that the Delorme isomorphisms are compatible with the $\vartheta$-twist, the action of $L_{\omega_n}$, and the determinant twist. Combining these two comparisons gives the inductive relation
\[
\mathbf{c}^{\varepsilon}(\pi)
=
\pm
(\sqrt{-1})^{\frac{n_1n_2}{2}}
\varepsilon(0,\pi_1\times\pi_2^\vee,\psi_\R)
\cdot C(\lambda^{(1)},\lambda^{(2)})
\cdot
\mathbf{c}^{\varepsilon}(\pi_1)
\mathbf{c}^{\varepsilon}(\pi_2),
\]
where the initial sign depends only on $n_1,n_2$. Here
$C(\lambda^{(1)},\lambda^{(2)})$ is the scalar comparing the two equivariant maps
\[
\imath_2
\quad\text{and}\quad
\tilde{\jmath},
\]
defined from the initial choice of $\imath$ by the diagrams \eqref{eq:diagram-1}, \eqref{eq:diagram-2}, and \eqref{eq:diagram-3}.

The cases $n=1,2$ are computed directly in Lemma \ref{lem:base-cases}. It remains to prove that
\[
(\sqrt{-1})^{\frac{n_1n_2}{2}}
\varepsilon(0,\pi_1\times\pi_2^\vee,\psi_\R)
\cdot C(\lambda^{(1)},\lambda^{(2)})
\]
is a sign depending only on $n_1,n_2$. This is reduced in Lemma \ref{lem:core-sign} to an explicit computation for the partitions
\[
(n_1,n_2)=
\begin{cases}
(2r,1) & \mbox{ if $n=2r+1$},\\
(2r,2) & \mbox{ if $n=2r+2$}.
\end{cases}
\]
For these choices, the relevant balanced Kostant representative is written down explicitly, and the scalar
$C(\lambda^{(1)},\lambda^{(2)})$ is computed by evaluating $\imath_2$ and $\tilde{\jmath}$ on highest weight vectors. The archimedean epsilon factor is evaluated from the Langlands parameters using the standard formula for real epsilon factors. The two contributions cancel all dependence on the weight $\lambda$, leaving only a sign depending on $n$.

\subsubsection{Kostant representatives}\label{sec:Kostant}

Let $W_n$ be the Weyl group of $\GL_n$ with respect to the diagonal maximal torus $T_0 = T_0(n)$. For $w \in W_n$, we denote by $\dot{w}$ the unique permutation matrix representing $w$. 
We shall use the dot action 
\[ 
w\cdot\lambda:=w(\lambda+\rho_n)-\rho_n, 
\] where $\rho_n$ denotes the half-sum of the positive roots of $\GL_n$.

Let $P$ be an upper or lower triangular maximal parabolic subgroup of $\GL_n$. Let $M_P$ be the standard Levi subgroup of $P$ consisting of block diagonal matrices. Note that 
\[
M_P = \left.\left\{ \bp g_1 & 0 \\ 0 & g_2 \ep \,\right\vert\, g_1 \in \GL_{n_1},\, g_2 \in \GL_{n_2} \right\} \cong \GL_{n_1} \times \GL_{n_2}
\]
for some $n_1$ and $n_2$. 
Let \(\delta_{P,1}\) and \(\delta_{P,2}\) be characters of \(\GL_{n_1}(\mathbb R)\) and \(\GL_{n_2}(\mathbb R)\), respectively, defined by 
\[ \delta_P^{1/2} \left( \bp g_1 & 0 \\ 0 & g_2 \ep \right) = \delta_{P,1}(g_1)\delta_{P,2}(g_2), 
\] 
where \(\delta_P\) denotes the modulus character of \(P(\mathbb R)\). 

Assume that \(P\) is upper triangular. We now spell out Kostant's theorem \cite[Theorem 5.14]{Kostant1961} in the present setting.
We identify \(\frak g\) with the space of \(n\times n\) complex matrices. For \(1\leq i,j\leq n\), let \(X_{ij}\in \frak g\) denote the elementary matrix with \(1\) in the \((i,j)\)-entry and \(0\) elsewhere, and let \(\{X_{ij}^*\}\subset \frak g^*\) denote the dual basis.
Let $W_n^P$ be the set of Kostant representatives of $W_{M_P}\backslash W_n$, where $W_{M_P}$ is the Weyl group of $M_P$. Recall that $w \in W_n^P$ if and only if
\[
w(X^+(T_0))\subset X^+(T_0(n_1)) \times X^+(T_0(n_2)).
\]
Let $w \in W_n^P$. 
Define a set
\[
S_w^P:=\left\{ (i,j)\,\left\vert\,1 \leq i \leq n_1,\,1 \leq j \leq n_2,\,w^{-1}(i) > w^{-1}(n_1+j) \right\}\right..
\]
By \emph{loc.\ cit.}, the pure wedge
\[
v_w^P:= \extp_{(i,j) \in S_w^P}X_{i,n_1+j}^* \in \extp^{\ell(w)}\frak{u}_P^*
\]
is a highest weight vector of weight $w \cdot 0$, where the exterior product is taken in lexicographic order.
Moreover, for $\lambda \in X^+(T_0(n))$, write 
\[
w\cdot \lambda = (\lambda^{(1)},\lambda^{(2)})
\] 
for some $\lambda^{(i)} \in X^+(T_0(n_i))$. Then $\M_{\lambda^{(1)}} \otimes \M_{\lambda^{(2)}}$ appears in $\extp^{\ell(w)}\frak{u}_P^* \otimes \M_\lambda$ with multiplicity one, and a highest weight vector is given by
\[
v_w^P \otimes \rho_{\lambda}(\dot{w})v_\lambda,
\]
where $v_\lambda \in \M_\lambda$ is a highest weight vector.

Recall that $w \in W_n^P$ is called \emph{balanced} (cf.\,\cite[Definition 5.9]{HR2020}) if 
\[
\ell(w) = \tfrac{1}{2}\dim_\C\frak{u}_P = \tfrac{n_1n_2}{2}.
\]
We have a bijection (cf.\,\cite[Lemma 5.6]{HR2020})
\[
W_n^P \longrightarrow W_n^Q,\quad w \longrightarrow w_Pw
\]
with $\ell(w) + \ell(w_Pw) = n_1n_2$ and 
\[
w_Pw \cdot \lambda = (\lambda^{(2)}-n_1, \lambda^{(1)}+n_2).
\]



\subsubsection{Delorme's lemma}\label{sec:Delorme}

Let $P$ be an upper or lower triangular maximal parabolic subgroup of $\GL_n$ such that $M_P\cong \GL_{n_1} \times \GL_{n_2}$.
For $i=1,2$, we write
\[
\frak{g}_i:= \Lie(\GL_{n_i}(\R))_\C,\quad C_i:={\rm O}(n_i),\quad K_i:=C_i\R_+^\times.
\]
Let $(\pi_1,V_{\pi_1})$ and $(\pi_2,V_{\pi_2})$ be irreducible admissible $(\frak{g}_1,C_1)$ and $(\frak{g}_2,C_2)$ modules respectively.
We denote by $(\pi_1^\infty,V_{\pi_1}^\infty)$ and $(\pi_2^\infty,V_{\pi_2}^\infty)$ the Casselman--Wallach smooth globalization of $\pi_1$ and $\pi_2$ respectively.
We write
\[
(\pi,V_{\pi}) := \left({\rm Ind}_{P(\R)}^{\GL_n(\R)}(\pi_1{\otimes} \pi_2),\, {\rm Ind}_{P(\R)}^{\GL_n(\R)}(V_{\pi_1}{\otimes} V_{\pi_2})\right)
\]
for the $(\frak{g},C)$-module realized as the space of $C$-finite vectors of the normalized smooth induced representation
\[
{\rm Ind}_{P(\R)}^{\GL_n(\R)}(V_{\pi_1}^\infty\widehat{\otimes} V_{\pi_2}^\infty).
\]
Consider the following conditions on $\pi_1$ and $\pi_2$:
\begin{align}\label{eq:temper}
\begin{gathered}
\mbox{\(\pi_1\) and \(\pi_2\) are both essentially unitary};\\
\mbox{\(\pi_1\) and \(\pi_2\) have the same exponents};\\
\mbox{the ratio }\frac{L(s+1,\pi_1 \times \pi_2^\vee)}
{L(s,\pi_1 \times \pi_2^\vee)}
\mbox{ is holomorphic and nonvanishing at \(s=0\)}.
\end{gathered}
\end{align}
Under the first two conditions, $\pi$ is irreducible and essentially unitary.
We also consider the following cohomological condition:
\begin{align}\label{eq:cohomological-condition}
\begin{gathered}
\pi \in \Omega(\lambda),\quad \pi_1\otimes\delta_{P,1} \in \Omega(\lambda^{(1)}),\quad \pi_2\otimes\delta_{P,2} \in \Omega(\lambda^{(2)})\\
\mbox{for some pure weights $\lambda \in X_0^+(T_0(n))$, $\lambda^{(i)} \in X_0^+(T_0(n_i))$ for $i=1,2$.}
\end{gathered}
\end{align}
In this case, all the conditions in (\ref{eq:temper}) are satisfied. Indeed, the first condition is immediate, the purity of $\lambda$ implies the second condition, and the third condition follows from the regularity of $\lambda + \rho_n$.
Moreover, $n_1n_2$ must be even (cf.\,\cite[Proposition 7.26]{HR2020}) and there exists a balanced Kostant representative $w$ such that $w \cdot \lambda = (\lambda^{(1)},\lambda^{(2)})$ by the combinatorial lemma \cite[Lemma 7.14]{HR2020}.
The following lemma is an explicit form of a well-known result of Borel--Wallach \cite[III, Theorem 3.3]{BW2000}, also known as Delorme's lemma, applied to the induced representation \(\pi\).

\begin{lemma}\label{lem:Delorme}
Let \(\bullet \in \{b_n,t_n\}\) be an extreme degree, and let \(\bullet_i\) be the corresponding
extreme degree for \(\mathrm{GL}_{n_i}(\mathbb R)\), \(i=1,2\). Thus
\(\bullet_i=b_{n_i}\) if \(\bullet=b_n\), and
\(\bullet_i=t_{n_i}\) if \(\bullet=t_n\). 
Assume the cohomological condition (\ref{eq:cohomological-condition}) holds.
Let
\[
\imath : \M_{\lambda^{(1)}} \otimes \M_{\lambda^{(2)}}
\longrightarrow
\extp^{\frac{n_1n_2}{2}}\frak{u}_P^* \otimes \M_\lambda
\]
be a nonzero \(M_P(\mathbb C)\)-equivariant homomorphism, and let
\[
\imath^*:
\M_{\lambda^{(1)}}^\vee\otimes \M_{\lambda^{(2)}}^\vee
\longrightarrow
\extp^{\frac{n_1n_2}{2}}\frak{u}_P\otimes \M_\lambda^\vee
\]
be the corresponding \(M_P(\mathbb C)\)-equivariant homomorphism on
contragredient representations. Then we have a
\(\pi_0(\mathrm{GL}_n(\mathbb R))\)-equivariant isomorphism
\begin{align*}
I_\imath^P:\ 
&{\rm Hom}_{K^\circ}\left(
\extp^{\bullet}\frak{g}/\frak{k} \otimes \mathcal{M}_\lambda^\vee,
V_\pi\right) \\
&\longrightarrow
\extp^{d(\bullet)}\frak{a}_P/\frak{a}_G
\otimes
\left(
{\rm Hom}_{K_1^\circ}\left(
\extp^{\bullet_1}\frak{g}_1/\frak{k}_1
\otimes \mathcal{M}_{\lambda^{(1)}}^\vee,
V_{\pi_1\otimes\delta_{P,1}}\right)
\right. \\
&\hspace{4.5cm}\left.
\otimes
{\rm Hom}_{K_2^\circ}\left(
\extp^{\bullet_2}\frak{g}_2/\frak{k}_2
\otimes \mathcal{M}_{\lambda^{(2)}}^\vee,
V_{\pi_2\otimes\delta_{P,2}}\right)
\right)^{\pi_0(K^{M_P})}
\end{align*}
defined by
\[
I_\imath^P(F)\left(
Z\otimes(X_1\otimes m_1^\vee)\otimes(X_2\otimes m_2^\vee)
\right)
:=
F\left(
Z\otimes X_1\otimes X_2\otimes
\imath^*(m_1^\vee\otimes m_2^\vee)
\right)(1)
\]
for
\[
Z \in \extp^{d(\bullet)}\frak{a}_P/\frak{a}_G,\quad
X_i \otimes m_i^\vee
\in
\extp^{\bullet_i}\frak{g}_i/\frak{k}_i
\otimes \mathcal{M}_{\lambda^{(i)}}^\vee .
\]
Here \(K^{M_P}:=K^\circ \cap M_P(\mathbb R)\), and
\[
d(\bullet):=
\bullet-\bullet_1-\bullet_2-\tfrac{n_1n_2}{2}
=
\begin{cases}
0 & \mbox{ if } \bullet=b_n,\\
1 & \mbox{ if } \bullet=t_n.
\end{cases}
\]
\end{lemma}

\begin{remark}
The notation \(\imath^*\) does not mean the ordinary linear dual of
\(\imath\). Rather, since
\(\M_{\lambda^{(1)}}\otimes \M_{\lambda^{(2)}}\) occurs with multiplicity one in
\(\extp^{\frac{n_1n_2}{2}}\frak{u}_P^* \otimes \M_\lambda\), the choice of
\(\imath\) determines the corresponding homomorphism on the contragredient representations.
In the definition of \(I_\imath^P\), we use the standard decomposition
\[
\frak{g}/\frak{k}
\cong
(\frak{a}_P/\frak{a}_G)
\oplus
(\frak{g}_1/\frak{k}_1)
\oplus
(\frak{g}_2/\frak{k}_2)
\oplus
\frak{u}_P
\]
and the induced exterior product. 
\end{remark}

\subsubsection{Intertwining operators and Jacquet's integrals}

Keep the notation and assumptions of \S\,\ref{sec:Delorme}. Let \(w_P\in W_n\) be the relative longest Weyl element, represented by the block permutation matrix interchanging the two Levi blocks:
\[
\dot w_P
=
\bp
0 & {\bf 1}_{n_2} \\
{\bf 1}_{n_1} & 0
\ep .
\]
Let $Q$ be the maximal parabolic subgroup defined by 
\[ 
Q := \dot{w}_P\cdot{}^tP\cdot \dot{w}_P^{-1}. 
\]
For $f \in {\rm Ind}_{P(\R)}^{\GL_n(\R)}(V_{\pi_1} \otimes V_{\pi_2})$ and $\underline{s} = (s_1,s_2) \in \C^2$, let 
\[
f_{\underline{s}} \in {\rm Ind}_{P(\R)}^{\GL_n(\R)}\left((V_{\pi_1}\otimes|\det|^{s_1}) \otimes (V_{\pi_2}\otimes|\det|^{s_2}) \right)
\]
be the associated flat section. 
The standard intertwining integral
\[
M^Q(\underline{s})f(g)
:=
\int_{U_Q(\mathbb R)}
f_{\underline{s}}(\dot w_P^{-1}ug)\,du,
\quad g\in\GL_n(\mathbb R)
\]
converges absolutely for \({\rm Re}(s_1-s_2)\) sufficiently large (resp.\,sufficiently small) if \(P\) is upper triangular (resp.\,lower triangular). 
Here the Haar measure $du$ is normalized to be the product measure of the Lebesgue measures on $\R$.
It admits meromorphic continuation to \(\underline{s}\in\mathbb C^2\). 
If the conditions in (\ref{eq:temper}) are satisfied, then \(M(\underline{s})\) is holomorphic at \(\underline{s}=0\). In this case, we obtain an intertwining isomorphism
\[
M^Q:
{\rm Ind}_{P(\mathbb R)}^{\GL_n(\mathbb R)}
(V_{\pi_1}\otimes V_{\pi_2})
\longrightarrow
{\rm Ind}_{Q(\mathbb R)}^{\GL_n(\mathbb R)}
(V_{\pi_2}\otimes V_{\pi_1}),
\quad
f\longmapsto \left. M^Q(\underline{s})f\right|_{\underline{s}=0}.
\]

Assume that \(\pi_1\) and \(\pi_2\) are generic. Recall that
\(N_0=N_0(n)\) denotes the maximal unipotent subgroup of \(\GL_n\)
consisting of upper triangular unipotent matrices. Let \(\xi\) be a
non-degenerate character of \(N_0(n,\mathbb R)\). It determines
characters \(\xi_i\) of \(N_0(n_i,\mathbb R)\), \(i=1,2\), by
\[
\xi\big|_{M_P(\mathbb R)\cap N_0(n,\mathbb R)}
=
\xi_1\times \xi_2
\]
under the isomorphism \(M_P\cong \GL_{n_1}\times \GL_{n_2}\). We realize each
\(\pi_i\) in its Whittaker model with respect to $\xi_i$
\[
V_{\pi_i}=\mathcal{W}(\pi_i,\xi_i),\quad i=1,2.
\]
Assume that \(P\) is upper triangular. 
For $f\in {\rm Ind}_{P(\mathbb R)}^{\GL_n(\mathbb R)}
\left(
\mathcal W(\pi_1,\xi_1)\otimes \mathcal W(\pi_2,\xi_2)
\right)$, Jacquet's integral is defined by
\[
\mathbb W_\xi^P(g;f_{\underline{s}})
:=
\int_{U_Q(\mathbb R)}
f_{\underline{s}}(\dot w_P^{-1}ug)
({\bf 1}_{n_1},{\bf 1}_{n_2})\overline{\xi(u)}\,du,
\quad g\in \GL_n(\mathbb R).
\]
It converges absolutely for \({\rm Re}(s_1-s_2)\) sufficiently large and admits analytic continuation to \(\underline{s}\in\mathbb C^2\). Its continuation defines a Whittaker function on \(\GL_n(\mathbb R)\) with respect to \(\xi\).
If the first two conditions in \((\ref{eq:temper})\) are satisfied, then the specialization at \(\underline{s}=0\) gives an intertwining isomorphism
\[
\mathbb W_\xi^P:
{\rm Ind}_{P(\mathbb R)}^{\GL_n(\mathbb R)}
\left(
\mathcal W(\pi_1,\xi_1)\otimes \mathcal W(\pi_2,\xi_2)
\right)
\longrightarrow
\mathcal W(\pi,\xi),
\quad
f\longmapsto
\left.\mathbb W_\xi^P(f_{\underline{s}})\right|_{\underline{s}=0}.
\]
Assume in addition that the third condition in \((\ref{eq:temper})\) is also  satisfied and that \(\xi=\psi_{N_0}\) is the character defined as in
\((\ref{eq:additive})\). Then the ratio between
\(\mathbb W_{\psi_{N_0}}^P\) and
\(\mathbb W_{\psi_{N_0}}^Q\circ M^Q\) is given by the \(\gamma\)-factor for
\(\pi_1\times\pi_2^\vee\) with respect to \(\psi_{\mathbb R}\). 
More precisely, by Shahidi's formula for local coefficients \cite[Theorem 3.1]{Shahidi1985}, we have
\begin{align}\label{eq:Shahidi}
\mathbb W_{\psi_{N_0}}^P
=
\omega_{\pi_2}(-1)^{n_1}\cdot
\gamma(0,\pi_1\times\pi_2^\vee,\psi_{\mathbb R})\cdot
\mathbb W_{\psi_{N_0}}^Q\circ M^Q .
\end{align}
For the factor $\omega_{\pi_2}(-1)^{n_1}$, we refer to \cite[Theorem 5.1]{Shahidi1984} and \cite[Remark 5.1.3]{Shahidi2010}.

We now turn to the cohomological setting, and assume that the cohomological condition (\ref{eq:cohomological-condition}) is satisfied. 
Let $\bullet \in \{b_n,t_n\}$ be an extreme degree, and $\varepsilon \in \widehat{\pi_0(\GL_n(\R))}$ be permissible for $\pi$.
Let \(\bullet_i\) be the corresponding
extreme degree for \(\mathrm{GL}_{n_i}(\mathbb R)\), \(i=1,2\). 
By Lemma \ref{lem:Delorme}, $\varepsilon \in \widehat{\pi_0(\GL_n(\R))}$ is also permissible for 
\[
\pi_1 \otimes \delta_{P,1} = \pi_1 \otimes |\det|^{\frac{n_2}{2}},\quad \pi_2 \otimes \delta_{P,2} = \pi_2 \otimes |\det|^{-\frac{n_1}{2}}.
\]
For $i=1,2$, we fix a generator
\begin{align*}
[\pi_i \otimes \delta_{P,i}]^\varepsilon \in H^{\bullet_i}\left(\frak{g}_i,K_i^\circ; \mathcal{W}(\pi_i \otimes \delta_{P,i}) \otimes \mathcal{M}_{\lambda^{(i)}}\right)(\varepsilon).
\end{align*}
Since $n_1n_2$ is even, it is easy to see that $\varepsilon \in \widehat{\pi_0(\GL_n(\R))}$ is permissible for 
\[
\pi_2 \otimes \delta_{Q,1} = \pi_2 \otimes |\det|^{\frac{n_1}{2}},\quad \pi_1 \otimes \delta_{Q,2} = \pi_1 \otimes |\det|^{-\frac{n_2}{2}}.
\]
The generators $[\pi_i \otimes \delta_{P,i}]^\varepsilon$ for $i=1,2$ then uniquely determine generators 
\begin{align*}
&[\pi_2 \otimes \delta_{Q,1}]^\varepsilon \in H^{\bullet_2}\left(\frak{g}_2,K_2^\circ; \mathcal{W}(\pi_2 \otimes \delta_{Q,1}) \otimes \mathcal{M}_{\lambda^{(2)}-n_1}\right)(\varepsilon),\\
&[\pi_1 \otimes \delta_{Q,2}]^\varepsilon \in H^{\bullet_1}\left(\frak{g}_1,K_1^\circ; \mathcal{W}(\pi_1 \otimes \delta_{Q,2}) \otimes \mathcal{M}_{\lambda^{(1)}+n_2}\right)(\varepsilon)
\end{align*}
by
\begin{align}\label{eq:det-twist-generators}
\begin{split}
[\pi_2 \otimes \delta_{Q,1}]^\varepsilon &:= ({\det}^{n_1}\otimes{\rm id})\circ [\pi_2 \otimes \delta_{P,2}]^\varepsilon,\\
[\pi_1 \otimes \delta_{Q,2}]^\varepsilon &:= ({\det}^{-n_2}\otimes{\rm id})\circ [\pi_1 \otimes \delta_{P,1}]^\varepsilon.
\end{split}
\end{align}
Let 
\[
\imath : \mathcal{M}_{\lambda^{(1)}} \otimes \mathcal{M}_{\lambda^{(2)}} \longrightarrow \extp^{\frac{n_1n_2}{2}} \frak{u}_P^* \otimes \mathcal{M}_\lambda
\]
be a nonzero $M_P(\C)$-equivariant homomorphism.
It determines an $M_Q(\C)$-equivariant homomorphism $\jmath = \jmath(\imath)$ by the following diagram
\begin{equation}\label{eq:diagram-1}
\begin{tikzcd}[row sep=normal, column sep=normal]
\mathcal{M}_{\lambda^{(1)}} \otimes \mathcal{M}_{\lambda^{(2)}} \arrow[r, "\imath"] & \extp^{\frac{n_1n_2}{2}} \frak{u}_P^* \otimes \mathcal{M}_\lambda \arrow[d, "\mathscr{C}_P\otimes{\rm id}"] \\
\mathcal{M}_{\lambda^{(1)}+n_2} \otimes \mathcal{M}_{\lambda^{(2)}-n_1} \arrow[u, "{\rm id}\otimes{\rm id}"]  & \extp^{\frac{n_1n_2}{2}} \frak{u}_{{}^t\!P}^* \otimes \mathcal{M}_\lambda \arrow[d, "{\rm Ad}(\dot{w}_P^{-1})^*\otimes{\rho_\lambda(\dot{w}_P)}"]\\
\mathcal{M}_{\lambda^{(2)}-n_1} \otimes \mathcal{M}_{\lambda^{(1)}+n_2}\arrow[r, "\jmath"]\arrow[u, "{\rm switch}"]  & \extp^{\frac{n_1n_2}{2}} \frak{u}_Q^* \otimes \mathcal{M}_\lambda.
\end{tikzcd}
\end{equation}
In the diagram, ``switch'' denotes the natural flip of tensor factors, and the arrows labelled with ``\({\rm id}\)'' are the canonical identifications of the underlying vector spaces induced by the determinant twists.
Here \(\mathscr C_P\) is the following contraction map. Let
\[
\Omega_P:=
\extp_{\substack{1\le i\le n_1\\ 1\le j\le n_2}}
X_{i,n_1+j}
\in
\extp^{n_1n_2}\frak u_P,
\]
where the exterior product is taken in lexicographic order. For any
subset
\[
I\subset \{1,...,n_1\}\times \{1,...,n_2\}
\]
of cardinality \(\frac{n_1n_2}{2}\), define
\[
X_I:=
\extp_{(i,j)\in I}X_{i,n_1+j} \in \extp^{\frac{n_1n_2}{2}}\frak u_P,
\quad
Y_I:=
\extp_{(i,j)\in I^c}X_{n_1+j,i} \in \extp^{\frac{n_1n_2}{2}}\frak u_{{}^t\!P},
\]
where both exterior products are taken in lexicographic order. We define
\[
\mathscr C_P:
\extp^{\frac{n_1n_2}{2}}\frak u_P^*
\longrightarrow
\extp^{\frac{n_1n_2}{2}}\frak u_{{}^t\!P}^*
\]
by
\[
\mathscr C_P(X_I^*)={\rm sgn}(I)Y_I^*,
\]
where \({\rm sgn}(I)\in\{\pm1\}\) is determined by
\[
X_I\wedge X_{I^c}={\rm sgn}(I)\Omega_P.
\]
This map is compatible with the determinant twist
\(\det^{\,n_2}\times \det^{-n_1}\) on $M_P$, which accounts for the shifts
\(\lambda^{(1)}\mapsto \lambda^{(1)}+n_2\) and
\(\lambda^{(2)}\mapsto \lambda^{(2)}-n_1\) in the diagram.

The following period relation is an important consequence of Shahidi's formula and Weselmann's result in \cite[Chapter 9]{HR2020} (see also \cite[Lemma 2.12]{Chen2026}).

\begin{proposition}\label{prop:Weselmann}
Assume the cohomological condition (\ref{eq:cohomological-condition}) is satisfied.
Then there exists $\eta \in \{\pm1\}$ depending only on $n_1$ and $n_2$ such that 
\begin{align*}
&{\mathbb W}_{\psi_{N_0}}^P\circ \left(I^P_{\imath}\right)^{-1}\left(Z\otimes[\pi_1\otimes\delta_{P,1}]^{\varepsilon} \otimes [\pi_2\otimes\delta_{P,2}]^{\varepsilon}\right)\\
& = \omega_{\pi_2}(-1)^{n_1}\eta(\sqrt{-1})^{\frac{n_1n_2}{2}}\varepsilon(0,\pi_1\times\pi_2^\vee,\psi_\R)\\
&\quad\times{\mathbb W}_{\psi_{N_0}}^Q\circ \left(I^Q_{\jmath}\right)^{-1}\left({\rm Ad}(\dot{w}_P)Z\otimes[\pi_2\otimes\delta_{Q,1}]^{\varepsilon} \otimes [\pi_1\otimes\delta_{Q,2}]^{\varepsilon}\right)
\end{align*}
for any $Z \in \extp^{d(\bullet)}\frak{a}_P/\frak{a}_G$.
\end{proposition}

\begin{proof}
Recall that
\[
\gamma(s,\pi_1 \times \pi_2^\vee,\psi_\R) = \varepsilon(s,\pi_1\times\pi_2^\vee,\psi_\R)\cdot \frac{L(1-s,\pi_1^\vee \times \pi_2)}{L(s,\pi_1 \times \pi_2^\vee)}.
\]
By the result of Weselmann in \cite[\S\,9.6.12]{HR2020}, we have
\begin{align*}
&M^Q\circ \left(I^P_{\imath}\right)^{-1}\left(Z\otimes[\pi_1\otimes\delta_{P,1}]^{\varepsilon} \otimes [\pi_2\otimes\delta_{P,2}]^{\varepsilon}\right)\\
& = \eta(\sqrt{-1})^{\frac{n_1n_2}{2}}\cdot \frac{L(0,\pi_1 \times \pi_2^\vee)}{L(1,\pi_1 \times \pi_2^\vee)}\cdot \left(I^Q_{\jmath}\right)^{-1}\left({\rm Ad}(\dot{w}_P)Z\otimes[\pi_2\otimes\delta_{Q,1}]^{\varepsilon} \otimes [\pi_1\otimes\delta_{Q,2}]^{\varepsilon}\right)
\end{align*}
for some $\eta$ depending only on $n_1$ and $n_2$. 
Note that the factor $\eta(\sqrt{-1})^{\frac{n_1n_2}{2}}$ is computed in \cite[Lemma 9.19]{HR2020}.
Since $\pi_1$ and $\pi_2$ have the same exponents, we have
\[
L(s,\pi_1 \times \pi_2^\vee) = L(s,\pi_1^\vee \times \pi_2).
\]
The assertion then follows immediately from (\ref{eq:Shahidi}).
\end{proof}

\subsubsection{Equivariant homomorphisms}\label{sec:Equivariant homomorphisms}

Keep the notation and assumptions of \S\,\ref{sec:Delorme}.
Assume that the cohomological condition \eqref{eq:cohomological-condition} is satisfied.
Let 
\[
\imath : \mathcal{M}_{\lambda^{(1)}} \otimes \mathcal{M}_{\lambda^{(2)}} \longrightarrow \extp^{\frac{n_1n_2}{2}} \frak{u}_P^* \otimes \mathcal{M}_\lambda,\quad \jmath : \mathcal{M}_{\lambda^{(2)}-n_1} \otimes \mathcal{M}_{\lambda^{(1)}+n_2} \longrightarrow \extp^{\frac{n_1n_2}{2}} \frak{u}_Q^* \otimes \mathcal{M}_\lambda
\]
be nonzero $M_P(\C)$-equivariant and $M_Q(\C)$-equivariant homomorphisms respectively.
They determine $M_P(\C)$-equivariant homomorphism $\imath_1$ and $M_Q(\C)$-equivariant homomorphisms $\imath_2$, $\tilde\jmath$ by the following diagrams
\begin{equation}\label{eq:diagram-2}
\begin{tikzcd}[row sep=normal, column sep=normal]
\mathcal{M}_{\lambda^{(1)}} \otimes \mathcal{M}_{\lambda^{(2)}} \arrow[r, "\imath"] & \extp^{\frac{n_1n_2}{2}} \frak{u}_P^* \otimes \mathcal{M}_\lambda \arrow[d, "(d\vartheta)^*\otimes{\rm id}"] \\
{}^{\vartheta} \mathcal{M}_{\lambda^{(1)}} \otimes {}^{\vartheta} \mathcal{M}_{\lambda^{(2)}} \arrow[r, "\imath_1"]\arrow[u, "{\rm id} \otimes {\rm id}"] & \extp^{\frac{n_1n_2}{2}} \frak{u}_{{}^t\!P}^* \otimes {}^{\vartheta}\mathcal{M}_\lambda \arrow[dd, "{\rm Ad}(\dot{w}_P^{-1})^*\otimes (\rho_{\lambda-{\sf w}}(\dot{w}_P)\circ\varphi_\lambda^{-1})"] \\
\mathcal{M}_{\lambda^{(1)}+n_2-{\sf w}} \otimes \mathcal{M}_{\lambda^{(2)}-n_1-{\sf w}}\arrow[u, "\varphi_{\lambda^{(1)}} \otimes \varphi_{\lambda^{(2)}}"]
\\
\mathcal{M}_{\lambda^{(2)}-n_1-{\sf w}} \otimes \mathcal{M}_{\lambda^{(1)}+n_2-{\sf w}}\arrow[r, "\imath_2"]\arrow[u, "{\rm switch}"]  & \extp^{\frac{n_1n_2}{2}} \frak{u}_{Q}^* \otimes \mathcal{M}_{\lambda-{\sf w}}
\end{tikzcd}
\end{equation}
and
\begin{equation}\label{eq:diagram-3}
\begin{tikzcd}
\mathcal{M}_{\lambda^{(2)}-n_1} \otimes \mathcal{M}_{\lambda^{(1)}+n_2}\arrow[r, "\jmath"] & \extp^{\frac{n_1n_2}{2}} \frak{u}_{Q}^* \otimes \mathcal{M}_{\lambda}\arrow[d, "{\rm id}\otimes {\rm id}"]\\
\mathcal{M}_{\lambda^{(2)}-n_1-{\sf w}} \otimes \mathcal{M}_{\lambda^{(1)}+n_2-{\sf w}}\arrow[r, "\tilde\jmath"]\arrow[u, "{\rm id}\otimes{\rm id}"]  & \extp^{\frac{n_1n_2}{2}} \frak{u}_{Q}^* \otimes \mathcal{M}_{\lambda-{\sf w}}.
\end{tikzcd}
\end{equation}
Here ${\sf w}$ is the purity weight of $\lambda$.

The main result of this subsection is the following key proposition on the compatibility of the isomorphisms in Delorme's lemma with the \(\vartheta\)-twist and the determinant twist.

\begin{proposition}\label{prop:Delorme-theta-det}
We have
\begin{align}\label{eq:key-1}
\begin{split}
&(L_{\omega_n}\otimes \varphi_\lambda^{-1})\circ {}^\vartheta\left(\mathbb{W}_{\psi_{N_0}}^P\circ (I^P_{\imath})^{-1}(Z\otimes[\pi_1\otimes\delta_{P,1}]^\varepsilon \otimes [\pi_2\otimes\delta_{P,2}]^\varepsilon)\right)\\
&= \omega_{\pi_2}(-1)^{n_1}\cdot
\mathbb{W}_{\psi_{N_0}}^Q\circ(I^Q_{\imath_2})^{-1}
\left(
\begin{aligned}
&({\rm Ad}(\dot{w}_P)\circ d\vartheta) Z \\
&\otimes
 (L_{\omega_{n_2}}\otimes \varphi_{\lambda^{(2)}}^{-1})
 \circ {}^\vartheta[\pi_2\otimes\delta_{P,2}]^\varepsilon \\
&\otimes
 (L_{\omega_{n_1}}\otimes \varphi_{\lambda^{(1)}}^{-1})
 \circ {}^\vartheta[\pi_1\otimes\delta_{P,1}]^\varepsilon
\end{aligned}
\right)
\end{split}
\end{align}
and
\begin{align}\label{eq:key-2}
\begin{split}
&({\det}^{\sf w}\otimes {\rm id})\circ
\left(
\mathbb{W}_{\psi_{N_0}}^Q\circ (I^Q_{\jmath})^{-1}
\left(
Z'\otimes[\pi_2\otimes\delta_{Q,1}]^\varepsilon
\otimes [\pi_1\otimes\delta_{Q,2}]^\varepsilon
\right)
\right) \\
&=
\mathbb{W}_{\psi_{N_0}}^Q\circ(I^Q_{\tilde\jmath})^{-1}
\left(
\begin{aligned}
& Z' \\
&\otimes
({\det}^{\sf w}\otimes {\rm id})\circ
[\pi_2\otimes\delta_{Q,1}]^\varepsilon \\
&\otimes
({\det}^{\sf w}\otimes {\rm id})\circ
[\pi_1\otimes\delta_{Q,2}]^\varepsilon
\end{aligned}
\right)
\end{split}
\end{align}
for any $Z \in \extp^{d(\bullet)}\frak{a}_P/\frak{a}_G$ and $Z' \in \extp^{d(\bullet)}\frak{a}_Q/\frak{a}_G$.
\end{proposition}

\begin{proof}
(\ref{eq:key-1}) follows from (1) and (2) of Lemmas \ref{lem:Delorme-theta-det-1} and \ref{lem:Delorme-theta-det-2} below.
Similarly, (\ref{eq:key-2}) follows from (3) of Lemmas \ref{lem:Delorme-theta-det-1} and \ref{lem:Delorme-theta-det-2}.
\end{proof}

\begin{lemma}\label{lem:Delorme-theta-det-1}
Let $\xi$ be a non-degenerate character of $N_0(\R)$.
\\
\emph{(1)} Define
\[
\mathscr T_\vartheta : {\rm Ind}_{P(\mathbb R)}^{\GL_n(\mathbb R)}
\left(
\mathcal W(\pi_1,\xi_1)\otimes \mathcal W(\pi_2,\xi_2)
\right) \longrightarrow {\rm Ind}_{{}^t\!P(\mathbb R)}^{\GL_n(\mathbb R)}
\left(
\mathcal W(\pi_1^\vee,{}^\vartheta\xi_1)\otimes \mathcal W(\pi_2^\vee,{}^\vartheta\xi_2)
\right)
\]
by 
\[
\mathscr T_\vartheta f(g):= (\vartheta_\W \otimes \vartheta_\W)\circ f({}^\vartheta g),\quad g \in \GL_n(\R).
\]
Then we have
\[
\vartheta_\W\circ \mathbb{W}_\xi^P = \mathbb{W}_{{}^\vartheta\xi}^{{}^t\!P}\circ\mathscr T_\vartheta.
\]
\\
\emph{(2)} Define
\[
\mathscr{L}_{\omega_n} : {\rm Ind}_{{}^t\!P(\mathbb R)}^{\GL_n(\mathbb R)}\left(
\mathcal W(\pi_1^\vee,\xi_1)\otimes \mathcal W(\pi_2^\vee,\xi_2) 
\right) \longrightarrow {\rm Ind}_{Q(\mathbb R)}^{\GL_n(\mathbb R)}\left(
\mathcal W(\pi_2^\vee,{}^{\omega_{n_2}}\xi_2)\otimes \mathcal W(\pi_1^\vee,{}^{\omega_{n_1}}\xi_1) 
\right) 
\]
by
\[
\mathscr{L}_{\omega_n} f(g):= {\rm switch}\circ(L_{\omega_{n_1}} \otimes L_{\omega_{n_2}})\circ f(\dot{w}_P^{-1} g),\quad g \in \GL_n(\R).
\]
Here $``{\rm switch}"$ denotes the natural flip of the two tensor factors.
Then we have
\[
L_{\omega_n}\circ \mathbb{W}_\xi^{{}^t\!P} = \omega_{\pi_2}(-1)^{n_1}\cdot \mathbb{W}_{{}^{\omega_n}\xi}^Q\circ \mathscr{L}_{\omega_n}.
\]
\\
\emph{(3)} Define
\[
\mathscr{D}_{\sf w} : {\rm Ind}_{Q(\mathbb R)}^{\GL_n(\mathbb R)}\left(
\mathcal W(\pi_2,\xi_2)\otimes \mathcal W(\pi_1,\xi_1)\right)  \longrightarrow {\rm Ind}_{Q(\mathbb R)}^{\GL_n(\mathbb R)}\left(
\mathcal W(\pi_2^\vee,\xi_2)\otimes \mathcal W(\pi_1^\vee,\xi_1) 
\right)
\]
by
\[
\mathscr{D}_{\sf w}f(g)
:=
(\det g)^{\sf w}\cdot
\bigl(f(g)\otimes({\det}^{\sf w}\times{\det}^{\sf w})\bigr),\quad g \in \GL_n(\R)
\]
where, for \(f(g)=W_2\otimes W_1\), we set
\[
f(g)\otimes({\det}^{\sf w}\times{\det}^{\sf w})
:=
(W_2\otimes{\det}^{\sf w})\otimes(W_1\otimes{\det}^{\sf w}).
\]
Then we have
\[
\mathbb{W}_{\xi}^Q \otimes {\det}^{\sf w} = \mathbb{W}_{\xi}^Q \circ \mathscr{D}_{\sf w}.
\]
\end{lemma}

\begin{proof}
We verify (2) and leave the details for (1) and (3) to the readers. 
Let
\[
f \in {\rm Ind}_{{}^t\!P(\mathbb R)}^{\GL_n(\mathbb R)}\left(
\mathcal W(\pi_1^\vee,\xi_1)\otimes \mathcal W(\pi_2^\vee,\xi_2) 
\right).
\]
Then
\begin{align*}
&(L_{\omega_n}\circ \mathbb{W}_\xi^{{}^t\!P})(g;f_{\underline s})\\
& = \int_{U_{{}^t\!Q}(\R)} f_{\underline{s}}(\dot{w}_P^{-1}u\omega_ng)({\bf 1}_{n_1},{\bf 1}_{n_2})\overline{\xi(u)}\,du\\
& = \int_{U_{{}^t\!Q}(\R)} f_{\underline{s}}({\rm diag}(\omega_{n_1}^{-1},(-1)^{n_1}\omega_{n_2}^{-1})\dot{w}_P^{-1}u\omega_ng)(\omega_{n_1},(-1)^{n_1}\omega_{n_2})\overline{\xi(u)}\,du\\
& = \int_{U_{{}^t\!Q}(\R)} f_{\underline{s}}(\dot{w}_P^{-1}\dot{w}_Q^{-1}\cdot \omega_n^{-1}u\omega_n\cdot g)(\omega_{n_1},(-1)^{n_1}\omega_{n_2})\overline{\xi(u)}\,du\\
& = \omega_{\pi_2}(-1)^{n_1}\cdot \int_{U_P(\R)} \left((L_{\omega_{n_1}} \otimes L_{\omega_{n_2}})\circ f_{\underline s}(\dot{w}_P^{-1}\dot{w}_Q^{-1}u g)\right)({\bf 1}_{n_1},{\bf 1}_{n_2})\,\overline{{}^{\omega_n}\xi(u)}\,du\\
& = \omega_{\pi_2}(-1)^{n_1}\cdot\int_{U_{P}(\R)} \mathscr{L}_{\omega_n} f_{\underline{s}}(\dot{w}_Q^{-1}ug)({\bf 1}_{n_2},{\bf 1}_{n_1})\overline{{}^{\omega_n}\xi(u)}\,du\\
& = \omega_{\pi_2}(-1)^{n_1}\cdot \mathbb{W}_{{}^{\omega_n}\xi}^Q(g;\mathscr{L}_{\omega_n} f_{\underline{s}}).
\end{align*}
This completes the proof.
\end{proof}

\begin{lemma}\label{lem:Delorme-theta-det-2}
Let $Z \in \extp^{d(\bullet)}\frak{a}_P/\frak{a}_G$ and $Z' \in \extp^{d(\bullet)}\frak{a}_Q/\frak{a}_G$.
\\
\emph{(1)} We have
\begin{align*}
&(\mathscr T_\vartheta \otimes {\rm id})\circ \left((I^P_{\imath})^{-1}(Z\otimes[\pi_1\otimes\delta_{P,1}]^\varepsilon \otimes [\pi_2\otimes\delta_{P,2}]^\varepsilon)\right) \circ \extp^\bullet d\vartheta \\ 
& =  (I_{\imath_1}^{{}^t\!P})^{-1}\left( d\vartheta Z \otimes {}^\vartheta[\pi_1\otimes\delta_{P,1}]^\varepsilon\otimes {}^\vartheta[\pi_2\otimes\delta_{P,2}]^\varepsilon\right).
\end{align*}
\\
\emph{(2)} We have
\begin{align*}
&(\mathscr L_{\omega_n} \otimes \varphi_\lambda^{-1})\circ \left((I_{\imath_1}^{{}^t\!P})^{-1}(Z\otimes{}^\vartheta[\pi_1\otimes\delta_{P,1}]^\varepsilon \otimes {}^\vartheta[\pi_2\otimes\delta_{P,2}]^\varepsilon)\right) \\ 
&=
(I_{\imath_2}^{Q})^{-1}
\left(
\begin{aligned}
&{\rm Ad}(\dot{w}_P)Z \\
&\otimes
(L_{\omega_{n_2}}\otimes \varphi_{\lambda^{(2)}}^{-1})\circ
{}^\vartheta[\pi_2\otimes\delta_{P,2}]^\varepsilon \\
&\otimes
(L_{\omega_{n_1}}\otimes \varphi_{\lambda^{(1)}}^{-1})\circ
{}^\vartheta[\pi_1\otimes\delta_{P,1}]^\varepsilon
\end{aligned}
\right).
\end{align*}
\\
\emph{(3)} We have
\begin{align*}
&(\mathscr D_{\sf w} \otimes {\rm id})\circ \left((I_{\jmath}^Q)^{-1}(Z'\otimes[\pi_2\otimes\delta_{Q,1}]^\varepsilon \otimes [\pi_1\otimes\delta_{Q,2}]^\varepsilon)\right) \\ 
&=
(I_{\tilde\jmath}^{Q})^{-1}
\left(
\begin{aligned}
&Z' \\
&\otimes
({\det}^{\sf w}\otimes {\rm id})\circ
[\pi_2\otimes\delta_{Q,1}]^\varepsilon \\
&\otimes
({\det}^{\sf w}\otimes {\rm id})\circ
[\pi_1\otimes\delta_{Q,2}]^\varepsilon
\end{aligned}
\right).
\end{align*}
\end{lemma}

\begin{proof}
Recall that
\[
H^\bullet(\frak{g},K^\circ;V_\pi\otimes\M_\lambda)
=
{\rm Hom}_{K^\circ}\left(
\extp^\bullet \frg/\frk,
V_\pi\otimes\M_\lambda
\right).
\]
In the verifications below, it will be convenient to move the coefficient
system \(\M_\lambda\) to the source, using the canonical isomorphism
\[
{\rm Hom}_\C\left(
\extp^\bullet \frg/\frk,
V_\pi\otimes\M_\lambda
\right)
\cong
{\rm Hom}_\C\left(
\extp^\bullet \frg/\frk\otimes \M_\lambda^\vee,
V_\pi
\right).
\]
We shall use the analogous identifications for \(\pi_1\) and \(\pi_2\).
With this convention, Delorme's lemma takes the more transparent form stated in Lemma \ref{lem:Delorme}.

Write $F_i = [\pi_i \otimes \delta_{P,i}]^\varepsilon$ for $i=1,2$ and
\[
F:= (I^P_{\imath})^{-1}(Z\otimes F_1 \otimes F_2).
\]
Let
\[
X_i \in \extp^{\bullet_i}\frak{g}_i/\frak{k}_i,\quad m_i^\vee \in {}^\vartheta \M_{\lambda^{(i)}}^\vee,\quad i=1,2.
\]
By the definition of $\mathscr{T}_\vartheta$, $\imath_1$ and Lemma \ref{lem:Delorme}, we have
\begin{align*}
&\left(\mathscr{T}_\vartheta \circ F \circ (\extp^\bullet d\vartheta \otimes {\rm id}) \right) (d\vartheta Z\otimes X_1 \otimes X_2 \otimes \imath_1^*(m_1^\vee \otimes m_2^\vee))(1)\\
& = (\vartheta_\W \otimes \vartheta_\W)\circ F \left( Z \otimes \extp^{\bullet_1} d\vartheta X_1 \otimes \extp^{\bullet_2} d\vartheta X_2 \otimes (d\vartheta\otimes{\rm id})\circ\imath_1^*(m_1^\vee \otimes m_2^\vee)\right)(1)\\
& = (\vartheta_\W \otimes \vartheta_\W)\circ F \left( Z \otimes \extp^{\bullet_1} d\vartheta X_1 \otimes \extp^{\bullet_2} d\vartheta X_2 \otimes \imath^*(m_1^\vee \otimes m_2^\vee)\right)(1)\\
& = (\vartheta_\W \otimes \vartheta_\W)\circ \left( F_1(\extp^{\bullet_1} d\vartheta X_1 \otimes m_1^\vee) \otimes F_2(\extp^{\bullet_2} d\vartheta X_2 \otimes m_2^\vee) \right)\\
& = {}^\vartheta F_1(X_1 \otimes m_1^\vee)\otimes {}^\vartheta F_2(X_2 \otimes m_2^\vee).
\end{align*}
By Lemma \ref{lem:Delorme} again, we then have
\[
I_{\imath_1}^{{}^t\!P}\left(\mathscr{T}_\vartheta \circ F \circ (\extp^\bullet d\vartheta \otimes {\rm id}) \right) = d\vartheta Z\otimes {}^\vartheta F_1 \otimes {}^\vartheta F_2.
\]
This proves (1).

Write $G_i = {}^\vartheta[\pi_i\otimes\delta_{P,i}]^\varepsilon$ for $i=1,2$ and
\[
G:= (I_{\imath_1}^{{}^t\!P})^{-1}(Z\otimes G_1 \otimes G_2).
\]
Associated to $\varphi_\lambda$, we have a $\GL_n(\C)$-equivariant isomorphism between the contragredient representations 
\[
\Psi_\lambda:=(\varphi_\lambda^\vee)^{-1} :  ({}^\vartheta\rho_{\lambda}^\vee, {}^\vartheta\mathcal{M}_{\lambda}^\vee)\longrightarrow (\rho_{\lambda-{\sf w}}^\vee, \mathcal{M}_{\lambda-{\sf w}}^\vee).
\]
Similarly we define $\Psi_{\lambda^{(i)}}:=(\varphi_{\lambda^{(i)}}^\vee)^{-1}$ for $i=1,2$.
Let
\[
X_i \in \extp^{\bullet_i}\frak{g}_i/\frak{k}_i,\quad m_1^\vee \in \M_{\lambda^{(1)}+n_2-{\sf w}}^\vee,\quad m_2^\vee \in \M_{\lambda^{(2)}-n_1-{\sf w}}^\vee.
\]
By the definition of $\mathscr{L}_{\omega_n}$, $\imath_2$ and Lemma \ref{lem:Delorme}, we have
\begin{align*}
&\left(\mathscr{L}_{\omega_n} \circ G \circ ({\rm id}\otimes \Psi_\lambda) \right) ({\rm Ad}(\dot{w}_P)Z\otimes X_2 \otimes X_1 \otimes \imath_2^*(m_2^\vee \otimes m_1^\vee))(1)\\
& = \left(\mathscr{L}_{\omega_n}\circ G \left( {\rm Ad}(\dot{w}_P)Z \otimes X_2 \otimes  X_1 \otimes ({\rm Ad}(\dot{w}_P)\otimes {}^\vartheta\rho_{\lambda}^\vee(\dot{w}_P))\circ\imath_1^*(\Psi_{\lambda^{(1)}}m_1^\vee \otimes \Psi_{\lambda^{(2)}}m_2^\vee)\right)\right)(1)\\
& = {\rm switch}\circ(L_{\omega_{n_1}} \otimes L_{\omega_{n_2}})\\
&\quad\circ G \left(  {\rm Ad}(\dot{w}_P)Z \otimes X_2 \otimes  X_1 \otimes ({\rm Ad}(\dot{w}_P)\otimes {}^\vartheta\rho_{\lambda}^\vee(\dot{w}_P))\circ\imath_1^*(\Psi_{\lambda^{(1)}}m_1^\vee \otimes \Psi_{\lambda^{(2)}}m_2^\vee)\right)(\dot{w}_P^{-1})\\
& = {\rm switch}\circ(L_{\omega_{n_1}} \otimes L_{\omega_{n_2}})\circ (\dot{w}_P^{-1}\cdot G)\left(Z \otimes X_1 \otimes X_2 \otimes \imath_1^*(\Psi_{\lambda^{(1)}}m_1^\vee \otimes \Psi_{\lambda^{(2)}}m_2^\vee) \right)(1)\\
& = L_{\omega_{n_2}}\circ G_2 \circ ({\rm id}\otimes \Psi_{\lambda^{(2)}}) (X_2 \otimes m_2^\vee) \otimes L_{\omega_{n_1}}\circ G_1 \circ ({\rm id}\otimes \Psi_{\lambda^{(1)}}) (X_1 \otimes m_1^\vee).
\end{align*}
Note that $n_1n_2$ is even, which implies that $\dot{w}_P \in K^\circ$ and hence $\dot{w}_P^{-1}\cdot G = G$.
By Lemma \ref{lem:Delorme} again, we then have
\begin{align*}
& I_{\imath_2}^Q\left( \mathscr{L}_{\omega_n} \circ G \circ ({\rm id}\otimes \Psi_\lambda)\right)\\
& = {\rm Ad}(\dot{w}_P)Z \otimes L_{\omega_{n_2}}\circ G_2 \circ ({\rm id}\otimes \Psi_{\lambda^{(2)}}) \otimes L_{\omega_{n_1}}\circ G_1 \circ ({\rm id}\otimes \Psi_{\lambda^{(1)}}).
\end{align*}
This proves (2).

Write $H_1 = [\pi_1 \otimes \delta_{Q,2}]^\varepsilon$,  $H_2 = [\pi_2 \otimes \delta_{Q,1}]^\varepsilon$, and
\[
H:= (I_{\jmath}^{Q})^{-1}(Z'\otimes H_2 \otimes H_1).
\]
Let $X_i$, $m_i^\vee$ as in the previous paragraph for $i=1,2$.
By the definition of $\mathscr{D}_{\sf w}$, $\tilde{\jmath}$ and Lemma \ref{lem:Delorme}, we have
\begin{align*}
& (\mathscr{D}_{\sf w}\circ H)\left( Z' \otimes X_2 \otimes X_1 \otimes (\tilde{\jmath})^*(m_2^\vee \otimes m_1^\vee) \right)(1)\\
& = H\left( Z' \otimes X_2 \otimes X_1 \otimes {\jmath}^*(m_2^\vee \otimes m_1^\vee) \right)(1) \otimes ({\det}^{\sf w} \times {\det}^{\sf w})\\
& = \left(H_2(X_2 \otimes m_2^\vee) \otimes {\det}^{\sf w}\right) \otimes \left(H_1(X_1 \otimes m_1^\vee) \otimes {\det}^{\sf w}\right).
\end{align*}
By Lemma \ref{lem:Delorme} again, we then have
\[
I_{\tilde\jmath}^Q(\mathscr{D}_{\sf w}\circ H) = Z'\otimes ({\det}^{\sf w}\circ H_2) \otimes ({\det}^{\sf w}\circ H_1).
\]
This proves (3).
\end{proof}

\subsubsection{Proof of Theorem \ref{thm:archi-period-relation-2}}

Let $\pi \in \Omega(\lambda)$ for some pure weight 
$\lambda \in X_0^+(T_0(n))$, and let ${\sf w}$ be the purity weight of $\lambda$.
We denote by $\mathbf{c}^\varepsilon(\pi)$ the scalar such that
\[
\widetilde{[\pi]}^\varepsilon
=
\mathbf{c}^\varepsilon(\pi) \cdot
(L_{\omega_n}\otimes \varphi_\lambda^{-1})\circ {}^\vartheta[\pi]^\varepsilon.
\]
It is clear that this scalar is independent of the choice of $[\pi]^\varepsilon$.
In this subsection, we prove Theorem \ref{thm:archi-period-relation-2}, which asserts that
$\mathbf{c}^\varepsilon(\pi)\cdot\varepsilon(-1)^n$ belongs to $\{\pm1\}$ and depends only on $n$.

By the explicit description of $\pi$ as an induced representation, for any partition
$(n_1,n_2)$ of $n$ with $n_1n_2$ even, there exist $\pi_1$ and $\pi_2$ such that
\[
\pi \cong {\rm Ind}_{P(\R)}^{\GL_n(\R)}(\pi_1\otimes\pi_2),
\]
and such that $\pi_i \otimes \delta_{P,i} \in \Omega(\lambda^{(i)})$ for some pure weight
$\lambda^{(i)} \in X_0^+(T_0(n_i))$ for $i=1,2$.
Here $P$ is the upper triangular maximal parabolic subgroup of $\GL_n$ such that
$M_P \cong \GL_{n_1} \times \GL_{n_2}$.
In other words, the cohomological condition \eqref{eq:cohomological-condition} is satisfied.
Fix a nonzero $M_P(\C)$-equivariant homomorphism
\[
\imath : \mathcal{M}_{\lambda^{(1)}} \otimes \mathcal{M}_{\lambda^{(2)}}
\longrightarrow
\extp^{\frac{n_1n_2}{2}} \frak{u}_P^* \otimes \mathcal{M}_\lambda.
\]
Let $\jmath=\jmath(\imath)$ be the $M_Q(\C)$-equivariant homomorphism associated with
$\imath$ via the diagram \eqref{eq:diagram-1}.
Fix a generator $[\pi]^\varepsilon$ and an element
$Z \in \extp^{d(\bullet)}\frak{a}_P/\frak{a}_G$. We choose generators
$[\pi_1\otimes \delta_{P,1}]^\varepsilon$ and
$[\pi_2\otimes \delta_{P,2}]^\varepsilon$ such that
\[
I_\imath^P\circ (\mathbb{W}_{\psi_{N_0}}^P)^{-1}([\pi]^\varepsilon)
=
Z \otimes [\pi_1\otimes \delta_{P,1}]^\varepsilon
\otimes [\pi_2\otimes \delta_{P,2}]^\varepsilon.
\]
These choices determine generators
$[\pi_2\otimes \delta_{Q,1}]^\varepsilon$ and
$[\pi_1\otimes \delta_{Q,2}]^\varepsilon$ as in
(\ref{eq:det-twist-generators}).
Since $d\vartheta Z=-Z$, it follows from (\ref{eq:key-1}) that
\begin{align*}
&(L_{\omega_n}\otimes \varphi_\lambda^{-1})\circ {}^\vartheta[\pi]^\varepsilon \\
&= (-1)^{d(\bullet)}\omega_{\pi_2}(-1)^{n_1}\cdot
\mathbb{W}_{\psi_{N_0}}^Q\circ(I^Q_{\imath_2})^{-1}
\left(
\begin{aligned}
&{\rm Ad}(\dot{w}_P) Z \\
&\otimes
 (L_{\omega_{n_2}}\otimes \varphi_{\lambda^{(2)}}^{-1})
 \circ {}^\vartheta[\pi_2\otimes\delta_{P,2}]^\varepsilon \\
&\otimes
 (L_{\omega_{n_1}}\otimes \varphi_{\lambda^{(1)}}^{-1})
 \circ {}^\vartheta[\pi_1\otimes\delta_{P,1}]^\varepsilon
\end{aligned}
\right).
\end{align*}
On the other hand, by Proposition \ref{prop:Weselmann} and
(\ref{eq:key-2}) applied with $Z'={\rm Ad}(\dot{w}_P)Z$, we obtain
\begin{align*}
& \omega_{\pi_2}(-1)^{n_1}\eta(\sqrt{-1})^{-\frac{n_1n_2}{2}}
\varepsilon(0,\pi_1\times\pi_2^\vee,\psi_\R)^{-1}\cdot
\widetilde{[\pi]}^\varepsilon \\
&=
\mathbb{W}_{\psi_{N_0}}^Q\circ(I^Q_{\tilde\jmath})^{-1}
\left(
\begin{aligned}
& {\rm Ad}(\dot{w}_P) Z \\
&\otimes
({\det}^{\sf w}\otimes {\rm id})\circ
[\pi_2\otimes\delta_{Q,1}]^\varepsilon \\
&\otimes
({\det}^{\sf w}\otimes {\rm id})\circ
[\pi_1\otimes\delta_{Q,2}]^\varepsilon
\end{aligned}
\right)
\end{align*}
for some $\eta \in \{\pm1\}$ depending only on $n_1$ and $n_2$.
Here $\imath_2$ and $\tilde\jmath$ are determined by $\imath$ and
$\jmath$ through the diagrams (\ref{eq:diagram-2}) and
(\ref{eq:diagram-3}), respectively.
By the normalization in (\ref{eq:det-twist-generators}), we have
\[
({\det}^{\sf w}\otimes {\rm id})\circ
[\pi_2\otimes\delta_{Q,1}]^\varepsilon
=
({\det}^{{\sf w}+n_1}\otimes{\rm id})\circ
[\pi_2\otimes\delta_{P,2}]^\varepsilon
=
\widetilde{[\pi_2\otimes\delta_{P,2}]}^\varepsilon.
\]
Similarly,
\[
({\det}^{\sf w}\otimes {\rm id})\circ
[\pi_1\otimes\delta_{Q,2}]^\varepsilon
=
\widetilde{[\pi_1\otimes\delta_{P,1}]}^\varepsilon.
\]
Comparing the preceding two identities, we conclude that
\[
\mathbf{c}^\varepsilon(\pi)
=
(-1)^{d(\bullet)}
\eta(\sqrt{-1})^{\frac{n_1n_2}{2}}
\varepsilon(0,\pi_1\times\pi_2^\vee,\psi_\R)
\cdot C(\lambda^{(1)},\lambda^{(2)})\cdot
\mathbf{c}^\varepsilon(\pi_1)\mathbf{c}^\varepsilon(\pi_2),
\]
where $C(\lambda^{(1)},\lambda^{(2)})$ is defined by
\[
\imath_2=C(\lambda^{(1)},\lambda^{(2)})\cdot \tilde{\jmath}.
\]
In Lemma \ref{lem:base-cases} below, we prove
Theorem \ref{thm:archi-period-relation-2} for $n=1,2$ by explicit computations. 
Hence, by mathematical induction,
to complete the proof of Theorem \ref{thm:archi-period-relation-2}, it remains only to show that
\[
(\sqrt{-1})^{\frac{n_1n_2}{2}}
\varepsilon(0,\pi_1\times\pi_2^\vee,\psi_\R)\cdot
C(\lambda^{(1)},\lambda^{(2)})
\]
belongs to $\{\pm1\}$ and depends only on $n_1$ and $n_2$.
Since $\lambda^{(1)}$ and $\lambda^{(2)}$ may be chosen freely subject to condition (\ref{eq:cohomological-condition}), we simplify the computation in Lemma \ref{lem:core-sign} below by specializing to a particular choice with $n_2=1$ or $n_2=2$.

\begin{lemma}\label{lem:base-cases}
For $n=1,2$, we have
\[
\mathbf{c}^\varepsilon(\pi) = \varepsilon(-1)^n.
\]
\end{lemma}

\begin{proof}
Assume $n=1$. Then $\lambda = \tfrac{\sf w}{2}$ and $\pi$ is a character ${\rm sgn}^\delta |\mbox{ }|^{-\frac{\sf w}{2}}$ for some $\delta \in \{0,1\}$. 
Note that $\varepsilon(-1) = (-1)^{\delta+\frac{\sf w}{2}}$ and $\M_\lambda = \C$ with $\rho_\lambda$ acts by the $\tfrac{\sf w}{2}$-th power. 
Fix a nonzero Whittaker function $W_\pi \in \W(\pi)$.
Then $W_\pi(a) = \pi(a)W(1)$ for $a \in \R^\times$.
A generator is given by
\[
[\pi]^\varepsilon:= W_\pi \otimes 1 \in (\W(\pi)\otimes \M_\lambda)^{K_1^\circ}.
\]
Since $\omega_1=-1$, the normalization \eqref{eq:normalization} gives $\varphi_\lambda(1)=(-1)^{\frac{\sf w}{2}}$.
It is easy to see that
\[
L_{\omega_1}\circ{}^\vartheta W_\pi = (-1)^\delta\cdot W_\pi \otimes {\det}^{\sf w}.
\]
Thus
\[
\mathbf{c}^\varepsilon(\pi) = (-1)^{\delta}\cdot (-1)^{\frac{{\sf w}}{2}} = \varepsilon(-1).
\]

Assume $n=2$. Then $\pi \cong D_\kappa \otimes |\det|^{-\frac{\sf w}{2}}$ for some $\kappa \geq 2$ such that $\kappa \equiv {\sf w}\,({\rm mod}\,2)$, where $D_\kappa$ is the discrete series representation of $\GL_2(\R)$ with central character ${\rm sgn}^\kappa$ and minimal ${\rm SO}(2)$-type $\pm \kappa$.
We take $\M_\lambda$ to be the space of homogeneous polynomials in variables $x,y$ of degree $\kappa-2$, and $\rho_\lambda$ acts by
\[
\rho_\lambda(g)P(x,y) := (\det g)^{\frac{-\kappa+2+{\sf w}}{2}}\cdot P((x,y)g),\quad g \in \GL_2(\C).
\]
Let $W_\pi^\pm \in \W(\pi)$ be Whittaker functions of ${\rm SO}(2)$-weight $\pm \kappa$ such that
\[
W_\pi^+({\rm diag}(-a,1)) = W_\pi^-({\rm diag}(a,1)),\quad a \in \R^\times.
\]
A generator
\[
[\pi]^\varepsilon\in   \left((\frak{g}_{2}/\frak{k}_{2})^* \otimes \mathcal{W}(\pi) \otimes \M_\lambda\right)^{K_2^\circ}(\varepsilon)
\]
is given by
\[
[\pi]^\varepsilon := Y_+^*\otimes W_\pi^+ \otimes (\sqrt{-1}\,x+y)^{\kappa-2}+\varepsilon(-1)\cdot  (\sqrt{-1})^{-{\sf w}}\cdot Y_-^*\otimes W_{\pi}^- \otimes  (x+\sqrt{-1}\,y)^{\kappa-2},
\]
where 
\[
Y_\pm := \bp -\sqrt{-1} & \pm1\\ \pm1& \sqrt{-1}\ep.
\]
A direct computation gives
\[
(d\vartheta)^*Y_\pm^* = -Y_\pm^*,\quad L_{\omega_2}\circ {}^\vartheta W_\pi^\pm = (\sqrt{-1})^{\pm\kappa}\cdot W_\pi^\pm\otimes {\det}^{\sf w}.
\]
Indeed,
\begin{align*}
(L_{\omega_2}\circ {}^\vartheta W_\pi^\pm)({\rm diag}(a,1))
 & = W_\pi^\pm({\rm diag}(1,a^{-1})\omega_2)\\
 & = (\sqrt{-1})^{\pm\kappa}\cdot a^{\sf w}W_\pi^\pm({\rm diag}(a,1))\\
 & = (\sqrt{-1})^{\pm\kappa}\cdot (W_\pi^\pm\otimes {\det}^{\sf w})({\rm diag}(a,1)).
\end{align*}
By the normalization (\ref{eq:normalization}), we have $\varphi_\lambda P(x,y) = P(y,-x)$. 
It follows that
\begin{align*}
\varphi_\lambda(\sqrt{-1}\,x+y)^{\kappa-2} &= (\sqrt{-1})^{\kappa-2}(\sqrt{-1}\,x+y)^{\kappa-2},\\
 \varphi_\lambda(x+\sqrt{-1}\,y)^{\kappa-2} &= (\sqrt{-1})^{-\kappa+2}(x+\sqrt{-1}\,y)^{\kappa-2}.
\end{align*}
Hence
\[
\mathbf{c}^\varepsilon(\pi) = (-1)\cdot(\sqrt{-1})^{\mp\kappa}\cdot(\sqrt{-1})^{\pm\kappa\mp2} = 1.
\]
\end{proof}

\begin{lemma}\label{lem:core-sign}
Let
\[
(n_1,n_2) = \begin{cases}
(2r,1) & \mbox{ if $n=2r+1$},\\
(2r,2) & \mbox{ if $n=2r+2$}.
\end{cases}
\]
Let $w \in W_n^P$ be the balanced Kostant representative defined by
\[
w^{-1} :=
\begin{cases}
\bp
1 & 2 & \cdots & r & r+1 & r+2 & \cdots & 2r & 2r+1\\
1 & 2 & \cdots & r & r+2 & r+3 & \cdots & 2r+1 & r+1
\ep
& \mbox{ if } n=2r+1,\\[1.2em]
\bp
1 & 2 & \cdots & 2r & 2r+1 & 2r+2\\
2 & 3 & \cdots & 2r+1 & 1 & 2r+2
\ep
& \mbox{ if } n=2r+2.
\end{cases}
\]
Let $\lambda^{(1)}$ and $\lambda^{(2)}$ be the pure weights determined by $w \cdot \lambda = (\lambda^{(1)},\lambda^{(2)})$.
Then we have
\begin{align*}
(\sqrt{-1})^{\frac{n_1n_2}{2}}
\varepsilon(0,\pi_1\times\pi_2^\vee,\psi_\R)\cdot
C(\lambda^{(1)},\lambda^{(2)})=\begin{cases}
(-1)^{\frac{r(r-1)}{2}} & \mbox{ if $n=2r+1$},\\
1 & \mbox{ if $n=2r+2$}.
\end{cases}
\end{align*}
\end{lemma}

\begin{proof}
Fix highest weight vectors $m_{\lambda^{(i)}} \in \M_{\lambda^{(i)}}$ for $i=1,2$. Let $m_\lambda \in \M_\lambda$ be the highest weight vector such that
\[
\imath( m_{\lambda^{(1)}} \otimes m_{\lambda^{(2)}}) = v_w^P \otimes \rho_\lambda(\dot{w})m_\lambda.
\]
Please refer to \S\,\ref{sec:Kostant} for the notation.
Note that $w_Pw \in W_n^Q$ is also a balanced Kostant representative. 
By definition of $w$, we have 
\begin{align*}
S_w^P &= \begin{cases}
\left\{ (i,1)\,\vert\,r+1 \leq i \leq 2r\right\} & \mbox{ if $n=2r+1$},\\
\left\{ (i,1)\,\vert\,1 \leq i \leq 2r\right\} & \mbox{ if $n=2r+2$},
\end{cases}\\
S_{w_Pw}^Q &= \begin{cases}
\left\{ (1,j)\,\vert\,1 \leq j \leq r\right\} & \mbox{ if $n=2r+1$},\\
\left\{ (2,j)\,\vert\,1 \leq j \leq 2r\right\} & \mbox{ if $n=2r+2$},
\end{cases}
\end{align*}
and
\begin{align*}
\lambda^{(1)} &= \begin{cases}
(\lambda_1,...,\lambda_r,\lambda_{r+2}-1,...,\lambda_{2r+1}-1) & \mbox{ if $n = 2r+1$},\\
(\lambda_2-1,\lambda_3-1,...,\lambda_{2r+1}-1) & \mbox{ if $n=2r+2$},
\end{cases}\\
\lambda^{(2)} &= \begin{cases}
\lambda_{r+1}+r & \mbox{ if $n = 2r+1$},\\
(\lambda_1+2r,\lambda_{2r+2}) & \mbox{ if $n=2r+2$}.
\end{cases}
\end{align*}
We consider the evaluations of $\imath_2$ and $\tilde\jmath$ at $m_{\lambda^{(2)}} \otimes m_{\lambda^{(1)}}$. 
Both evaluations are scalar multiples of
\[
v_{w_Pw}^Q\otimes \rho_{\lambda-\mathsf w}(\dot w_P\dot w)m_\lambda.
\]
Then $C(\lambda^{(1)},\lambda^{(2)})$ is precisely the ratio of these two scalars.
By definition of $\imath_2$ and $\tilde\jmath$, we have
\begin{align*}
\imath_2(m_{\lambda^{(2)}} \otimes m_{\lambda^{(1)}}) & = ({\rm Ad}(\dot{w}_P^{-1})^*\circ (d\vartheta)^*\circ {\rm Ad}({\rm diag}(\omega_{n_1}, \omega_{n_2}))^*) v_w^P\\
& \quad\otimes \rho_{\lambda-{\sf w}}(\dot{w}_P{\rm diag}(\omega_{n_1}, \omega_{n_2})^{-1}\dot{w}\omega_n)m_\lambda,\\
\tilde{\jmath}(m_{\lambda^{(2)}} \otimes m_{\lambda^{(1)}}) & = ({\rm Ad}(\dot{w}_P^{-1})^*\circ \mathscr{C}_P) v_w^P \otimes \rho_{\lambda}(\dot{w}_P\dot{w})m_\lambda.
\end{align*}
By direct computations, we have
\begin{align*}
\left({\rm Ad}(\dot w_P^{-1})^*\circ (d\vartheta)^*
\circ {\rm Ad}({\rm diag}(\omega_{n_1},\omega_{n_2}))^*\right)v_w^P
&=
\begin{cases}
(-1)^rv_{w_Pw}^Q & \text{ if } n=2r+1,\\[0.4em]
 v_{w_Pw}^Q & \text{ if } n=2r+2,
\end{cases}\\
\left({\rm Ad}(\dot w_P^{-1})^*\circ \mathscr C_P\right)v_w^P
&=
(-1)^r v_{w_Pw}^Q,
\end{align*}
and 
\[
{\rm diag}(\omega_{n_1}, \omega_{n_2})^{-1}\dot{w}\omega_n = \dot{w}\cdot t_w,
\]
where
\[
t_w:=
\begin{cases}
{\rm diag}\left(
{\bf 1}_r,\,
(-1)^r,\,
-{\bf 1}_r
\right)
& \text{if } n=2r+1,\\[1.2em]
{\rm diag}\left(
1,\,
-{\bf 1}_{2r},\,
1
\right)
& \text{if } n=2r+2.
\end{cases}
\]
Also note that $(\det\dot{w}_P\dot{w})^{\sf w}$ is always equal to one. 
Therefore, we have
\begin{align*}
C(\lambda^{(1)},\lambda^{(2)}) = \begin{cases}
(-1)^{\frac{r{\sf w}}{2} + \sum_{j=r+2}^{2r+1}\lambda_j} & \mbox{ if $n=2r+1$},\\
(-1)^{r+r{\sf w}} & \mbox{ if $n=2r+2$}.
\end{cases}
\end{align*}
It remains to compute $\varepsilon(0,\pi_1\times\pi_2^\vee,\psi_\R)$. 
Let
\[
\kappa_i:= \lambda_i - \lambda_{n+1-i}+n+2-2i,\quad 1 \leq i \leq \lfloor \tfrac{n}{2} \rfloor.
\]
Then we have 
\[
\kappa_1 > \cdots > \kappa_{\lfloor \frac{n}{2} \rfloor} \geq 2.
\]
By our description for $\lambda^{(1)}$ and $\lambda^{(2)}$, we see that $\phi_{\kappa_1},...,\phi_{\kappa_r}$ (resp.\,$\phi_{\kappa_2},...,\phi_{\kappa_{r+1}}$) appear in the Langlands parameter of $\pi_1 \otimes |\det|^{\frac{\sf w}{2}}$ if $n=2r+1$ (resp.\,$n=2r+2$), and $\phi_{\kappa_1}$ is the Langlands parameter of $\pi_2 \otimes |\det|^{\frac{\sf w}{2}}$ if $n_2=2$. Here $\phi_\kappa$ denotes the Langlands parameter of the discrete series representation of $\GL_2(\R)$ with minimal ${\rm SO}(2)$-weight $\kappa \geq 2$ and central character ${\rm sgn}^\kappa$.
By the formula of archimedean $\varepsilon$-factors (cf.\,\cite[(3.7)]{Knapp1994}), we then have
\begin{align*}
(\sqrt{-1})^{\frac{n_1n_2}{2}}\varepsilon(0,\pi_1\times\pi_2^\vee,\psi_\R) &= \begin{cases}
(\sqrt{-1})^{r+\sum_{i=1}^r\kappa_i} & \mbox{ if $n=2r+1$},\\
(\sqrt{-1})^{2r+2r\kappa_1} & \mbox{ if $n=2r+2$},
\end{cases}\\
& = \begin{cases}
(-1)^{\frac{r{\sf w}}{2}+\frac{r(r+3)}{2} - \sum_{j=r+2}^{2r+1}\lambda_j} & \mbox{ if $n=2r+1$},\\
(-1)^{r+r{\sf w}} & \mbox{ if $n=2r+2$}.
\end{cases}
\end{align*}
This completes the proof.
\end{proof}

\subsection{Complex cases}
\label{sec:glnc}

In this subsection, let $v\in S_c$ be a complex place. Since $\GL_n(\C)$ is connected, there is no nontrivial component character to consider. We therefore suppress the superscript $\varepsilon$ and write
\[
[\pi]=[\pi]^\varepsilon,\quad c(\pi)=c^\varepsilon(\pi).
\]
Recall that $\pi\in \Omega(\lambda)$ for some pure weight
\[
\lambda=(\lambda^\tau,\lambda^{\overline{\tau}})
\in X_0^+(T_0(n))\times X_0^+(T_0(n)),
\]
and ${\sf w}$ is the purity weight of $\lambda$. We assume that $\lambda^\tau=\lambda^{\overline{\tau}}$,
or equivalently, $\lambda^\vee=\lambda-{\sf w}$.
We have define a representation of $\GL_n(\C)\times \GL_n(\C)$ by
\[
(\rho_\lambda,\mathcal M_\lambda)
=
\bigl(\rho_{\lambda^\tau}\otimes \rho_{\lambda^\tau},
\mathcal M_{\lambda^\tau}\otimes \mathcal M_{\lambda^\tau}\bigr).
\]
We have also fixed an equivariant isomorphism
\[
\Phi_\lambda:
(\rho_{\lambda-{\sf w}},\mathcal M_{\lambda-{\sf w}})
\longrightarrow
({}^\theta\rho_\lambda,{}^\theta\mathcal M_\lambda),
\]
normalized so that
\[
\Phi_\lambda(m_{\lambda})=m_\lambda
\]
for any highest weight vector $m_\lambda\in \mathcal M_\lambda$.
Then $c(\pi)$ is the scalar such that 
\[
({\rm id}\otimes \Phi_\lambda)\circ \widetilde{[\pi]} = c(\pi)\cdot {}^\theta[\pi].
\]

The main result of this subsection is the period relation stated in Theorem \ref{thm:archi-period-relation-3}. Unlike in the real cases, our argument here is more straightforward and eventually reduced to the explicit formulas of Ishii--Miyazaki \cite{IshiiMiyazaki2022} for highest weight Whittaker functions in the minimal $K$-type of principal series representations.
We note that analogous formulas for generalized principal series representations are not yet known in general.

\begin{lemma}\label{eq:equivariant-(-,+)}
Let $m_{\lambda^{\tau}} \in \M_{\lambda^\tau}$ be a highest weight vector. Then 
\[
\Phi_\lambda(\rho_\lambda(\omega_n)m_{\lambda^{\tau}} \otimes m_{\lambda^{\tau}}) = \rho_\lambda(\omega_n)m_{\lambda^{\tau}} \otimes m_{\lambda^{\tau}}.
\]
\end{lemma}

\begin{proof}
This is clear, since
\[
\rho_{\lambda-{\sf w}}(\omega_n) = (-1)^{n{\sf w}}\rho_\lambda(\omega_n),\quad {}^\theta\rho_\lambda(\omega_n) = \rho_\lambda(\omega_n).
\]
Also note that we always have $(-1)^{n{\sf w}}=1$.
\end{proof}

Let $\frak{g}$ and $\frak{k}$ denote the complexified
Lie algebras of $\GL_n(\C)$ and ${\rm U}(n)$, respectively, and let
\[
\frak{g}=\frak{k}\oplus\frak{p}
\]
be the Cartan decomposition. 
With respect to the natural adjoint action of the diagonal torus of ${\rm U}(n)$ on $\frak{p}$, the
weight-space decomposition of $\frak{p}$ is given by
\[
\frak{p}
=
\frak{a}
\oplus
\bigoplus_{\alpha\in\Phi^+}\C\cdot X_\alpha^p
\oplus
\bigoplus_{\beta\in\Phi^-}\C\cdot Y_\beta^p,
\]
where $\Phi^+$ denotes the set of positive roots, identified in the usual
way with pairs of indices $1\leq i<j\leq n$, and $\Phi^-=-\Phi^+$. More
explicitly, for $1\leq i<j\leq n$, we put
\begin{align*}
X_{ij}^p
&:=
(E_{ij}+E_{ji})\otimes \sqrt{-1}
+
\sqrt{-1}(E_{ij}-E_{ji})\otimes 1,\\
Y_{ij}^p
&:=
(E_{ij}+E_{ji})\otimes \sqrt{-1}
-
\sqrt{-1}(E_{ij}-E_{ji})\otimes 1.
\end{align*}
For $1\leq k\leq n-1$, we define
\[
Z_k:=(E_{kk}-E_{k+1,k+1})\otimes 1 \in \frak{a}.
\]
Define
\[
X_\bullet:=
\begin{cases}
\extp_{1\leq i<j\leq n} X_{ij}^p,
& \text{if } \bullet=b_n,\\
\extp_{1\leq k\leq n-1} Z_k
\wedge
\extp_{1\leq i<j\leq n} X_{ij}^p,
& \text{if } \bullet=t_n.
\end{cases}
\]
Here all exterior products are taken in lexicographic order, with the
elements $Z_k$ ordered before the elements $X_{ij}^p$. Then $X^+$ is a
highest weight vector in $\bigwedge^\bullet\frak{p}$ as a representation of
${\rm U}(n)$.

\begin{lemma}\label{eq:theta-X_+}
We have
\[
d\theta(X_\bullet) = \begin{cases}
X_\bullet & \mbox{ if $\bullet = b_n$},\\
(-1)^{\lfloor\frac{n-1}{2}\rfloor}X_\bullet& \mbox{ if $\bullet = t_n$}.
\end{cases}
\]
\end{lemma}
\begin{proof}
By direct computations, we have
\begin{align*}
d\theta(X_{b_n}) &= \extp_{1 \leq i < j \leq n}(-1)^{i+j+1}X_{n+1-j, n+1-i}\\
& = (-1)^{\sum_{1\leq i<j \leq n}(i+j+1)} {\rm sgn}(\sigma) X_{b_n},
\end{align*}
where $\sigma$ is the permutation on $\{(i,j): 1\leq i<j \leq n\}$ defined by
\[
 (i,j)\mapsto (n+1-j,n+1-i).
\] 
Note that 
\[
\sum_{1\leq i<j\leq n}(i+j+1)=\tfrac{n(n-1)(n+2)}{2},\quad {\rm sgn}(\sigma)=(-1)^{\frac{1}{2}(\frac{n(n-1)}{2}- \lfloor \frac{n}{2}\rfloor)}.
\] 
Therefore, $d\theta(X_{b_n}) = X_{b_n}$.
The extra factor $(-1)^{\lfloor\frac{n-1}{2}\rfloor}$ for the top degree comes from $d\theta (Z_k) = Z_{n-k}$.
\end{proof}

The following result is a consequence of a theorem of Ishii and Miyazaki in \cite{IshiiMiyazaki2022}.

\begin{proposition}[Ishii--Miyazaki]\label{prop:Ishii-Miyazaki}
Let $W_\pi \in \mathcal{W}(\pi)$ be a highest weight Whittaker function in the minimal ${\rm U}(n)$-type of $\pi$. Then the value $W_\pi(a)$ is positive for $a=\diag(a_1,...,a_n)$ such that $a_i\in \R^\times_+$. 
In particular, $W_\pi(1) \neq 0$.
\end{proposition}

\begin{proof}
The assertion follows by induction, using the propagation formula for $W_\pi$ due to Ishii--Miyazaki \cite[Theorem A.1]{IshiiMiyazaki2022}. In the notation of \emph{loc. cit.}, the parameters $d=(d_1,...,d_n)$ and $\nu=(\nu_1,...,\nu_n)$ are, in the present situation, given by
\[
d_i=\lambda_i^{\bar{\tau}}-\lambda_{n-i+1}^{\tau}+(n-2i+1),
\quad
\nu_i=-\tfrac{{\sf w}}{2}+\tfrac{n-2i+1}{2}.
\]
\end{proof}

We have the following period relation.

\begin{theorem}\label{thm:archi-period-relation-3}
We have 
\[
c(\pi) = \begin{cases}
1 & \mbox{ if $\bullet = b_n$},\\
(-1)^{\lfloor\frac{n-1}{2}\rfloor} & \mbox{ if $\bullet = t_n$}.
\end{cases}
\]
\end{theorem}
\begin{proof}
Fix a highest weight vector $m_{\lambda^{\tau}} \in \M_{\lambda^{\tau}}$, and define
\[
m_\lambda^{-,+}:=\rho_{\lambda}(\omega_n)m_{\lambda^{\tau}} \otimes m_{\lambda^{\tau}} \in \M_\lambda.
\]
As shown in \cite[Lemma 6.17]{DR2024}, the cohomological generators $[\pi]$, $\widetilde{[\pi]}$, and $ {}^\theta[\pi]$ are determined by their values at $X_\bullet$. Moreover, 
\[
[\pi](X_\bullet) = W_\pi\otimes m_{\lambda}^{-,+}
\] 
for some highest weight Whittaker function $W_\pi \in \W(\pi)$ in the minimal ${\rm U}(n)$-type of $\pi$.
By the definition of $\widetilde{[\pi]}$, $ {}^\theta[\pi]$, and Lemmas \ref{eq:theta-X_+}, \ref{eq:equivariant-(-,+)}, we have
\begin{align*}
    \left(({\rm id}\otimes \Phi_\lambda)\circ \widetilde{[\pi]}\right)(X_\bullet) & = (W_\pi\otimes |{\det}|^{\sf w})\otimes m_{\lambda}^{-,+}, \\ 
{}^\theta[\pi](X_\bullet) & =\theta_\W W_\pi\otimes m_{\lambda}^{-,+}\cdot \begin{cases}
1 & \mbox{ if $\bullet = b_n$},\\
(-1)^{\lfloor\frac{n-1}{2}\rfloor} & \mbox{ if $\bullet = t_n$}.
\end{cases}
\end{align*}
It remains only to determine the ratio between $W_\pi\otimes |\det|^{\sf w}$ and $\theta_\mathcal W W_\pi$. Their values at the identity are all equal to $W_\pi(1)$, which is nonzero by Proposition \ref{prop:Ishii-Miyazaki}. This completes the proof.
\end{proof}

\section{Global period relations and applications to critical $L$-values}
\label{sec:critical-l-values}

\subsection{Essentially self-dual automorphic representations}\label{sec:esse.-self-dual}

Let $\Pi$ be a cuspidal automorphic representation of $G(\A) = \GL_n(\A_F)$.
We say $\Pi$ is \emph{essentially self-dual} if there exists a Hecke character $\chi$ of $F$ such that
\begin{align}\label{eq:essen.-self-dual}
\Pi^\vee \cong \Pi \otimes \chi^{-1}\circ\det.
\end{align}
In this case, we have a factorization of $L$-functions:
\[
L(s,\Pi \times \Pi^\vee) = L(s, \Pi, {\rm Sym}^2 \otimes \chi^{-1})\cdot L(s, \Pi, \wedge^2 \otimes \chi^{-1}),
\]
where the left-hand side is the Rankin--Selberg $L$-function of $\Pi \times \Pi^\vee$ and the right-hand side are the $\chi^{-1}$-twisted symmetric square 
and exterior square $L$-functions of $\Pi$.
Since the Rankin--Selberg $L$-function admits a simple pole at $s=1$ by the result of Jacquet and Shalika \cite{JS1981b}, 
exactly one of the $L$-functions on the right-hand side has a pole at $s=1$.
We say $\Pi$ is \emph{$\chi$-orthogonal} (resp.\,\emph{$\chi$-symplectic}) if (\ref{eq:essen.-self-dual}) 
holds and the $\chi^{-1}$-twisted symmetric square (resp.\,$\chi^{-1}$-twisted exterior square) $L$-function of $\Pi$ has a pole at $s=1$.
Note that $\Pi$ must be $\chi$-orthogonal if $n$ is odd by \cite[Theorem 2]{JS1990b}.

Let $\chi$ be a Hecke character of $F$.
An isobaric automorphic representation
\[
\Pi=\Pi_1\boxplus\cdots\boxplus \Pi_k
\]
of $\GL_n(\A_F)$ is said to be $\chi$-orthogonal (resp.\,$\chi$-symplectic) if the $\Pi_i$'s are pairwise non-isomorphic 
and each $\Pi_i$ is $\chi$-orthogonal (resp.\,$\chi$-symplectic).
Assume that
\[
n=2r
\quad\text{and}\quad
(\chi^r\omega_\Pi^{-1})^2=1.
\]
This is a necessary condition for $\Pi$ to be $\chi$-orthogonal or $\chi$-symplectic when $n$ is even. In this situation, the quadratic character $\chi^r\omega_\Pi^{-1}$ determines a quasi-split general spin group ${\rm GSpin}_n^*$ over $F$, as in \cite[\S\,2.1.3]{AS2014}. We also write ${\rm GSpin}_{n+1}$ for the split general spin group over $F$.
By the results of Asgari--Shahidi \cite{AS2006b,AS2014} and Hundley--Sayag \cite{HS2016}, $\Pi$ is $\chi$-orthogonal (resp.\,$\chi$-symplectic) if and only if $\Pi$ is the weak functorial transfer of a globally generic cuspidal automorphic representation of
\[
{\rm GSpin}_n^*(\A_F)
\quad
(\text{resp.\,}{\rm GSpin}_{n+1}(\A_F))
\]
with central character $\chi$. Moreover, the transfer is strong at the archimedean places.

When $F$ is not totally imaginary and $\Pi$ is regular algebraic, the type of $\Pi$ is determined by the signature character $\varepsilon(\chi)$ of $\chi$, defined in \S\,\ref{sec:Hecke-char-Gauss-sums}.
Note that, in this case, $\chi$ is necessarily algebraic, so that $\varepsilon(\chi)$ is defined.
More precisely, we have the following result, which is a straightforward generalization of \cite[Lemma 4.1]{Chen2023}.

\begin{proposition}\label{prop:sgn-criterion}
Assume $F$ is not totally imaginary and $n$ is even. 
Let $\Pi$ be a regular algebraic, essentially unitary, isobaric automorphic representation of $\GL_n(\A_F)$. 
Then $\Pi$ is $\chi$-orthogonal (resp.\,$\chi$-symplectic) if and only if $\varepsilon(\chi) = -1$ (resp.\,$\varepsilon(\chi) = 1$).
\end{proposition}

\begin{proof}

We first reduce to the case where $\Pi$ is cuspidal.
Write
\[
\Pi = \Pi_1 \boxplus \cdots \boxplus \Pi_k
\]
for some cuspidal automorphic representations $\Pi_i$ of $\GL_{n_i}(\A_F)$ for $1 \leq i \leq k$.
By assumption, there exists ${\sf w} \in \Z$ such that $\Pi_i \otimes |\det|^{\frac{\sf w}{2}}$ is unitary for all $i$.
The regularity condition then implies that $\Pi_i \ncong \Pi_j$ for $i \neq j$ and that
\[
\Pi_i^\vee \cong \Pi_i \otimes \chi^{-1}\circ\det,\quad 1 \leq i \leq k.
\]
Moreover, since $F$ is not totally imaginary, $n_i$ is even for all $i$. Indeed, if $n_i$ is odd for some $i$, then, since $n$ is even, $n_j$ is also odd for some $j \neq i$. Consider a real place $v \in S_r$. Then the Langlands parameters of $\Pi_{i,v}$ and $\Pi_{j,v}$ both contain a one-dimensional summand. This contradicts the regularity of $\Pi$.

Assume that $\Pi$ is cuspidal.
Let $v \in S_r$ be a real place of $F$. The Langlands parameter $\phi_{\Pi_v}$ of $\Pi_v$ is an $n$-dimensional admissible representation of the Weil group $W_\R$ of $\R$.
As explained in the proofs of \cite[Theorem 5.3]{GR2013} and \cite[Lemma 4.1]{Chen2023}, when $n$ is even, the image of $\phi_{\Pi_v}$ factors through exactly one of the two subgroups ${\rm GO}_n(\C)$ and ${\rm GSp}_n(\C)$ of $\GL_n(\C)$, according as $\varepsilon(\chi)_v=-1$ or $\varepsilon(\chi)_v=1$.
On the other hand, by the results of Asgari--Shahidi and Hundley--Sayag recalled above, if $\Pi$ is $\chi$-orthogonal (resp.\,$\chi$-symplectic), then the image of $\phi_{\Pi_v}$ factors through ${\rm GO}_n(\C)$ (resp.\,${\rm GSp}_n(\C)$) for all archimedean places $v$.
Conversely, since $\Pi$ is cuspidal, it is of exactly one of the two global types. Since the transfer is strong at $v \in S_r$, the global type of $\Pi$ must agree with the type of $\phi_{\Pi_v}$. The assertion then follows from the existence of a real place.
\end{proof}

As an immediate consequence, we obtain an arithmeticity result for the orthogonal and symplectic types, 
extending a result of Gan and Raghuram \cite[Theorem 5.3]{GR2013}. 

\begin{corollary}\label{coro:arithmeticity}
Assume $F$ is not totally imaginary and $n$ is even. 
Let $\Pi$ be a regular algebraic cuspidal automorphic representation of $\GL_n(\A_F)$. 
If $\Pi$ is $\chi$-orthogonal (resp.\,$\chi$-symplectic), then ${}^\sigma\Pi$ is ${}^\sigma\chi$-orthogonal (resp.\,${}^\sigma\chi$-symplectic) for all $\sigma \in {\rm Aut}(\C)$.
\end{corollary}

\begin{proof}
Assume $\Pi$ satisfies (\ref{eq:essen.-self-dual}).
For $\sigma \in {\rm Aut}(\C)$, ${}^\sigma\Pi$ is regular algebraic and cuspidal, and 
\[
{}^\sigma\Pi^\vee \cong {}^\sigma\Pi \otimes {}^\sigma\chi^{-1}\circ\det.
\]
Thus, by Proposition \ref{prop:sgn-criterion}, ${}^\sigma\Pi$ is either ${}^\sigma\chi$-orthogonal or ${}^\sigma\chi$-symplectic according as $\varepsilon({}^\sigma\chi)=-1$ or $\varepsilon({}^\sigma\chi)=1$.
Since $\varepsilon({}^\sigma\chi) = \varepsilon(\chi)$, the assertion follows at once. 
\end{proof}

\begin{remark}\label{rem:Clozel-Kret-Taibi}
For an arbitrary number field $F$, if $\chi = |\mbox{ }|_{\A_F}^{-{\sf w}}$, then the assertion is established by Clozel--Kret--Ta\"ibi \cite[Proposition 8.2]{CKT2026}, subject to Arthur's endoscopic classification \cite{Arthur2013}.
\end{remark}

\subsubsection{Automorphic tensor product}
For automorphic representations $\Pi_1$ and $\Pi_2$ of $\GL_{n_1}(\A_F)$ and $\GL_{n_2}(\A_F)$, respectively, 
recall that the \emph{automorphic tensor product}
$\Pi_1 \boxtimes \Pi_2$
is the irreducible admissible $(\frak{g}_\infty,C_\infty) \times G(\A_f)$-module such that, for each place $v$ of $F$, the Langlands parameter of $(\Pi_1 \boxtimes \Pi_2)_v$ is the tensor product
$\phi_{\Pi_{1,v}} \otimes \phi_{\Pi_{2,v}}$
of the Langlands parameters of $\Pi_{1,v}$ and $\Pi_{2,v}$.
As a special instance of the Langlands functoriality conjecture, $\Pi_1 \boxtimes \Pi_2$ is expected to be automorphic. Moreover, it should be isobaric if $\Pi_1$ and $\Pi_2$ are.
Another consequence concerns the criterion for the automorphic tensor product to be orthogonal or symplectic.
Note that the criterion reduces to Proposition \ref{prop:sgn-criterion} itself when $n_1=n$ and $n_2=1$.

\begin{corollary}\label{coro:automorphic-tensor-product}
Let $\Pi_1$ and $\Pi_2$ be regular algebraic, essentially unitary, 
isobaric automorphic representations of $\GL_{n_1}(\A_F)$ and $\GL_{n_2}(\A_F)$ such that
\[
\Pi_1^\vee\cong\Pi_1 \otimes \chi_1^{-1}\circ\det,\quad \Pi_2^\vee\cong\Pi_2 \otimes \chi_2^{-1}\circ\det.
\]
Assume $F$ is not totally imaginary, $n_1n_2$ is even, and $\Pi_1 \boxtimes \Pi_2$ is regular and isobaric automorphic. 
Then 
\[
(\Pi_1 \boxtimes \Pi_2) \otimes |\det|_{\A_F}^{\frac{n_1+n_2-1}{2}}
\] 
is regular algebraic, and it is $\chi_1\chi_2|\det|_{\A_F}^{n_1+n_2-1}$-orthogonal (resp.\,$\chi_1\chi_2|\det|_{\A_F}^{n_1+n_2-1}$-symplectic) 
if and only if $\varepsilon(\chi_1\chi_2) = (-1)^{n_1+n_2}$ (resp.\,$\varepsilon(\chi_1\chi_2) = (-1)^{1+n_1+n_2}$).
\end{corollary}

\begin{proof}
The assertion follows from applying Proposition \ref{prop:sgn-criterion} to  
\[
\Pi := (\Pi_1 \boxtimes \Pi_2) \otimes |\det|_{\A_F}^{\frac{n_1+n_2-1}{2}}, \quad \chi := \chi_1\chi_2|\det|_{\A_F}^{n_1+n_2-1}.
\]
Note that $\Pi$ is algebraic since $n_1n_2$ is even; and $\varepsilon(\chi) = (-1)^{n_1+n_2+1}\varepsilon(\chi_1\chi_2)$.
\end{proof}

\subsubsection{Asai transfer}

Let $E/F$ be a finite extension of degree $d=[E:F]$, and let $\Sigma$ be an automorphic representation of $\GL_m(\A_E)$. 
We recall the \emph{Asai transfer} of $\Sigma$,
\[
{\rm As}(\Sigma)={\bigotimes_v}' {\rm As}(\Sigma_v),
\]
in terms of local Langlands parameters. Let $v$ be a place of $F$, and write
$E_v = E\otimes_F F_v=\prod_{w\mid v}E_w.$
Set
$I_v:=\Hom_{F_v}(E_v,\overline F_v) = \coprod_{w\mid v}\Hom_{F_v}(E_w,\overline F_v).$
Then $|I_v|=d$, and
\[
{}^L(\Res_{E_v/F_v}\GL_m)
=
\prod_{\tau\in I_v}\GL_m(\C)\rtimes W_{F_v},
\]
where $W_{F_v}$ acts on $I_v$ by $\gamma\cdot\tau=\gamma\circ\tau$, and hence
acts on $\prod_{\tau\in I_v}\GL_m(\C)$ by
$\gamma\cdot(g_\tau)_{\tau\in I_v} =
(g_{\gamma^{-1}\tau})_{\tau\in I_v}.$
The local Langlands parameter of
$\Sigma_v=\otimes_{w\mid v}\Sigma_w$
is then a homomorphism
\[
\phi_{\Sigma_v}:W_{F_v}'\longrightarrow
\prod_{\tau\in I_v}\GL_m(\C)\rtimes W_{F_v}.
\]
Let
\[
{\rm As}:
\prod_{\tau\in I_v}\GL_m(\C)\rtimes W_{F_v}
\longrightarrow
\GL\left(\bigotimes_{\tau\in I_v}\C^m\right)
\]
be the representation defined by
\[ 
{\rm As}\bigl((g_\tau)_{\tau\in I_v}\bigr)\cdot
\bigotimes_{\tau\in I_v}v_\tau : =
\bigotimes_{\tau\in I_v}g_\tau v_\tau,\ \ {\rm and} \ \ 
{\rm As}(\gamma)\cdot
\bigotimes_{\tau\in I_v}v_\tau
:=
\bigotimes_{\tau\in I_v}v_{\gamma^{-1}\tau},
\quad \gamma\in W_{F_v}.
\]
We define ${\rm As}(\Sigma_v)$ to be the representation of
$\GL_{m^d}(F_v)$ corresponding to the parameter
$
{\rm As}\circ\phi_{\Sigma_v}
$
under the local Langlands correspondence.
Conjecturally, the Asai transfer is automorphic. Moreover, it should be isobaric if $\Sigma$ is.

\begin{corollary}\label{coro:Asai-transfer}
Assume $F$ is not totally imaginary, and all real places split in $E$. Let $\Sigma$ be a regular algebraic cuspidal automorphic representation of $\GL_m(\A_E)$ such that
\[
\Sigma^\vee \cong \Sigma \otimes \mu^{-1}\circ\det.
\]
Assume $m$ is even, ${\rm As}(\Sigma)$ is regular and isobaric automorphic. 
Then 
\[
{\rm As}(\Sigma) \otimes |\det|_{\A_F}^{\frac{d-1}{2}}
\]
is regular algebraic, and it is $(\mu\vert_{\A_F^\times})|\mbox{ }|_{\A_F}^{d-1}$-orthogonal if and only if $d$ is even or $\Sigma$ is $\mu$-orthogonal. 
\end{corollary}

\begin{proof}
The assertion follows from applying Proposition \ref{prop:sgn-criterion} to 
\[
\Pi := {\rm As}(\Sigma) \otimes |\det|_{\A_F}^{\frac{d-1}{2}},\quad \chi:=(\mu\vert_{\A_F^\times})|\mbox{ }|_{\A_F}^{d-1}.
\]
By Proposition \ref{prop:sgn-criterion}, both $\varepsilon(\mu\vert_{\A_F^\times})$ and $\varepsilon(\mu)$ are constant sign characters. 
Moreover, since every real place splits in $E$, we see that 
$\varepsilon(\mu\vert_{\A_F^\times}) = \varepsilon(\mu)^d.$ 
Therefore, $\varepsilon(\mu\vert_{\A_F^\times}) = (-1)^d$ if and only if $d$ is even or $\varepsilon(\mu) = -1$.
\end{proof}

\begin{remark}
If some real place of $F$ is not split in $E$, then ${\rm As}(\Sigma)$ is not regular.
\end{remark}

\subsection{Global period relations}

Let $\Pi$ be a regular algebraic cuspidal automorphic representation of $G(\A)$ such that $\Pi_f \in \Coh_{\rm cusp}(G,\lambda)$ for some $\lambda \in X_{00}^+(T\times\C)$. 
Let $\bullet \in \{b_n^F,t_n^F\}$ be an extreme degree, and let $\varepsilon \in \widehat{\pi_0(G(\R))}$ be permissible for $\Pi_\infty$.
Fix a generator
\[
[\Pi_\infty]^\varepsilon \in H^\bullet(\frak{g}_\infty,K_\infty^\circ; \W(\Pi_\infty)\otimes\M_\lambda)(\varepsilon).
\]

In this section, we establish two period relations for the Betti--Whittaker periods. 
The first is a relation between the $\det$-twisted and $\theta$-twisted constructions, and the second is a relation between opposite sign characters.

\subsubsection{Period relation under $\det$-twist and $\theta$-twist}

The generator $[\Pi_\infty]^\varepsilon$ gives the $\theta$-twisted generator (cf.\,\S\,\ref{sec:archi-local-dual})
\[
{}^\theta[\Pi_\infty]^\varepsilon \in H^\bullet(\frak{g}_\infty,K_\infty^\circ;\W(\Pi_\infty^\vee)\otimes{}^\theta\M_\lambda)(\varepsilon).
\]
Assume further that $\lambda^\vee = \lambda-{\sf w}$.
Then we also have the determinant-twisted generator defined in (\ref{eq:archi-local-dual-2})
\[
\widetilde{[\Pi_\infty]}^\varepsilon \in H^\bullet(\frak{g}_\infty,K_\infty^\circ;\W(\Pi_\infty^\vee)\otimes\M_{\lambda-{\sf w}})(\varepsilon).
\]
With respect to these generators, we have two Betti--Whittaker periods $p_{\det}^\varepsilon(\Pi^\vee)$ and $p_{\theta}^\varepsilon(\Pi^\vee)$ of $\Pi^\vee$, defined in \S\,\ref{sec:comparison-isomorphism}. The subscripts refer to the dependence on the determinant-twisted and $\theta$-twisted generators, respectively.
Similarly, for each $\sigma \in {{\rm Aut}(\C)}$, we have the Betti--Whittaker periods $p_{\det}^\varepsilon({}^\sigma\Pi^\vee)$ and $p_{\theta}^\varepsilon({}^\sigma\Pi^\vee)$ of ${}^\sigma\Pi^\vee$, with respect to a chosen generator $[{}^\sigma\Pi_\infty]^\varepsilon$.
For compatibility, we assume $[{}^\sigma\Pi_\infty]^\varepsilon = [\Pi_\infty]^\varepsilon$ if $\sigma \in {\rm Aut}(\C/\Q(\Pi))$.
We normalize the periods $(p_{\det}^\varepsilon({}^\sigma\Pi^\vee))_{\sigma \in {{\rm Aut}(\C)}}$ and $(p_{\theta}^\varepsilon({}^\sigma\Pi^\vee))_{\sigma \in { {\rm Aut}(\C)}}$ by (\ref{eq:W-H-periods}).

\begin{theorem}\label{thm:det-theta}
With the notations as above, one has
\[
\sigma \left( \frac{p_{\det}^\varepsilon(\Pi^\vee)}{p_{\theta}^\varepsilon(\Pi^\vee)} \right) = \frac{p_{\det}^\varepsilon({}^\sigma\Pi^\vee)}{p_{\theta}^\varepsilon({}^\sigma\Pi^\vee)},\quad \sigma \in {\rm Aut}(\C).
\]
\end{theorem}

\begin{proof}
Recall the $G(\C)$-equivariant isomorphism 
\[
\Phi_\lambda : (\rho_{\lambda-{\sf w}}, \mathcal{M}_{\lambda-{\sf w}}) \longrightarrow ({}^\theta\rho_{\lambda}, {}^\theta\mathcal{M}_{\lambda}) 
\]
normalized by (\ref{eq:canonical-normalization}).
By the normalization, for each $\sigma \in {\rm Aut}(\C)$, we have ${}^\sigma\Phi_\lambda = \Phi_{{}^\sigma\lambda}$ if we take ${}^\sigma \M_\lambda$ to be the model of $\M_{{}^\sigma\lambda}$. In other words, the following diagram
\[
\begin{tikzcd}
	\mathcal{M}_{\lambda-{\sf w}} \arrow[r,"\Phi_\lambda"] \arrow[d,"{t_\sigma}"] 
	& {}^\theta\mathcal{M}_\lambda \arrow[d,"{t_\sigma}"] \\
	{}^\sigma\mathcal{M}_{\lambda-{\sf w}} \arrow[r,"\Phi_{{}^\sigma\lambda}"] 
	& {}^{\sigma}({}^\theta\mathcal{M}_\lambda) = {}^{\theta}({}^\sigma\mathcal{M}_\lambda)
\end{tikzcd}
\]
is commutative, where the vertical maps are the natural maps defined in (\ref{eq:Galois-identity}).
Let $c^\varepsilon({}^\sigma\Pi_\infty)$ be the scalar defined as in (\ref{eq:local-period}) for ${}^\sigma\Pi_\infty$.
By the above diagram and the definition of Betti--Whittaker periods, we then have
\[
\sigma \left( \frac{p_{\det}^\varepsilon(\Pi^\vee)}{c^\varepsilon(\Pi_\infty)\cdot p_{\theta}^\varepsilon(\Pi^\vee)} \right) = \frac{p_{\det}^\varepsilon({}^\sigma\Pi^\vee)}{c^\varepsilon({}^\sigma\Pi_\infty)\cdot p_{\theta}^\varepsilon({}^\sigma\Pi^\vee)},\quad \sigma \in {\rm Aut}(\C).
\]
The assertion then follows from the archimedean period relation Theorem \ref{thm:archi-period-relation}.
\end{proof}

\subsubsection{Period relation for cuspidal automorphic representations of orthogonal type}

Assume $n$ is even.
The generator $[\Pi_\infty]^\varepsilon$ determines a generator in the $-\varepsilon$-isotypic part
\[
[\Pi_\infty]^{-\varepsilon} \in H^\bullet(\frak{g}_\infty,K_\infty^\circ; \W(\Pi_\infty)\otimes\M_\lambda)(-\varepsilon)
\]
by the diagram
\begin{equation*}
\begin{tikzcd}[row sep=normal, column sep=normal]
\extp^{\bullet}\frak{g}_\infty/\frak{k}_\infty \arrow[r, "{[\Pi_\infty]}^\varepsilon"]  & \mathcal{W}(\Pi_\infty)\otimes \mathcal{M}_{\lambda}\arrow[d, "(\,\cdot\, \otimes {\rm sgn}\circ\det)\otimes {\rm id}"]\\
\extp^{\bullet}\frak{g}_\infty/\frak{k}_\infty \arrow[r, "{[\Pi_\infty]}^{-\varepsilon}"]\arrow[u, "{\rm id}"] & \mathcal{W}(\Pi_\infty)\otimes \mathcal{M}_{\lambda},
\end{tikzcd}
\end{equation*}
where $``{\rm sgn}"$ is the character of $F_\infty^\times$ defined by
\[
{\rm sgn}_v:= \begin{cases}
{\rm sgn} & \mbox{ if $v \in S_r$},\\
1 & \mbox{ if $v \in S_c$}.
\end{cases}
\]
With respect to the generators $[\Pi_\infty]^\varepsilon$ and $[\Pi_\infty]^{-\varepsilon}$, we have the Betti--Whittaker periods $p^\varepsilon(\Pi)$ and $p^{-\varepsilon}(\Pi)$ of $\Pi$ defined in \S\,\ref{sec:comparison-isomorphism}.

We have the following period relation when $\Pi$ is of orthogonal type.

\begin{theorem}\label{thm:opposite-period-relation}
Assume $n$ is even, $\Pi$ is $\chi$-orthogonal, and either $F$ is not totally imaginary or $\chi = |\mbox{ }|_{\A_F}^{-{\sf w}}$. Then we have
\begin{align*}
\sigma \left(\frac{p^\varepsilon(\Pi)}{\Gsum(\chi^{\frac{n}{2}}\omega_\Pi^{-1})\cdot p^{-\varepsilon}(\Pi)} \right) = \frac{p^\varepsilon({}^\sigma \Pi)}{\Gsum(\chi^{\frac{n}{2}}\omega_{\Pi}^{-1})\cdot p^{-\varepsilon}({}^\sigma \Pi)} ,\quad \sigma \in {\rm Aut}(\C).
\end{align*}
\end{theorem}

\begin{proof}
Note that the assumptions imply that $\lambda^\vee = \lambda-{\sf w}$.
Indeed, if $\chi = |\mbox{ }|_{\A_F}^{-{\sf w}}$, then the second assumption implies that $\Pi_\infty^\vee \cong \Pi_\infty \otimes |\det|_{\infty}^{\sf w}$.
If $F$ is not totally imaginary, then it is shown in
\cite[Proposition 3.3]{DR2024} that $\lambda$ is a base change from the maximal totally real subfield of $F$. 
In both cases, we have $\lambda^\tau = \lambda^{\overline{\tau}}$ for all $\tau:F\rightarrow \C$.

By the period relations in Theorems \ref{thm:period-relation} and \ref{thm:det-theta}, we have
\begin{equation}\label{eq:opposite-period-relation-pf1}
		\sigma\left( \frac{p_{\det}^\eps(\Pi\dual)}{\Gsum(\omega_{\Pi})^{1-n}\cdot p^\eps(\Pi)} \right) = \frac{p_{\det}^\eps({}^\sigma\Pi\dual)}{\Gsum({}^\sigma\omega_{\Pi})^{1-n}\cdot p^\eps({}^\sigma\Pi)},\quad \sigma \in {\rm Aut}(\C).
\end{equation}
On the other hand, for each algebraic Hecke character $\omega$ of $F$, we have the period relation established by Raghuram--Shahidi \cite[Theorem 4.1]{RS2008}:
\begin{align}\label{eq:opposite-period-relation-pf2}
\sigma \left( \frac{p^\varepsilon(\Pi\otimes\omega)}{\Gsum(\omega)^{{\frac{n(n-1)}{2}}}\cdot p^{\varepsilon\cdot\varepsilon(\omega)}(\Pi)} \right) = \frac{p^\varepsilon({}^\sigma\Pi\otimes{}^\sigma\omega)}{\Gsum({}^\sigma\omega)^{{\frac{n(n-1)}{2}}}\cdot p^{\varepsilon\cdot\varepsilon(\omega)}({}^\sigma\Pi)},\quad \sigma \in {\rm Aut}(\C).
\end{align}
Here the Betti--Whittaker period $p^\varepsilon(\Pi\otimes\omega)$ is defined with respect to the generator $[\Pi_\infty \otimes \omega_\infty]^\varepsilon$ normalized by the diagram
\begin{equation*}
\begin{tikzcd}[row sep=normal, column sep=normal]
\extp^{\bullet}\frak{g}_\infty/\frak{k}_\infty \arrow[rr, "{[\Pi_\infty]}^{\varepsilon\cdot\varepsilon(\omega)}"]  && \mathcal{W}(\Pi_\infty)\otimes \mathcal{M}_{\lambda}\arrow[d, "(\,\cdot\, \otimes \omega_\infty\circ\det)\otimes {\rm id}"]\\
\extp^{\bullet}\frak{g}_\infty/\frak{k}_\infty \arrow[rr, "{[\Pi_\infty\otimes \omega_\infty]}^{\varepsilon}"]\arrow[u, "{\rm id}"] && \mathcal{W}(\Pi_\infty\otimes\omega_\infty)\otimes (\mathcal{M}_{\lambda}\otimes \omega_{\rm alg}^{-1}\circ\det),
\end{tikzcd}
\end{equation*}
where $\omega_{\rm alg}$ is the algebraic character of $\prod_{\tau : F \rightarrow \C}\C^\times$ defined by (cf.\,\S\,\ref{sec:Hecke-char-Gauss-sums})
\[
\omega_{\rm alg}(x):= \prod_{\tau : F \rightarrow \C} x_\tau^{a(\omega)_\tau},\quad x = (x_\tau)_{\tau : F \rightarrow \C}.
\]
Assume $\Pi$ is $\chi$-orthogonal. By Proposition \ref{prop:sgn-criterion}, we have $\varepsilon(\chi) = -1$ if $F$ is not totally imaginary.
Together with the condition $\lambda^\vee = \lambda - {\sf w}$, we deduce that $\chi_\infty = {\rm sgn}\cdot \nu_{\sf w}^{-1}$, where $\nu_{\sf w}$ is defined in (\ref{eq:det-character}).
Therefore, 
\[
[\Pi_\infty \otimes \chi_\infty^{-1}]^\varepsilon = \widetilde{[\Pi_\infty]}^\varepsilon
\] 
and hence
\[
p^\varepsilon(\Pi \otimes \chi^{-1}) = p_{\det}^\varepsilon(\Pi^\vee).
\] 
The proposed period relation then follows from (\ref{eq:opposite-period-relation-pf1}) and (\ref{eq:opposite-period-relation-pf2}).
This completes the proof.
\end{proof}

\begin{remark}
\label{rem:bhagwat-raghuram-1}
In the above theorem, suppose $F$ is totally real and $n \equiv 0 \pmod{4}$,  
and the cuspidal $\Pi \otimes |\det|_{\A_F}^{\frac{\sf w}{2}}$ is the transfer of a cuspidal automorphic representation $\sigma$ of the 
split orthogonal group ${\rm O}(n/2, n/2)$; then $\Pi$ is orthogonal ($\chi = |\mbox{ }|_{\A_F}^{-{\sf w}}$ and $\omega_\Pi = |\mbox{ }|_{\A_F}^{-\frac{n{\sf w}}{2}}$).
The relative period $\Omega^\epsilon(\Pi)$ of $\Pi$, defined in \cite[\S\,5.2.3]{HR2020}, being 
expressible as the ratio $p^\varepsilon(\Pi)/p^{-\varepsilon}(\Pi)$ of the Betti--Whittaker periods (cf.\,\cite[Theorem 3.1]{Raghuram2013}), 
becomes trivial by Theorem~\ref{thm:opposite-period-relation}. 
This gives another proof of the main theorem of Bhagwat and Raghuram \cite{BR2020} on the ratios of special values of the degree-$n$ $L$-function of $\sigma$. 
See also Remark~\ref{rem:bhagwat-raghuram-2}. 
\end{remark}

\subsection{Ratio of consecutive critical $L$-values}

\subsubsection{Triple product $L$-functions}

Let $\Pi_i$ be cuspidal automorphic representations of $\GL_{n_i}(\A_F)$ for $i=1,2,3$.
The associated triple product $L$-function is defined by
\[
L(s,\Pi_1\times\Pi_2\times\Pi_3) := \prod_v L(s,\phi_{\Pi_{1,v}}\otimes\phi_{\Pi_{2,v}}\otimes\phi_{\Pi_{3,v}}),\quad {\rm Re}(s) \gg 0.
\]
Here $v$ varies over the places of $F$.
Conjecturally, it admits meromorphic continuation to the whole complex plane and satisfies a functional equation relating $L(s,\Pi_1\times\Pi_2\times\Pi_3)$ with $L(1-s,\Pi_1^\vee\times\Pi_2^\vee\times\Pi_3^\vee)$.
This is known if the automorphic tensor product $\Pi_i \boxtimes \Pi_j$ is isobaric automorphic for some $i\neq j$.

Assume $\Pi_1,\Pi_2,\Pi_3$ are regular algebraic.
A critical point for the triple product $L$-function is a half-integer $m_0 \in \Z+ \tfrac{n_1+n_2+n_3+1}{2}$ such that the local factors $L(s,\Pi_{1,\infty}\times\Pi_{2,\infty}\times\Pi_{3,\infty})$ and $L(1-s,\Pi_{1,\infty}^\vee\times\Pi_{2,\infty}^\vee\times\Pi_{3,\infty}^\vee)$ are holomorphic at $s = m_0$.
As an application to the period relation Theorem \ref{thm:opposite-period-relation} and the 
algebraicity result of Harder--Raghuram \cite{HR2020}, we prove the algebraicity of the ratio of the triple product $L$-functions at consecutive critical points.

\begin{theorem}\label{thm:ratio-triple}
Assume the following conditions are satisfied:
\begin{itemize}
\item[(1)] $F$ is totally real.
\item[(2)] $n_1n_2$ is even.
\item[(3)] There exist algebraic Hecke characters $\chi_1$ and $\chi_2$ of $F$ such that 
\[
\Pi_1^\vee\cong\Pi_1 \otimes \chi_1^{-1}\circ\det,\quad \Pi_2^\vee\cong\Pi_2 \otimes \chi_2^{-1}\circ\det.
\]
\item[(4)] $\varepsilon(\chi_1\chi_2) = (-1)^{n_1+n_2}$. 
\item[(5)] $\Pi_1 \boxtimes \Pi_2$ is regular and automorphic.
\end{itemize}
Let $m_0,m_0+1 \in \Z+\tfrac{n_1+n_2+n_3+1}{2}$ be critical points such that $L(m_0+1,\Pi_1 \times \Pi_2 \times \Pi_3) \neq 0$. Then we have
\begin{align*}
&\sigma \left( \frac{L(m_0,\Pi_1 \times \Pi_2 \times \Pi_3)}{(\sqrt{-1})^{\frac{n_1n_2n_3}{2}}\cdot \mathcal{G}\left(\chi_1^{\frac{n_1n_2}{2}}\omega_{\Pi_1}^{-n_2}\right)^{n_3} \mathcal{G}\left(\chi_2^{\frac{n_1n_2}{2}}\omega_{\Pi_2}^{-n_1}\right)^{n_3} \cdot L(m_0+1,\Pi_1 \times \Pi_2 \times \Pi_3)} \right)\\
& =  \frac{L(m_0,{}^\sigma\Pi_1 \times {}^\sigma\Pi_2 \times {}^\sigma\Pi_3)}{(\sqrt{-1})^{\frac{n_1n_2n_3}{2}}\cdot \mathcal{G}\left(\chi_1^{\frac{n_1n_2}{2}}\omega_{\Pi_1}^{-n_2}\right)^{n_3} \mathcal{G}\left(\chi_2^{\frac{n_1n_2}{2}}\omega_{\Pi_2}^{-n_1}\right)^{n_3} \cdot L(m_0+1,{}^\sigma\Pi_1 \times {}^\sigma\Pi_2 \times {}^\sigma\Pi_3)}
\end{align*}
for all $\sigma \in {\rm Aut}(\C)$.
\end{theorem}

\begin{proof}
Let $n:=n_1n_2$, and set
\[
\Pi := (\Pi_1 \boxtimes \Pi_2)\otimes |\det|_{\A_F}^{\frac{n_1+n_2-1}{2}},
\quad
\chi := \chi_1\chi_2|\det|_{\A_F}^{n_1+n_2-1}.
\]
By conditions (2) and (5), $\Pi$ is a regular algebraic automorphic representation of $\GL_n(\A_F)$.
We note that condition (1), together with the results of Caraiani et al. (cf.\,\cite[Theorem 1.2]{Caraiani2012} and the references therein), implies that $\Pi_1$ and $\Pi_2$ are essentially tempered everywhere.
Hence $\Pi$ is essentially tempered everywhere.
By \cite[Lemme 1.5]{Clozel1990}, it follows that $\Pi$ is isobaric.

Write
\[
\Pi = \Sigma_1 \boxplus \cdots \boxplus \Sigma_k
\]
for some cuspidal automorphic representations $\Sigma_i$ of $\GL_{t_i}(\A_F)$ for $1 \leq i \leq k$.
By condition (4) and Corollary \ref{coro:automorphic-tensor-product}, $\Pi$ is $\chi$-orthogonal. Moreover, as explained in the proof of 
Proposition \ref{prop:sgn-criterion}, $t_i$ is even for all $i$.
Since
\[
\prod_{i=1}^k \omega_{\Sigma_i} = \omega_{\Pi_1}^{n_2}\omega_{\Pi_2}^{n_1}|\mbox{ }|_{\A_F}^{\frac{n_1n_2(n_1+n_2-1)}{2}},
\]
by (\ref{eq:Gauss-sum-2}) we have
\begin{align*}
\sigma \left( \frac{\mathcal{G}\left(\chi_1^{\frac{n_1n_2}{2}}\omega_{\Pi_1}^{-n_2}\right) 
\mathcal{G}\left(\chi_2^{\frac{n_1n_2}{2}}\omega_{\Pi_2}^{-n_1}\right)}{\prod_{i=1}^k \Gsum(\chi^{\frac{t_i}{2}}\omega_{\Sigma_i}^{-1})} \right) = 
\frac{\mathcal{G}\left(\chi_1^{\frac{n_1n_2}{2}}\omega_{\Pi_1}^{-n_2}\right) 
\mathcal{G}\left(\chi_2^{\frac{n_1n_2}{2}}\omega_{\Pi_2}^{-n_1}\right)}{\prod_{i=1}^k \Gsum(\chi^{\frac{t_i}{2}}\omega_{\Sigma_i}^{-1})},
\quad \sigma \in {\rm Aut}(\C).
\end{align*}
Note that 
\[
L(s,\Pi_1 \times \Pi_2 \times \Pi_3) = \prod_{i=1}^kL(s-\tfrac{n_1+n_2-1}{2},\Sigma_i \times \Pi_3).
\]
Let 
\[
m_0':=m_0-\tfrac{n_1+n_2-1}{2} \in \Z+\tfrac{n_3}{2},
\]
which is a critical point for $L(s,\Sigma_i \times \Pi_3)$ for all $i$.
We recall the result of Harder--Raghuram \cite[Theorem 7.21]{HR2020} on the ratio of Rankin--Selberg $L$-functions at consecutive critical points: If $n_3$ is odd, then we have
\begin{align*}
&\sigma \left( \frac{L(m_0',\Sigma_i \times \Pi_3)}{\Omega^{\varepsilon_i\cdot\varepsilon(\Pi_3)}(\Sigma_i)\cdot L(m_0'+1,\Sigma_i \times \Pi_3)} \right)\\
& = \frac{L(m_0',{}^\sigma\Sigma_i \times {}^\sigma\Pi_3)}{\Omega^{\varepsilon_i\cdot\varepsilon(\Pi_3)}({}^\sigma\Sigma_i)\cdot L(m_0'+1,{}^\sigma\Sigma_i \times {}^\sigma\Pi_3)},\quad \sigma \in {\rm Aut}(\C)
\end{align*}
where 
\[
\varepsilon_i:= (-1)^{m_0' + \frac{t_i+n_3}{2}},
\]
and
$\Omega^\varepsilon(\Sigma_i)$ is the relative period of $\Sigma_i$ defined in \cite[\S\,5.2.3]{HR2020} 
and can be expressed in terms of ratios of the Betti--Whittaker periods (cf.\,\cite[Theorem 3.1]{Raghuram2013})
\[
\Omega^\varepsilon(\Sigma_i) = (\sqrt{-1})^{\frac{t_i}{2}}\cdot \frac{p^\varepsilon(\Sigma_i)}{p^{-\varepsilon}(\Sigma_i)}.
\]
If $n_3$ is even, then similar result holds with the relative periods replaced by $1$.
Finally, since $t_i$ is even and $\Sigma_i$ is $\chi$-orthogonal, by the period relation in Theorem \ref{thm:opposite-period-relation}, we have
\[
\sigma \left( \frac{\Omega^\varepsilon(\Sigma_i)}{(\sqrt{-1})^{\frac{t_i}{2}}\cdot \Gsum(\chi^{\frac{t_i}{2}}\omega_{\Sigma_i}^{-1})} \right) = \frac{\Omega^\varepsilon({}^\sigma\Sigma_i)}{(\sqrt{-1})^{\frac{t_i}{2}}\cdot \Gsum(\chi^{\frac{t_i}{2}}\omega_{\Sigma_i}^{-1})},\quad \sigma \in {\rm Aut}(\C).
\]
This completes the proof.
\end{proof}

\begin{remark}\label{rem:bhagwat-raghuram-2}
If $n_3$ is even, then conditions (3) and (4) can be removed.

If $n_2 = 1$ and $\Pi_2$ is the trivial representation, then conditions (2)-(5) can be replaced by (i) $n_1$ is even, (ii) $\Pi_1$ is $\chi_1$-orthogonal. 
In this case, the triple product $L$-function is a tensor product $L$-function for ${\rm GSpin}_{n_1}^* \times \GL_{n_3}$, and our result generalizes the 
theorem of Bhagwat--Raghuram \cite{BR2020}, where $\chi_1$ is a power of $|\cdot|_{\A_F}$ and $n_1$ is assumed to be divisible by $4$.
See also Remark~\ref{rem:bhagwat-raghuram-1}. 

The automorphy of $\Pi_1 \boxtimes \Pi_2$ is known when $n_1=n_2=2$ by the work of Ramakrishnan \cite{Rama2000}, and when $n_1=2$ and $n_2=3$ by the work of Kim and Shahidi \cite{KS2002}. 
\end{remark}

\subsubsection{Twisted Asai $L$-functions}

Let $E/F$ be a finite extension of degree $d=[E:F]$. 
Let $K$ be the étale quadratic algebra over $F$ (discriminant algebra of $E/F$) defined by
\[
K:= \begin{cases}
F\left(\sqrt{{\rm disc}(E/F)}\right) & \mbox{ if ${\rm disc}(E/F) \notin (F^\times)^2$},\\
F \times F & \mbox{ if ${\rm disc}(E/F) \in (F^\times)^2$}.
\end{cases}
\]
Let $\Sigma$ and $\Pi'$ be cuspidal automorphic representations of $\GL_m(\A_E)$ and $\GL_{n'}(\A_F)$ respectively. 
The associated twisted Asai $L$-function is defined by
\[
L(s,{\rm As}(\Sigma) \times \Pi'):=\prod_v L(s,({\rm As}\circ \phi_{\Sigma_v}) \otimes \phi_{\Pi_v'}),\quad {\rm Re}(s) \gg0.
\]
Here $v$ varies over places of $F$. 
Conjecturally, it admits meromorphic continuation to the whole complex plane and satisfies a functional equation relating $L(s,{\rm As}(\Sigma) \times \Pi')$ with $L(1-s,{\rm As}(\Sigma^\vee) \times (\Pi')^\vee)$.
This is known if $d=m=2$ by the result of Krishnamurthy \cite{Kris2003}.

Assume $\Sigma$ and $\Pi'$ are regular algebraic.
A critical point for the twisted Asai $L$-function is a half-integer $m_0 \in \Z+ \tfrac{d+n'+1}{2}$ such that the local factors $L(s,{\rm As}(\Sigma_\infty) \times \Pi_\infty')$ and $L(1-s,{\rm As}(\Sigma_\infty^\vee) \times (\Pi_\infty')^\vee)$ are holomorphic at $s = m_0$.
We have the algebraicity of the ratio of the twisted Asai $L$-functions at consecutive critical points.

\begin{theorem}\label{thm:ratio-Asai}
Assume the following conditions are satisfied:
\begin{itemize}
\item[(1)] $E$ is totally real and $d>1$.
\item[(2)] $m$ is even.
\item[(3)] There exists an algebraic Hecke character $\mu$ of $E$ such that
\[
\Sigma^\vee \cong \Sigma \otimes \mu^{-1}\circ\det.
\]
\item[(4)] $d$ is even or $\Sigma$ is $\mu$-orthogonal.
\item[(5)] ${\rm As}(\Sigma)$ is regular and automorphic.
\end{itemize}
Let $m_0,m_0+1 \in \Z+\frac{d+n'+1}{2}$ be critical points such that $L(m_0+1,{\rm As}(\Sigma_\infty) \times \Pi_\infty') \neq 0$.
Then we have
\begin{align*}
&\sigma \left( \frac{L(m_0,{\rm As}(\Sigma) \times \Pi')}{|D_{K/\Q}|^{\frac{n'}{2}}\cdot L(m_0+1,{\rm As}(\Sigma) \times \Pi')} \right)\\
&  = \frac{L(m_0,{\rm As}({}^\sigma\Sigma) \times {}^\sigma\Pi')}{|D_{K/\Q}|^{\frac{n'}{2}}\cdot L(m_0+1,{\rm As}({}^\sigma\Sigma) \times {}^\sigma\Pi')},\quad \sigma \in {\rm Aut}(\C).
\end{align*}
Here $D_{K/\Q}$ is the absolute discriminant of the discriminant algebra $K$ of $E/F$.
\end{theorem}

\begin{proof}
We apply Theorem \ref{thm:ratio-triple} to
\[
\Pi_1:={\rm As}(\Sigma)\otimes |\det|_{\A_F}^{\frac{d-1}{2}},
\quad
\chi_1:=(\mu\vert_{\A_F^\times})|\mbox{ }|_{\A_F}^{d-1},
\]
with $n_2=1$ and with $\Pi_2$ equal to the trivial representation.
Since $E$ is totally real, the automorphic representation $\Pi_1$ is isobaric, as explained
in the first paragraph of the proof of Theorem \ref{thm:ratio-triple}. By the
assumptions and Corollary \ref{coro:Asai-transfer}, $\Pi_1$ is
$\chi_1$-orthogonal.

By \cite[Lemma 7.1-(f)]{Prasad1992}, we have
\[
\omega_{{\rm As}(\Sigma)}
=
\omega_{K/F}\cdot
(\omega_\Sigma\vert_{\A_F^\times})^{m^{d-1}}.
\]
Here $\omega_{K/F}$ is the Hecke character of $F$ associated with $K/F$ by class field theory.
There is a typo in \emph{loc. cit.}; the formula above is the corrected one, as
recorded in \cite[p.\,755]{Chen2021b}. Since $d>1$, $m$ is even, and
$\mu^m=\omega_\Sigma^2$, it follows that
\[
\chi_1^{\frac{m^d}{2}}\omega_{\Pi_1}^{-1}
=(\mu\vert_{\A_F^\times})^{\frac{m^d}{2}}
\omega_{{\rm As}(\Sigma)}^{-1}
=
(\omega_\Sigma\vert_{\A_F^\times})^{m^{d-1}}
\omega_{{\rm As}(\Sigma)}^{-1}
=
\omega_{K/F}.
\]
Finally, it is well known that
\[
\Gsum(\omega_{K/F})\in |D_{K/\Q}|^{\frac{1}{2}}\cdot \Q^\times.
\]
The assertion therefore follows from Theorem \ref{thm:ratio-triple}.
\end{proof}

\begin{remark}
If $E$ is totally real and $m=d=2$, then all conditions are satisfied by the result of Krishnamurthy \cite{Kris2003}.
\end{remark}

\begin{remark}\label{rem:muthu-raghuram}
When $d = 2$ and $n'=1$, without assumption (5), i.e., without assuming the existence of the Asai transfer, the above theorem will follow from a forthcoming 
work of Krishnamurthy and Raghuram \cite{muthu-raghuram}. 
\end{remark}

\section{Variations for other periods}

Let $\Pi$ be a regular algebraic cuspidal automorphic representation of $G(\A)$ such that $\Pi_f \in \Coh_{\rm cusp}(G,\lambda)$ for some $\lambda \in X_{00}^+(T\times\C)$ with purity weight ${\sf w}$.
Let $\bullet \in \{b_n^F,t_n^F\}$ be an extreme degree, and $\varepsilon \in \widehat{\pi_0(G(\R))}$ be permissible for $\Pi_\infty$.
In this section, we consider two variants of the period relation under duality. The first concerns the Betti--Shalika periods under duality in the case where $\Pi$ admits a Shalika model. The second relates the Betti--Whittaker periods of $\Pi$ and $\Pi^c$, where $c$ is a $\Q$-algebra automorphism of $F$.

\subsection{Betti--Shalika periods}

In this subsection, assume that 
\[
n=2r \quad\mbox{and}\quad \mbox{$\Pi$ is $\chi$-symplectic}
\] 
for some algebraic Hecke character $\chi$ of $F$ (cf.\,\S\,\ref{sec:esse.-self-dual}).
By the well-known theorem of Jacquet--Shalika \cite[\S\,8, Theorem 1]{JS1990b}, this is equivalent to $\chi^r = \omega_\Pi$ and the $(\chi,\psi_F)$-Shalika integral
\[
{S}_\chi(g;\phi):= \int_{Z_G(\A)R(\Q)\backslash R(\A)}\phi(ug)(\chi\otimes\psi_F)^{-1}(u)\,du^{\rm Tam},\quad g \in G(\A)
\]
is nonzero for some cusp form $\phi \in V_\Pi$.
Here $R$ is the Shalika subgroup of $G$ defined by
\[ 
R := \left.\left\{u(h,x)= \begin{pmatrix} h & 0 \\ 0 & h \end{pmatrix} \begin{pmatrix} {\bf 1}_r & x \\ 0 & {\bf 1}_r\end{pmatrix}\,\right\vert\, h\in {\rm Res}_{F/\Q}\GL_r ,\, x\in {\rm Res}_{F/\Q}{\rm M}_{r\times r}\right\}, 
\]
$du^{\rm Tam}$ is the Tamagawa measure on $Z_G(\A)\backslash R(\A)$, 
and $\chi\otimes\psi_F$ denotes the character 
\[
(\chi\otimes\psi_F)(u(h,x)):= \chi(\det h)\psi_F({\rm tr}(x)).
\]
In this case, let 
\[
\mathcal{S}(\Pi,\chi) \defeq \{ S_\chi(\phi) \,\vert\, \phi \in V_\Pi\}
\] 
be the \emph{$(\chi,\psi_F)$-Shalika model} of $\Pi$; it is a model of $\Pi$ in the sense that we have an isomorphism of $(\frak{g}_\infty,C_\infty)\times G(\A_f)$-modules:
\begin{align}\label{eq:Shalika-realization}
V_\Pi \longrightarrow \mathcal{S}(\Pi,\chi), \quad \phi \longmapsto S_\chi(\phi).
\end{align}
For each place $v$ of $\Q$, let $\mathcal{S}(\Pi_v,\chi_v)$ be the $(\chi_v,\psi_{F,v})$-Shalika model of $\Pi_v$. It is realized in the space of smooth functions $S:G(\Q_v) \rightarrow \C$ such that
\[
S(ug) = (\chi_v\otimes\psi_{F,v})(u)S(g),\quad u \in R(\Q_v),\ g \in G(\Q_v)
\]
and $S$ is right $C_\infty$-finite and of moderate growth when $v=\infty$. 
Then we have an isomorphism
\[
{\bigotimes_v}' \mathcal{S}(\Pi_v,\chi_v) \longrightarrow \mathcal{S}(\Pi,\chi),\quad \bigotimes_v S_v \longmapsto \prod_v S_v.
\]

\subsubsection{An algebraic automorphism of $G$}\label{sec:alg-auto-2}

Let $\iota : G \rightarrow G$ be an algebraic automorphism defined by
\[
\iota(g):= \omega_{r,r}\cdot {}^tg^{-1} \cdot \omega_{r,r}^{-1},
\]
where $\omega_{r,r}$ denotes denotes the standard symplectic matrix:
\begin{align*}
\omega_{r,r}:=\bp 0 & {\bf 1}_r \\ -{\bf 1}_r & 0\ep.
\end{align*}
This is analogous to the automorphism $\theta$ defined in \S\,\ref{sec:alg-auto-1} with the long Weyl element $\omega_n$ replaced by $\omega_{r,r}$.
We also write $\iota(g) = {}^\iota g$. Note that ${}^\iota C_\infty = C_\infty$ and 
${}^\iota u(h,x) = u({}^t h^{-1},{}^tx)$.
In particular, ${}^\iota R = R$ and 
\begin{align*}
(\chi\otimes\psi_F)({}^\iota u) = (\chi^{-1}\otimes\psi_F)(u),\quad u \in R(\A).
\end{align*}

\subsubsection{${\rm Aut}(\C)$-action and $\iota$-twist on the Shalika models}

For $\sigma \in {\rm Aut}(\C)$, the $\sigma$-conjugate ${}^\sigma\Pi_f$ of $\Pi_F$ admits a $({}^\sigma\chi_f,\psi_{F,f})$-Shalika model. Moreover, we have a $\sigma$-linear $G(\A_f)$-equivariant isomorphism 
\[
\sigma_{\mathcal S} : \mathcal{S}(\Pi_f,\chi_f)\longrightarrow \mathcal{S}({}^\sigma\Pi_f,{}^\sigma\chi_f)
\]
defined by
\[
(\sigma_\mathcal{S} S)(g):= \sigma(S(\overline{u}_\sigma^{-1} g)),\quad g \in G(\A_f).
\]
Here
\[
\overline{u}_\sigma:=\diag(u_{\sigma}\cdot {\bf 1}_r, {\bf 1}_r) \in T(\A_f),
\]
and $u_\sigma \in \prod_p \Z_p^\times \subset \A_{F,f}^\times$ is the unique element such that
$\sigma(\psi(x)) = \psi(u_\sigma x)$ for all $x \in \A_f$.
By taking the Galois invariants over $\Q(\Pi,\chi)$, we obtain a subspace of $\mathcal{S}(\Pi_f,\chi_f)$ over $\Q(\Pi,\chi)$:
\begin{align}\label{eq:Shalika-mod-rational-structure}
\mathcal{S}(\Pi_f,\chi_f)^{{\rm Aut}(\C/\Q(\Pi,\chi))} := \left.\left\{S\in \mathcal{S}(\Pi_f,\chi_f)\,\right\vert\,\sigma_\cS S=S\mbox{ for $\sigma \in {\rm Aut}(\C/\Q(\Pi,\chi))$}\right\}.
\end{align}
This defines a $\Q(\Pi,\chi)$-rational structure of $\cS(\Pi_f,\chi_f)$.

As explained in \S\,\ref{sec:local-dual}, the $\iota$-twist representation ${}^\iota\Pi_v$ is isomorphic to $\Pi_v^\vee$ for all places $v$ of $\Q$. Moreover, we have a $\C$-linear isomorphism
\[
\iota_\cS : \cS(\Pi_v,\chi_v) \longrightarrow \cS(\Pi_v^\vee,\chi_v^{-1}),\quad S\longmapsto {}^\iota S.
\]

\begin{lemma}\label{lem:Sh-rational}
Let $\sigma \in {\rm Aut}(\C)$. As $\sigma$-linear isomorphisms from $\cS(\Pi_f,\chi_f)$ to $\cS({}^\sigma\Pi_f^\vee,{}^\sigma\chi_f^{-1})$, we have
\[
\sigma_\cS \circ \frac{\iota_\cS}{\mathcal{G}(\omega_\Pi)}
=
\frac{\iota_\cS}{\mathcal{G}({}^\sigma\omega_\Pi)} \circ \sigma_\cS.
\]
\end{lemma}

\begin{proof}
This is clear, since
\[
{}^\iota \overline{u}_\sigma = u_\sigma\cdot \overline{u}_\sigma.
\]
\end{proof}

\subsubsection{Betti--Shalika periods}

Fix a generator
\[
[\Pi_\infty]_\cS^\varepsilon \in H^\bullet(\frak{g}_\infty,K_\infty^\circ;\mathcal{S}(\Pi_\infty,\chi_\infty)\otimes\M_\lambda)(\varepsilon).
\]
Similarly as in \S\,\ref{sec:comparison-isomorphism}, we have a $G(\A_f)$-equivariant isomorphism 
\[
\mathcal F_{\Pi_f,[\Pi_\infty]_\cS^\varepsilon}^\cS : \cS(\Pi_f,\chi_f) \longrightarrow H_{\rm cusp}^\bullet(S^G,\widetilde\M_\lambda)(\varepsilon\times\Pi_f)
\]
defined by
\[
\mathcal{F}^\cS_{\Pi_f,[\Pi_\infty]_\cS^\varepsilon}(W):= (\Upsilon_\Pi^\cS)^\bullet([\Pi_\infty]_\cS^\varepsilon\otimes W),
\]
where $\Upsilon_\Pi^\cS$ is the inverse of the isomorphism (\ref{eq:Shalika-realization}). 
The \emph{Betti--Shalika period} of $\Pi$ with respect to $[\Pi_\infty]_\cS^\varepsilon$ is the scalar, well-defined up to multiplication by an element of $\Q(\Pi,\chi)^\times$, denoted by
\[
p_\cS^\varepsilon(\Pi) \in \C^\times / \Q(\Pi,\chi)^\times,
\]
and obtained by comparing the two $\Q(\Pi,\chi)$-rational structures on $\Pi_f$ via the isomorphism $\mathcal F_{\Pi_f,[\Pi_\infty]_\cS^\varepsilon}^\cS$.
Assume further that either $F$ is not totally imaginary or $\chi = |\mbox{ }|_{\A_F}^{-{\sf w}}$. 
Then, for each $\sigma \in {{\rm Aut}(\C)}$, 
${}^\sigma\Pi$ is also ${}^\sigma\chi$-symplectic by either Corollary \ref{coro:arithmeticity} or \cite[Proposition 8.2]{CKT2026} (see Remark \ref{rem:Clozel-Kret-Taibi}).
Similarly, we have the Betti--Shalika period $p_\cS^\varepsilon({}^\sigma\Pi)$ with respect to a chosen generator $[{}^\sigma\Pi_\infty]_\cS^\varepsilon$. 
{For compatibility, we assume $[{}^\sigma\Pi_\infty]_\cS^\varepsilon = [\Pi_\infty]_\cS^\varepsilon$ if $\sigma \in {\rm Aut}(\C/\Q(\Pi,\chi))$.}
We normalize the collection of periods
\[
\bigl(p_\cS^\varepsilon({}^\sigma\Pi)\bigr)_{\sigma \in {{\rm Aut}(\C)}}
\]
in the same manner as in (\ref{eq:W-H-periods}). 
When $F$ is totally real these periods were defined by Grobner--Raghuram \cite{GR2014a}. 

\subsubsection{Period relation}

As in \S\,\ref{sec:archi-local-dual}, let
\[
{}^\iota[\Pi_\infty]_\cS^\varepsilon
\in
H^\bullet\bigl(\frak{g}_\infty,K_\infty^\circ;
\mathcal{S}(\Pi_\infty^\vee,\chi_\infty^{-1})
\otimes{}^\iota\M_\lambda\bigr)(\varepsilon)
\]
be the dual generator associated with $[\Pi_\infty]_\cS^\varepsilon$, defined as in
(\ref{eq:theta-twist}) with $\theta$ replaced by $\iota$. Similarly, for each
$\sigma \in {\rm Aut}(\C)$, let ${}^\iota[{}^\sigma\Pi_\infty]_\cS^\varepsilon$ denote the dual generator associated with $[{}^\sigma\Pi_\infty]_\cS^\varepsilon$.
This gives rise to a collection of Betti--Shalika periods
\[
\bigl(p_\cS^\varepsilon({}^\sigma\Pi^\vee)\bigr)_
{\sigma \in {\rm Aut}(\C)}
\]
which we normalize as in (\ref{eq:W-H-periods}).

We have the following result on the period relation under duality.

\begin{theorem}\label{thm:period-relation-2}
Assume $n$ is even, $\Pi$ is $\chi$-symplectic, and either $F$ is not totally imaginary or $\chi = |\mbox{ }|_{\A_F}^{-{\sf w}}$.
We have
	\begin{equation*}
		\sigma\left( \frac{p_\cS^\eps(\Pi\dual)}{\Gsum(\omega_{\Pi})^{-1}\cdot p_\cS^\eps(\Pi)} \right) = \frac{p_\cS^\eps({}^\sigma\Pi\dual)}{\Gsum({}^\sigma\omega_{\Pi})^{-1}\cdot p_\cS^\eps({}^\sigma\Pi)},\quad \sigma \in {\rm Aut}(\C).
	\end{equation*}
\end{theorem}

\begin{proof}
Consider the diagram (non-commutative)
\begingroup
\begin{equation*}
	\begin{tikzpicture}[baseline= (a).base]
		\node[scale=0.813] (a) at (0,0){
		\begin{tikzcd}[column sep={4.8cm,between origins}]
		& [0.5cm] \cS(\Pi_f^\vee,\chi_f^{-1}) \arrow[rr,"\mathcal F^\cS_{\Pi_f^\vee,{}^\iota[\Pi_\infty]_\cS^\varepsilon}"] \arrow[d,"\sigma_\cS"] & & [0.5cm] H_{\rm cusp}^\bullet(S^G, {}^\iota\widetilde \M_{\lambda})(\eps\times\Pi_f^\vee) \arrow[d,"\sigma_{\rm dR}^\bullet"] \\[20pt]
		& [0.5cm] \cS({}^\sigma\Pi_f^\vee,{}^\sigma\chi_f^{-1}) \arrow[ddl,swap,"\iota_\cS" {xshift=-5pt,yshift=-5pt},<-] & & [0.5cm] H_{\rm cusp}^\bullet(S^G, {}^{\sigma\circ\iota}\widetilde\M_{\lambda})(\eps\times{}^\sigma\Pi_f^\vee) \arrow[ll,<-,swap,"\mathcal F^\cS_{{}^\sigma\Pi_f^\vee,{}^\iota[{}^\sigma\Pi_\infty]_\cS^\varepsilon}"] \\
		\cS(\Pi_f,\chi_f)
  \arrow[rr,crossing over,swap,
    "{\mathcal F^\cS_{\Pi_f,[\Pi_\infty]_\cS^\varepsilon}}" {yshift=-0.5ex}] \arrow[d,"\sigma_\cS"] \arrow[uur,"\iota_\cS" {xshift=-5pt,yshift=-5pt}] & & H_{\rm cusp}^\bullet(S^G, \widetilde \M_{\lambda})(\eps\times \Pi_f) \arrow[d,"\sigma_{\rm dR}^\bullet"] \arrow[uur,crossing over,"\iota_{\rm dR}^\bullet" {xshift=15pt,yshift=10pt}] & \\[20pt]
		\cS({}^\sigma\Pi_f,{}^\sigma\chi_f) \arrow[rr,"\mathcal F^\cS_{{}^\sigma \Pi_f,[{}^\sigma\Pi_\infty]_\cS^\varepsilon}"]  & & H_{\rm cusp}^\bullet(S^G, {}^\sigma\widetilde \M_{\lambda})(\eps\times{}^\sigma \Pi_f) \arrow[uur,"\iota_{\rm dR}^\bullet" {xshift=15pt,yshift=10pt}] &
		\end{tikzcd}
	};
	\end{tikzpicture}
\end{equation*}
\endgroup
Here $\iota_{\rm dR}^\bullet$ is defined as in \S\,\ref{sec:dual-arithm} with $\theta$ replaced by $\iota$.
As in the proof of Theorem \ref{thm:period-relation}, it follows from the definition of the Betti--Shalika period and Lemma \ref{lem:Sh-rational} that, in order to prove the desired period relation, it suffices to show that the top, bottom, and right-hand faces of the diagram commute. The verification is identical to that in the proof of Theorem \ref{thm:period-relation}, with $\theta$, $\Upsilon_\Pi$, and $\Upsilon_{\Pi^\vee}$ replaced respectively by $\iota$, $\Upsilon_\Pi^\cS$, and $\Upsilon_{\Pi^\vee}^\cS$.
\end{proof}

\subsection{Period relations under algebraic conjugation}

In this subsection, fix a $\Q$-algebra automorphism
\[
c:F\longrightarrow F.
\]
By abuse of notation, we also denote by
\[
c:G\longrightarrow G,\quad g\longmapsto g^c:=c(g),
\]
the induced algebraic automorphism of $G$.

\subsubsection{$c$-conjugate representations}
For an algebraic representation
$(\rho,\M)$ of $G\times\C$, define its $c$-conjugate to be the algebraic representation
$(\rho^c,\M^c)$ given by
\[
\M^c:=\M,\quad \rho^c(g):=\rho(g^c).
\]
Here
\[
(g_\tau)_{\tau:F\rightarrow\C}^c
=
(g_{\tau\circ c})_{\tau:F\rightarrow\C}.
\]
It is immediate from the definitions that
\[
({}^\sigma(\rho^c),{}^\sigma(\M^c))
=
(({}^\sigma\rho)^c,({}^\sigma\M)^c),
\quad \sigma\in {\rm Aut}(\C).
\]
If $(\rho,\M) = (\rho_\lambda,\M_\lambda)$ for some $\lambda \in X^+(T\times\C)$, then we have a canonical isomorphism
\[
(\rho_\lambda^c,\M_\lambda^c) \longrightarrow (\rho_{\lambda^c},\M_{\lambda^c}),\quad \bigotimes_\tau m_\tau \longmapsto \bigotimes_\tau m_{\tau\circ c^{-1}}.
\]
Here $\lambda^c \in X^+(T\times\C)$ is the $c$-conjugate of $\lambda$ defined by
\[
(\lambda^c)^\tau := \lambda^{\tau\circ c^{-1}},\quad \tau \in {\rm Hom}(F,\C).
\]

For each place $v$ of $\Q$, define the $c$-conjugate $\Pi_v^c$ of $\Pi_v$ as follows. If $v=p$ is finite, let
$\Pi_p^c$ be the representation of $G(\Q_p)$ given by
\[
\pi_p^c(g):=\pi_p(g^c),\quad g\in G(\Q_p).
\]
If $v=\infty$, let $\Pi_\infty^c$ be the $(\frak g_\infty,C_\infty)$-module defined by
\[
\pi_\infty^c(Y):=\pi_\infty(c^*Y),\quad
\pi_\infty^c(k):=\pi_\infty(k^c),
\quad
Y\in \frak g_\infty,\; k\in C_\infty.
\]
It is clear that 
\[
{}^\sigma(\Pi_v^c) = ({}^\sigma\Pi_v)^c,\quad\sigma \in {\rm Aut}(\C).
\]
Define a $\C$-linear isomorphism
\[
c_\W : \W(\Pi_v) \longrightarrow \W(\Pi_v^c),\quad W \longmapsto W^c.
\]

Note that
\[
\Pi^c:={\bigotimes_v}' \Pi_v^c
\]
is an irreducible admissible $(\frak g_\infty,C_\infty)\times G(\A_f)$-module. It is again cuspidal automorphic, and its automorphic realization is given by
\[
V_{\Pi^c}
=
\{\phi^c\mid \phi\in V_\Pi\}.
\]

\subsubsection{Period relation}

Let $\varepsilon^c \in \widehat{\pi_0(G(\R))}$ be the $c$-conjugate of $\varepsilon$ defined by
\[
\varepsilon^c(g):=\varepsilon(g^c), \quad g \in G(\R).
\]
In other words, $(\varepsilon^c)_v = \varepsilon_{v\circ c^{-1}}$ for $v \in S_r$.
Consider the association 
\begin{equation*}
\begin{tikzcd}[row sep=normal, column sep=normal]
\extp^{\bullet}\frak{g}_\infty/\frak{k}_\infty \arrow[r, "F"] & \mathcal{W}({}^\sigma\Pi_\infty)\otimes {}^\sigma\mathcal{M}_{\lambda}\arrow[d, "c_\W \otimes {\rm id}"]\\
\extp^{\bullet}\frak{g}_\infty/\frak{k}_\infty \arrow[r, "F^c"]\arrow[u, "\extp^\bullet dc"]  & \mathcal{W}({}^\sigma\Pi_\infty^c)\otimes {}^\sigma\mathcal{M}_{\lambda}^c.
\end{tikzcd}
\end{equation*}
It is easy to verify that, for all $k \in K_\infty$, we have
\[
k\cdot F^c = (k^c\cdot F)^c.
\]
In particular, $\varepsilon^c$ is permissible for $\Pi_\infty^c$.

For each $\sigma \in {{\rm Aut}(\C)}$, we fix a generator 
\[
[{}^\sigma\Pi_\infty]^\varepsilon \in H^\bullet(\frak{g}_\infty,K_\infty^\circ;\W({}^\sigma\Pi_\infty)\otimes{}^\sigma\M_\lambda)(\varepsilon).
\]
Then we have an associated $c$-conjugate generator 
\[
([{}^\sigma\Pi_\infty]^\varepsilon)^c \in H^\bullet(\frak{g}_\infty,K_\infty^\circ;\W({}^\sigma\Pi_\infty^c)\otimes{}^\sigma\M_\lambda^c)(\varepsilon^c).
\]
{For compatibility, we assume $[{}^\sigma\Pi_\infty]^\varepsilon = [\Pi_\infty]^\varepsilon$ if $\sigma \in {\rm Aut}(\C/\Q(\Pi))$.}
With respect to these generators, we obtain two collections of Betti--Whittaker periods
\[
\bigl(p^\varepsilon({}^\sigma\Pi)\bigr)_
{\sigma \in {{\rm Aut}(\C)}},\quad \bigl(p^{\varepsilon^c}({}^\sigma\Pi^c)\bigr)_
{\sigma \in {{\rm Aut}(\C)}}
\]
normalized as in (\ref{eq:W-H-periods}).

We have the following period relation under $c$-conjugation.

\begin{theorem}\label{thm:period-relation-3}
We have
\[
\sigma \left(\frac{p^{\varepsilon^c}(\Pi^c)}{p^\varepsilon(\Pi)}\right) = \frac{p^{\varepsilon^c}({}^\sigma\Pi^c)}{p^\varepsilon({}^\sigma\Pi)},\quad \sigma \in {\rm Aut}(\C).
\]
\end{theorem}

\begin{proof}
Consider the diagram
\begingroup
\begin{equation*}
	\begin{tikzpicture}[baseline= (a).base]
		\node[scale=0.813] (a) at (0,0){
		\begin{tikzcd}[column sep={4.8cm,between origins}]
		& [0.5cm] \W(\Pi_f^c) \arrow[rr,"\mathcal F_{\Pi_f^c,([\Pi_\infty]^\varepsilon)^c}"] \arrow[d,"\twistW"] & & [0.5cm] H_{\rm cusp}^\bullet(S^G, \widetilde \M_{\lambda}^c)(\eps^c\times\Pi_f^c) \arrow[d,"\sigma_{\rm dR}^\bullet"] \\[20pt]
		& [0.5cm] \W({}^\sigma\Pi_f^c) \arrow[ddl,swap,"c_\W" {xshift=-5pt,yshift=-5pt},<-] & & [0.5cm] H_{\rm cusp}^\bullet(S^G, {}^{\sigma}\widetilde\M_{\lambda}^c)(\eps^c\times{}^\sigma\Pi_f^c) \arrow[ll,<-,swap,"\mathcal F_{{}^\sigma\Pi_f^c,([{}^\sigma\Pi_\infty]^\varepsilon)^c}"] \\
		\W(\Pi_f)
  \arrow[rr,crossing over,swap,
    "{\mathcal F_{\Pi_f,[\Pi_\infty]^\varepsilon}}" {yshift=-0.5ex}] \arrow[d,"\twistW"] \arrow[uur,"c_\W" {xshift=-5pt,yshift=-5pt}] & & H_{\rm cusp}^\bullet(S^G, \widetilde \M_{\lambda})(\eps\times \Pi_f) \arrow[d,"\sigma_{\rm dR}^\bullet"] \arrow[uur,crossing over,"c_{\rm dR}^\bullet" {xshift=15pt,yshift=10pt}] & \\[20pt]
		\W({}^\sigma\Pi_f) \arrow[rr,"\mathcal F_{{}^\sigma \Pi_f,[{}^\sigma\Pi_\infty]^\varepsilon}"]  & & H_{\rm cusp}^\bullet(S^G, {}^\sigma\widetilde \M_{\lambda})(\eps\times{}^\sigma \Pi_f) \arrow[uur,"c_{\rm dR}^\bullet" {xshift=15pt,yshift=10pt}] &
		\end{tikzcd}
	};
	\end{tikzpicture}
\end{equation*}
\endgroup
Here $c_{\rm dR}^\bullet$ is defined as in \S\,\ref{sec:dual-arithm} with $\theta$ replaced by $c$.
Note that the left-hand face is commutative.
Indeed, $u_\sigma \in \widehat{\Z}^\times$ is invariant under $c$.
As in the proof of Theorem \ref{thm:period-relation}, together with the definition of the Betti--Whittaker period, it suffices to show that the top, bottom, and right-hand faces of the diagram commute. The verification is identical to that in the proof of Theorem \ref{thm:period-relation}, with $\theta$, $\Pi^\vee$ replaced by $c$, $\Pi^c$.
\end{proof}


\end{document}